\documentclass[reqno]{amsart}
\usepackage{amssymb}
\usepackage[mathscr]{euscript}
\usepackage{amsmath}
\makeatletter
\@addtoreset{equation}{section}
\makeatother

\renewcommand\thefigure{\thesection.\@arabic\c@figure}
\renewcommand\thetable{\thesection.\@arabic\c@table}

\newtheorem{theorem}{Theorem}[section]
\newtheorem{lemma}[theorem]{Lemma}
\newtheorem{proposition}[theorem]{Proposition}
\newtheorem{corollary}[theorem]{Corollary}

\newtheorem{definition}[theorem]{Definition}
\newtheorem{remark}[theorem]{Remark}
\newtheorem{example}[theorem]{Example}

\newcommand{\mc}[1]{{\mathcal #1}}
\newcommand{\ms}[1]{{\mathscr #1}}
\newcommand{\mf}[1]{{\mathfrak #1}}
\newcommand{\mb}[1]{{\mathbf #1}}
\newcommand{\bb}[1]{{\mathbb #1}}
\newcommand{\bs}[1]{{\boldsymbol #1}}

\newcommand{\prob}{{\boldsymbol {\rm P}}}
\newcommand{\E}{\boldsymbol{{\rm E}}}
\renewcommand{\sp}{\eta_{\bs\cdot}}

\newcommand{\<}{\langle}
\renewcommand{\>}{\rangle}

\renewcommand{\Cap}{{\rm cap}}

\begin{document}

\title[Tunneling and Metastability of continuous time Markov
  chains]{Tunneling and Metastability of continuous time Markov
  chains}

\author{J. Beltr\'an, C. Landim}

\address{\noindent IMCA, Calle los Bi\'ologos 245, Urb. San C\'esar
  Primera Etapa, Lima 12, Per\'u and PUCP, Av. Universitaria cdra. 18,
  San Miguel, Ap. 1761, Lima 100, Per\'u. 
\newline e-mail: \rm
  \texttt{johel@impa.br} }

\address{\noindent IMPA, Estrada Dona Castorina 110, CEP 22460 Rio de
  Janeiro, Brasil and CNRS UMR 6085, Universit\'e de Rouen, Avenue de
  l'Universit\'e, BP.12, Technop\^ole du Madril\-let, F76801
  Saint-\'Etienne-du-Rouvray, France.  \newline e-mail: \rm
  \texttt{landim@impa.br} }


\begin{abstract}
  We propose a new definition of metastability of Markov processes on
  countable state spaces. We obtain sufficient conditions for a
  sequence of processes to be metastable. In the reversible case these
  conditions are expressed in terms of the capacity and of the
  stationary measure of the metastable states.
\end{abstract}

\maketitle

\section{Introduction}
\label{sec0}

In the framework of non-equilibrium statistical mechanics,
metastability is a relevant dynamical phenomenon taking place in the
vicinities of first order phase transitions. There has been along the
years several proposals of a rigorous mathematical description of the
phenomenon starting with Lebowitz and Penrose \cite{lp} who derived
the canonical free energy for Kac potentials in the Van der Waals
limit. The seminal paper of Cassandro, Galves, Olivieri and Vares
\cite{cgov} proposed a pathwise approach to metastability which
highlighted the underlying Markov structure behind metastability which
is exploited here. In the sequel, Scoppola \cite{s} examined the
metastable behavior of finite state space Markov chains with
transition probabilities exponentially small in a parameter.  More
recently, Bovier and co-authors (\cite{b2} and references therein)
presented a new approach based on the spectral properties of the
generator of the process. We refer to \cite{ov} for a recent monograph
on the subject.

We propose in this article an alternative formulation of metastability
for sequences of Markov processes on countable state spaces.
Informally, a process is said to exhibit a metastable behavior if it
remains for a very long time in a state before undergoing a rapid
transition to a stable state. After the transition, the process
remains in the stable state for a period of time much longer than the
time spent in the first state, called for this reason metastable. In
certain cases, there are two or more ``metastable wells'' with the
same depth, a situation called by physicists ``competing metastable
states''. In these cases, the process thermalizes in each well before
jumping abruptly to another well where the same qualitative behavior
is observed.

To describe our approach, denote by $E_N$, $N\ge 1$, a sequence of
countable spaces and by $(\theta_N : N\ge 1)$ a sequence of positive
real numbers. For each $N\ge 1$, consider a partition $\ms E^1_N,
\dots, \ms E^\kappa_N$, $\Delta_N$ of $E_N$ and a $E_N$-valued Markov
process $\{\eta^N_t : t\ge 0\}$.  Fix a state $\xi^N_x$ in $\ms
E^x_N$, $1\le x\le \kappa$. We say that the sequence of Markov
processes $\{\eta^N_t : t\ge 0\}$, $N\ge 1$, exhibits a tunneling
behavior in the time scale $(\theta_N : N\ge 1)$ with metastates $\ms
E^1_N, \dots, \ms E^\kappa_N$, attractors $\xi^N_1, \dots,
\xi^N_\kappa$, and asymptotic behavior described by the Markov process
on $S = \{1, \dots, \kappa\}$ with rates $\{r(x,y) : x,y \in S\}$ if
the following three conditions are fulfilled:

\begin{enumerate}
\item For every $1\le x\le \kappa$, starting from a state $\eta^N$ in
  $\ms E_N^x$, with overwhelming probability, the process $\{\eta^N_t
  : t\ge 0\}$ reaches $\xi^N_x$ before attaining $\bigcup_{y\not =
    x}\ms E^y_N$.

\item Let $\{X^N_t : t\ge 0\}$ be the process $X^N_t =
  \Psi_N(\eta^{\ms E_N}_t)$, where $\{\eta^{\ms E_N}_t : t\ge 0\}$ is
  the trace of the Markov process $\{\eta_t : t\ge 0\}$ on $\ms E_N =
  \bigcup_{1\le x\le \kappa} \ms E^x_N$ and where $\Psi_N(\eta) =
  \sum_{1\le x\le \kappa} x \mb 1\{\eta\in \ms E^x_N\}$. The speed up
  process $\{X^N_{t\theta_N}: t\ge 0\}$ converges to the Markov process
  on $S$ which jumps from $x$ to $y$ at rate $r(x,y)$.

\item Starting from any point of $\ms E_N$, the time spent by the
  speed up Markov process $\{\eta_{t\theta_N} : t\ge 0\}$ on the set
  $\Delta_N$ in any time interval $[0,s]$, $s>0$, vanishes in
  probability.
\end{enumerate}

All the terminology used in the previous definition is explained in
the next section. Condition (1) states that the process thermalizes in
each set $\ms E^x_N$ before reaching another metastate set $\ms
E^y_N$, $y\not = x$. The assumption of the existence of an attractor
can clearly be relaxed, but is satisfied in several interesting
examples, as in the condensed zero-range processes \cite{bl3, bl4}
which motivated the present work. Condition (2) describes the
intervalley dynamics and reveals the loss of memory of the jump times
from a well to another, put in evidence in \cite{cgov}.  In condition
(3) we assume that the starting point belongs to $\ms E_N$. It may
therefore happen that the discarded set $\Delta_N$ hides wells deeper
than the wells $\ms E^x_N$, $1\le x\le \kappa$, but which can not be
attained from $\ms E_N$. When we remove in this condition the
assumption that the starting point belongs to $\ms E_N$, we say that
the process exhibits a metastable behavior, instead of a tunneling
behavior. In this case, the wells $\ms E^x_N$, $1\le x\le \kappa$ are
the deepest ones.

In contrast with the pathwise approach to metastability \cite{cgov},
the present one does not give a precise description of the saddle
points between the wells nor of the typical path which drives the
system from one well to another. Its description of metastability is
in some sense rougher, but keeps the main ingredients, as
thermalization and asymptotic Markovianity.

The main results of this article, stated in the next section,
establish sufficient conditions for recurrent Markov processes on
countable state spaces to exhibit a tunneling behavior. In the
reversible case, these sufficient conditions can be expressed in terms
of the capacity and of the stationary probability measure of the
metastates.

A theory is meaningless if no interesting example is provided which
fits in the framework presented. Besides the mean field models
considered in \cite{cgov} and the Freidlin--Wentsell Markov chains
proposed in \cite{ov}, which naturally enter in the present framework,
we examine in \cite{bl3, bl4} a new class of processes which exhibit a
metastable behavior. This family, known as the condensed zero-range
processes, have been introduced in the physics literature
\cite{evans, gss, emz} to model the Bose-Einstein condensation
phenomena. It has been proved in several different contexts \cite{jmp,
  fls, al} that, above a critical density, all but a small number of
particles concentrate on one single site in the canonical stationary
states of these processes. In \cite{bl3, bl4} we prove that, in the
reversible case, the condensed zero range processes exhibit a
tunneling behavior by showing that in an appropriate time scale the
condensed site evolves according to a random walk on $S$. We also
prove that the jump rates of the asymptotic Markov dynamics can be
expressed in terms of the capacities of the underlying random walks
performed by the particles.

The article is organized as follows. In Section \ref{sec1}, we
introduce the notation, the definitions and state the main theorems.
In Section \ref{examples} we present some elementary examples which
justify the definitions proposed. In Sections \ref{proof},
\ref{proof2}, we prove the main results. Finally, in Section
\ref{sec03}, we prove some results on the trace of Markov processes
needed in the article and which we did not find in the literature.

\section{Notation and Results}
\label{sec1}

Fix a sequence $(E_N: N\ge 1)$ of countable state spaces. The elements
of $E_N$ are denoted by the Greek letters $\eta$, $\xi$. For each
$N\ge 1$ consider a matrix $R_N : E_N \times E_N \to \bb R$ such that
$R_N(\eta, \xi) \ge 0$ for $\eta \not = \xi$, $-\infty < R_N (\eta,
\eta)\le 0$ and $\sum_{\xi\in E_N} R_N(\eta,\xi)=0$ for all $\eta\in
E_N$. Denote by $L_N$ the generator which acts on bounded functions
$f:E_N\to \bb R$ as
\begin{equation}
\label{c01}
(L_Nf) (\eta) \,=\, \sum_{\xi\in E_N} R_N(\eta,\xi)
\, \big\{f(\xi)-f(\eta)\big\}\;.
\end{equation}

Let $\{\eta^N_t : t\ge 0\}$ be the {\sl minimal} right-continuous
Markov process associated to the generator $L_N$. We refer to \cite{c,
  f, n} for the terminology and the main facts on Markov processes
alluded to in this article. It is well known, for instance, that
$\{\eta^N_t : t\ge 0\}$ is a strong Markov process with respect to the
filtration $\{\mc F^N_t : t\ge 0\}$ given by $\mc F^N_t = \sigma
(\eta^N_s : s\le t)$. To avoid unnecessary technical considerations,
we assume throughout this article that there is no explosion.

Denote by $D(\bb R_+,E_N)$ the space of right-continuous trajectories
$e: \bb R_+ \to E_N$ with left limits endowed with the Skorohod
topology. Let $\prob^N_{\eta}$, $\eta\in E_N$, be the probability
measure on $D(\bb R_+,E_N)$ induced by the Markov process $\{\eta^N_t
: t\ge 0\}$ starting from $\eta$. Expectation with respect to
$\prob^N_{\eta}$ is denoted by $\E^N_{\eta}$ and we frequently omit
the index $N$ in $\prob^N_{\eta}$, $\mb E^N_\eta$.

For every $N\ge 1$ and any subset $A\subseteq E_N$, denote by
$\tau_A:D(\bb R_+,E_N)\to\bb R_+$ the hitting time of the set $A$:
$$
\tau_A \,:=\, \inf \big\{ s > 0 : e_s \in A \big\}\,,
$$
with the convention that $\tau_A = \infty$ if $e_s\not\in A$ for all
$s>0$. When the set $A$ is a singleton $\{\eta\}$, we denote
$\tau_{\{\eta\}}$ by $\tau_\eta$.  This convention is adopted
everywhere below for any variable depending on a set. In addition, for
each $t\ge 0$, define the additive functional $\mc T^{A}_t:D(\bb
R_+,E)\mapsto\bb R_+$ as the amount of time the process stayed in the
set $A$ in the interval $[0,t]$:
\begin{equation}
\label{timecut}
\mc T^{A}_t \;:=\;\int_{0}^t \mathbf{1}\{e_s\in
A\}\, ds\,,\quad t\ge 0\,,
\end{equation}
where $\mb 1\{B\}$ stands for the indicator of the set $B$.

A sequence of states $\bs \eta=(\eta^N\in E_N : N\ge 1)$ is said to be
a point in a sequence ${\ms A}$ of subsets of $E_N$, ${\ms A} =
(A_N\subseteq E_N : N\ge 1)$, if $\eta^N$ belongs to $A_N$ for every
$N\ge 1$. For a point $\bs \eta=(\eta^N\in E_N : N\ge 1)$ and a set
${\ms A} = (A_N\subseteq E_N : N\ge 1)$, denote by $T_{\bs \eta}$,
$T_{\ms A}$, the hitting times of the sets $\{\bs \eta\}$,
$\ms A$:
\begin{equation*}
T_{{\bs \eta}} \;=\; T^N_{{\bs \eta}} \;:=\; \tau_{\eta^N}\;,
\quad  T_{\ms A} \;=\; T^N_{\ms A} \;:=\; \tau_{A_N}\; .
\end{equation*}

For any sequence of subsets $\ms A = (A_N \subset E_N : N\ge 1)$, $\ms
F = (F_N \subset E_N : N\ge 1)$, denote by $T_{\ms A} (\ms F)$ the
time spent on the set $\ms F$ before hitting the set $\ms A$:
\begin{equation*}
T_{\ms A} (\ms F) \;=\; T^N_{\ms A} (\ms F) 
\;:=\; \int_0^{\tau_{A_N}} \mb 1\{ \eta^N_s \in F_N\} \, ds\; .
\end{equation*}

\subsection{Valley with attractor}
We introduce in this subsection the concept of valley. Intuitively, a
subset $W$ of the state space $E_N$ is a valley for the Markov process
$\{\eta^N_t : t\ge 0\}$ if the process starting from $W$ thermalizes
in $W$ before leaving $W$ at an exponential random time.

To define precisely a valley, consider two sequences $\ms W$, $\ms B$
of subsets of $E_N$, the second one containing the first and being
properly contained in $E_N$:
\begin{equation}
\label{val1}
\ms W \,=\, (W_N\subseteq E_N : N\ge 1)\,, \;\;
\ms B \,=\, (B_N \subseteq E_N : N\ge 1)\,, \;\;
W_N \subseteq B_N \subsetneqq E_N\,.
\end{equation}
Fix a point $\bs \xi = (\xi_N \in W_N : N\ge 1)$ in $\ms W$, a
sequence of positive numbers $\bs \theta = (\theta_N : N\ge 1)$ and
denote by $\ms B^c$ the complement of $\ms B$: $\ms B^c =
({B}^c_N : N\ge 1)$. 

\begin{definition}[Valley]
\label{well}
The triple $(\ms W, \ms B, \bs \xi)$ is a valley of depth $\bs \theta$
and attractor $\bs \xi$ for the Markov process $\{\eta^N_t : t\ge 0\}$
if for every point $\bs \eta = (\eta^N : {N\ge 1})$ in $\ms W$

\begin{enumerate}
\item[({\bf V1})] With overwhelming probability, the attractor $\bs
  \xi$ is attained before the process leaves $\ms B$: 
\begin{equation*}
\label{well1}
\lim_{N\to\infty} \mb P_{\eta^N} \big[\, T_{\bs \xi} < 
T_{{\ms B}^c} \,\big] \, = \,1\,;
\end{equation*}

\item[({\bf V2})] The law of $T_{{\ms B}^c}/\theta_N\,$ under
  $\prob_{\eta^N}$ converges to a mean $1$ exponential distribution,
  as $N\to\infty$;

\item[({\bf V3})] 
For every $\delta>0$,
\begin{equation*}
\lim_{N\to\infty} \prob_{\eta^N}\Big[\, \frac 1{\theta_N} 
\, T_{{\ms B}^c} (\bs \Delta) > \delta\,\Big] \;=\; 0 \;,
\end{equation*}
where $\bs \Delta = (\Delta_N : N\ge 1)$ and $\Delta_N$ is the annulus
$B_N \setminus W_N$. 
\end{enumerate}
\end{definition}

We refer to $\ms W$ as the well, and $\ms B$ as the basin of the
valley $(\ms W, \ms B, \bs \xi)$. We present in Section \ref{examples}
examples of Markov processes on finite state spaces and triples $(\ms
W, \ms B, \bs \xi)$ in which all conditions but one in the above
definition hold.

\subsubsection*{Condition {\rm ({\bf V1})}.}
The first condition guarantees that the process thermalizes in $\ms W$
before leaving the basin $\ms B$.  We prove in Lemma \ref{s05} that
conditions ({\bf V1}), ({\bf V2}) imply that the attractor $\bs \xi$
is reached from any point in the well $\ms W$ faster than $\theta_N$:
\begin{equation*}
\tag*{({\bf V1}')}
\lim_{N\to\infty}  \sup_{\eta\in W_N} \prob_{\eta}
\Big[\, \frac 1{\theta_N} 
\, T_{\bs \xi} > \delta\,\Big] \;=\; 0 \;.
\end{equation*}
Conversely, this condition and ({\bf V2}) warrant the validity of
({\bf V1}). We may therefore replace ({\bf V1}) by ({\bf V1}') in the
definition.

Example \ref{ex8} illustrates the fact that conditions ({\bf V2}),
({\bf V3}) may hold while ({\bf V1}) fails. In this example, with
overwhelming probability, the process, starting from one state in the
well $\ms W$, leaves the basin $\ms B$ at an exponential time before
hitting the attractor $\xi$.

Of course, the existence of an attractor is superfluous, as shown by
Example \ref{ex3}, where we present a valley without an
attractor. This requirement could be replaced by weaker requisites on
the spectrum of the generator in the reversible case or on the total
variation distance between the state of the process and the invariant
measure restricted to the well $\ms W$.  Nevertheless, in several non
trivial examples, as in the case of condensed zero-range processes
\cite{bl3, bl4} which motivated this paper, attractors do exist.

\subsubsection*{Condition {\rm ({\bf V2})}.}

The second condition asserts that the process leaves the basin at an
exponential time of order given by the depth of the valley.  Example
\ref{ex4} presents a situation in which conditions ({\bf V1}), ({\bf
  V3}) hold but not ({\bf V2}) nor ({\bf V1}'). There, the order of
magnitude of the time needed for the process to reach ${\ms B}^c$ from
$\ms W$ depends on the starting point of $\ms W$.

Clearly, the depth of a valley is defined up to an equivalence
relation: if $\bs \theta'= (\theta_N' : N\ge 1)$ is another sequence
of positive numbers such that $\lim_{N\to\infty}
(\theta_N/\theta_N')=1$, the valley has also depth $\bs
\theta'$. Moreover, the depth of a valley depends on the basin. As we
shall see in Example \ref{ex2}, two different valleys $(\ms W, \ms B,
\bs \xi)$, $(\ms W, \ms B', \bs \xi)$, with $\ms B \subset \ms B'$,
may have depths of different order. Finally, the depth has not an
intrinsic character, in contrast with valleys, in the sense that it
changes if we speed up or slow down the underlying Markov process.

\subsubsection*{Condition {\rm ({\bf V3})}.}

The last condition requires the process starting from the well to
spend a negligible amount of time in the part of the basin which does
not belong to the well.

We prove in Lemma \ref{s04b} that we may replace condition ({\bf V3})
by the assumption that for every point $\bs \eta = (\eta^N : {N\ge
  1})$ in $\ms W$ and every $t>0$,
\begin{equation}
\label{f02}
\lim_{N\to\infty} \mb E_{\eta^N}\Big[\, \int_0^{\min\{t,
    \theta_N^{-1} T_{\ms B^c}\}}
\mb 1\{ \eta_{s\theta_N} \in \Delta_N \} \, ds \,\Big] \;=\; 0 \;.
\end{equation}

Condition ({\bf V3}) is necessary, as we shall see in Example
\ref{ex1}, to ensure that $\ms W$ is the well of the valley and not an
evanescent set. The Markov process presented in this example fulfills
conditions ({\bf V1}), ({\bf V2}) but not condition ({\bf V3}).
\medskip

The definition of valley focus on paths of the Markov process starting
from the well $\ms W$. Nothing is imposed on the process starting from
the annulus ${\bs \Delta}$, which may hide other wells, even deeper
than the well $\ms W$, as illustrated by example \ref{ex5}. To rule
out this eventuality, we replace condition ({\bf V3}) by assumption
({\bf V3}') which reads:

For every $\delta >0$,
\begin{equation*}
\tag*{({\bf V3}')}
\lim_{N\to\infty} \sup_{\eta\in B_N} \prob_{\eta}\Big[\, \frac 1{\theta_N} 
\, T_{{\ms B}^c} (\bs \Delta) > \delta\,\Big] \;=\; 0 \;.
\end{equation*}

Fix $\eta$ in $\Delta_N$ and note that $T_{{\ms W}\cup {\ms B}^c} =
T_{{\ms W}\cup {\ms B}^c} (\bs \Delta) \le T_{{\ms B}^c} (\bs \Delta)$
$\mb P_\eta$ almost surely. Therefore, it follows from condition ({\bf
  V3}') that the process starting from $\bs \Delta$ immediately
reaches $\ms W \cup {\ms B}^c$: For every $\delta>0$,
\begin{equation}
\label{f03}
\lim_{N\to\infty}  \sup_{\eta\in \Delta_N}
\prob_{\eta}\Big[\, \frac 1{\theta_N} 
\, T_{{\ms W}\cup {\ms B}^c} > \delta\,\Big] \;=\; 0 \;.
\end{equation}
It also follows from conditions ({\bf V2}), ({\bf V3}') that
\begin{equation}
\label{f05}
\lim_{N\to\infty} \sup_{\eta\in B_N} \mb E_{\eta}\Big[\, \int_0^{\min\{t,
    \theta_N^{-1} T_{\ms B^c}\}}
\mb 1\{ \eta_{s\theta_N} \in \Delta_N \} \, ds \,\Big] \;=\; 0 \;.
\end{equation}

These remarks lead naturally to a more restrictive definition of
valley.

\begin{definition}[S-Valley]
\label{well2}
The triple $(\ms W, \ms B, \bs \xi)$ is a S-valley of depth $\bs
\theta$ and attractor $\bs \xi$ for the Markov process $\{\eta^N_t :
t\ge 0\}$ if, for every point $\bs \eta = (\eta^N : {N\ge 1})$ in $\ms
W$, assumptions {\rm ({\bf V1}), ({\bf V2})}, {\rm ({\bf V3}')} are
fulfilled.
\end{definition}

In Example \ref{ex5} we present a triple $(\ms W, \ms B, \bs \xi)$
satisfying assumptions ({\bf V1}), ({\bf V2}), ({\bf V3}) but not
({\bf V3}'), \eqref{f03} and \eqref{f05} because $\bs \Delta$ contains
a well deeper than $\ms W$.

In many cases, it is possible to transfer from $\bs \Delta$ to ${\ms
  B}^c$ all points in $\bs \Delta$ which do not reach immediately $\ms
W \cup {\ms B}^c$, in the sense of condition \eqref{f03}, to obtain
from a valley $(\ms W, \ms B, \bs \xi)$ satisfying conditions ({\bf
  V1}), ({\bf V2}), ({\bf V3}) a new valley $(\ms W, \ms B', \bs \xi)$
satisfying conditions ({\bf V1}), ({\bf V2}), ({\bf V3}'). We refer to
Example \ref{ex5}. 

We present in Example \ref{ex6} a triple which satisfies conditions
({\bf V1}), ({\bf V2}), \eqref{f03} but not ({\bf V3}). In particular,
the first three conditions do not imply ({\bf V3}). In this example,
there is a state in the annulus $\Delta_N$ which immediately jumps to
the well $W_N$, but which is visited several times before leaving the
basin $B_N$.  

\subsection{Tunneling and Metastability}
\label{meta}

Given a sequence of Markov processes $\{\eta^N_t : t\ge 0\}$ with
values in $E_N$, we might observe a complex landscape of valleys with
a wide variety of depths. We describe in this subsection the
inter-valley dynamics.

Fix a finite number of disjoint subsets $\ms E^1_N, \dots, \ms
E^\kappa_N$, $\kappa\ge 2$, of $E_N$: $\ms E^x_N\cap \ms
E^y_N=\varnothing$, $x\neq y$. Let $\ms E_N=\cup_{x\in S}\ms E^x_N$
and let $\Delta_N=E_N \setminus \ms E_N$ so that
$$
E_N \,=\, \underbrace{\ms E^1_N\cup\dots \cup \ms E^{\kappa}_N}_{\ms E_N}
\cup\, \Delta_N\;.
$$
Denote by $\Psi_N:\ms E_N\mapsto S = \{1,2,\dots,\kappa\}$,
the projection given by
$$
\Psi_N(\eta) \;=\; \sum_{x\in S} x\, \mathbf 1\{\eta \in \ms E^x_N\}
$$
and let 
\begin{equation*}
\breve{\ms E}^x_N \,:=\, \ms E_N\setminus \ms E^x_N\,,\quad 
{\ms E}^x=(\ms E^x_N : N\ge 1) \quad\textrm{and}\quad 
\breve{\ms E}^x=(\breve{\ms E}^x_N : N\ge 1) \,. 
\end{equation*}

For a subset $A$ of $E_N$, let $\mc S^A_t$ be the generalized inverse
of the additive functional $\mc T^{A}_t$ introduced in the beginning
of this section:
$$
\mc S^{A}_t(e_{\bs \cdot})\;:=\;\sup\{s\ge 0 : 
\mc T^{A}_s(e_{\bs \cdot})\le t\}\,.
$$
It is clear that $\mc S^{A}_t< +\infty$ for every $t\ge 0$ if, and
only if, $\mc T^{A}_t\to+\infty$ as $t\to+\infty$. To circumvent the
case $\mc S^{A}_t = \infty$, add an artificial point $\mf d$ to the
subset $A$. For any path $e_{\bs \cdot}\in D(\bb R_+,E_N)$ starting at
$e_0\in A$, denote by $e^{A}_{\bs \cdot}$ the \emph{trace} of the path
$e_{\bs \cdot}$ on the set $A$ defined by $e^{A}_{t} = e_{S^{A}_t}\,$
if $S^{A}_t <+\infty$, and $e^{A}_{t} = \mf d$ otherwise. Clearly, if
$e^{A}_{t} = \mf d$ for some $t$, then $e^{A}_{s} = \mf d$, for every
$s>t$.

Denote by $\{\eta^{\ms E_N}_t: t\ge 0\}$ the $\ms E_N\cup\{\mf
d\}$-valued Markov process obtained as the trace of $\{\eta_t: t\ge
0\} $ on $\ms E_N$, and by $\{X^N_t: t\ge 0\}$ the stochastic process
defined by $X^N_t=\Psi_N(\eta^{\ms E_N}_{t})$ whenever $\eta^{\ms
  E_N}_{t}\in \ms E_N$ and $X^N_t=\mf d$ otherwise. Clearly, besides
trivial cases, $X^N_{\bs \cdot}$ is not Markovian.

Let $\bs \theta=(\theta_N : N\ge 1)$ denote a sequence of positive
numbers and, for each $x\in S$, let ${\bs \xi}_x=(\xi^N_x : N\ge 1)$
be a point in $\ms E^x$. In order to describe the asymptotic
behaviour of the Markov process on the time-scale ${\bs \theta}$ we
use a Markov process $\{\bb P_{\,x} : x\in S\}$ defined on the
canonical path space $D(\bb R_+,S)$.

\renewcommand{\theenumi}{\Alph{enumi}}
\renewcommand{\labelenumi}{(\theenumi)}

\begin{definition}[Tunneling]
\label{locmetadef} 
A sequence of Markov processes $\{\eta^N_t : t\ge 0\}$, $N\ge 1$, on a
countable state space $E = (E_N : N\ge 1)$ exhibits a tunneling
behaviour on the time-scale $\bs \theta$, with metastates $\{\ms E^x :
x\in S\}$, metapoints $\{\bs \xi_x : x\in S\}$ and asymptotic Markov
dynamics $\{\bb P_{\,x} : x\in S\}$ if, for each $x\in S$,

\begin{enumerate} 
\item[({\bf M1})]
The point ${\bs \xi}_x$ is an attractor on $\ms E^x$ in the sense that
\begin{equation*}
\lim_{N\to\infty} \inf_{\eta\in \ms E^x_N} \prob_{\eta}
\big[\,T_{{\bs \xi}_x} < T_{\breve{\ms E}^x}\,\big]\, =\, 1\;;
\end{equation*}

\item[({\bf M2})]
For every
point $\bs \eta=(\eta^N : N\ge 1)$ in $\ms E^x$, the law of
the speeded up process $\{X^N_{t\theta_N} : t\ge 0\}$ under
$\prob_{\eta^N}$ converges to $\bb P_{\,x}$ as $N\uparrow\infty$;

\item[({\bf M3})]
For every $t> 0$, 
$$
\lim_{N\to+\infty} \sup_{\eta\in \ms E^x_N} \E_{\eta} \Big[\, 
\int_0^t {\bs 1}\{\eta^N_{s\theta_N}\in \Delta_N\}\, ds \,\Big]\,
=\,0\,.
$$
\end{enumerate}
\end{definition}

Let $\bs \Delta$ denote the sequence $(\Delta_N : N\ge 1)$ and
consider the triple $(\ms E^x, \ms E^x \cup \bs \Delta, \bs \xi_x)$
for a fixed $x$ in $S$. Clearly, if $x$ is not an absorbing state for
the asymptotic Markov dynamics, the triple $(\ms E^x, \ms E^x \cup \bs
\Delta, \bs \xi_x)$ is a valley of depth of the order of $\bs \theta$.
In this case, it may happen that the triple $(\ms E^{x}, \ms E^{x}
\cup \bs \Delta, \xi_{x})$ is an inaccessible valley in the sense that
once the process escapes from $\ms E^{x}$ it never returns to $\ms
E^{x}$.  This is illustrated in Example \ref{ex7}.  In contrast, if
$x$ is an absorbing state for the asymptotic Markov dynamics not much
information is available on the triple $(\ms E^x, \ms E^x \cup \bs
\Delta, \bs \xi_x)$. Example \ref{ex4} presents a Markov process which
exhibits a tunneling behavior in which a triple is not a valley. In
this example the triple contains a well of larger order depth than
$\bs \theta$. \medskip

Suppose that property $\bf (M2)$ is satisfied for a sequence of Markov
processes and denote by $S_*\subset S$ the subset of non-absorbing
states for $\{\mathbb P_x : x \in S\}$. For the states in $S_*$ we may
replace requirement $\bf (M3)$ by property $\bf (V3)$ of valley,
namely: For each $x\in S_*$,
\begin{equation*}
\tag*{\bf (C1)}
\lim_{N \to \infty} \sup_{\eta\in \ms E^{x}_N} {\bf P}_{\eta} 
\Big[\, \frac{1}{\theta_N} T_{\breve{\ms E}^{x}}({\bf \Delta}) 
> \delta \,\Big] \,=\, 0\,.
\end{equation*}

\begin{proposition}
\label{prom3}
Assume that $\bf (M2)$ is fulfilled for a sequence of Markov processes
$\{\eta^N_t : t\ge 0\}$, $N\ge 1$. If $\bf (M3)$ is satisfied for each
$x\in S\setminus S_*$ and if {\rm ({\bf C1})} holds for any $x\in S_*$,
then $\bf (M3)$ is in force for any $x\in S$.
\end{proposition}

We arrive to the same conclusion in Proposition \ref{prom3} if we
assume instead that ({\bf C1}) holds for every state $x\in S$. This is
the content of Lemma \ref{cb3}. Actually, for an absorbing state $x$,
property ({\bf C1}) is stronger than $\bf (M3)$ because in this case
$\theta_N^{-1} T_{{\breve{\ms E}^x}}$ diverges. \medskip

The definition of tunneling examines the inter-valley dynamics between
wells with depths of the same order. It is far from a global
description since it does not exclude the possibility that $\bs
\Delta$ contains a landscape of valleys of depths of larger order than
$\bs \theta$. This situation is illustrated in Example \ref{ex5}. We
have also just seen that if $x$ is an absorbing state for the
asymptotic Markov dynamics, the set $\ms E^x$ may also contain a
landscape of valleys of larger order depth.  In order to exclude these
eventualities, we impose more restrictive conditions in the definition
of metastability. We replace ({\bf M1}) by ({\bf M1}') to ensure that
there are no wells in $\ms E^x$ of depth of order $\theta_N$ if $x$ is
an absorbing point for the asymptotic Markov dynamics; and we replace
({\bf M3}) by ({\bf M3}') to avoid wells in $\bs \Delta$ of depth of
order $\theta_N$ or larger.

\begin{definition}[Metastability]
\label{metadef}
A sequence of Markov processes $\{\eta^N_t : t\ge 0\}$, $N\ge 1$, on a
countable state space $E = (E_N : N\ge 1)$ exhibits a metastable
behaviour on the time-scale $\bs \theta$, with metastates $\{\ms E^x :
x\in S\}$, metapoints $\{\bs \xi_x : x\in S\}$ and asymptotic Markov
dynamics $\{\bb P_{\,x} : x\in S\}$, if for each $x\in S$,

\begin{enumerate} 
\item[({\bf M1}')] The point ${\bs \xi}_x$ is an attractor on $\ms
  E^x$ in the sense that for every $\delta>0$
\begin{equation*}
\lim_{N\to\infty} \sup_{\eta\in \ms E^x_N} 
\prob_{\eta}\big[\,T_{{\bs \xi}_x} > \delta
\theta_N \,\big]\, =\, 0\;;
\end{equation*}

\item[({\bf M2})] For every point $\bs \eta=(\eta^N : N\ge 1)$ in $\ms
  E^x$, the law of the speeded up process $\{X^N_{t\theta_N} : t\ge
  0\}$ under $\prob_{\eta^N}$ converges to $\bb P_{\,x}$ as
  $N\uparrow\infty$;

\item[({\bf M3}')] For every $t> 0$, 
$$
\lim_{N\to+\infty} \sup_{\eta\in E_N} \E_{\eta} \Big[\, 
\int_0^t {\bs 1}\{\eta^N_{s\theta_N}\in \Delta_N\}\, ds \,\Big]\,
=\,0\,.
$$
\end{enumerate}
\end{definition}

As for valleys, it follows from ({\bf M3}') that $\bs \Delta$ is
evanescent in the sense that for every $\delta > 0$,
\begin{equation}
\label{f06}
\lim_{N\to\infty}  \sup_{\eta\in \Delta_N} \prob_{\eta}\big[\, 
T_{\ms E_N} > \delta \theta_N \,  \,\big] \;=\; 0 \;.
\end{equation}

Example \ref{ex4} presents a Markov process which exhibits a tunneling
behaviour and fulfills condition ({\bf M3}') but violates assumption
({\bf M1}'). Example \ref{ex5} presents a Markov process with the
opposite properties. It fulfills conditions ({\bf M1}'), ({\bf M2}),
({\bf M3}) but violates assumption ({\bf M3}'). This latter example is
very instructive. It shows that the same Markov process may have
distinct metastable behaviors at different time scales. This occurs
when on one time scale there is an isolated point in the asymptotic
Markov dynamics. In longer time scales this metastate is reached by
other metastates, previous metastates coalesce in one larger
metastate, and a new metastable picture emerges.  \medskip

We conclude this subsection observing that we may define metastability
without referring to trace processes. Indeed, consider the $S$-valued
stochastic process $\hat X^N_t$ defined as
$$
\hat X^N_t\;=\; \Psi_N(\eta^N_{\sigma(t)}) \;,
$$
where $\sigma(t):=\sup\{s\le t : \eta^N_s\in \ms E_N\}.$ Note that
$\hat X^N_t$ is well defined whenever $\eta^N_t$ starts from a point
in $\ms E_N$.

\begin{proposition}
\label{corol}
In condition ({\bf M2}) of Definitions \ref{locmetadef},
\ref{metadef}, we may replace the stochastic process
$\{X^N_{t\theta_N} : t\ge 0\}$ by $\{\hat{X}^N_{t\theta_N} : t\ge
0\}$.
\end{proposition}

\subsection{The positive recurrent case}
\label{prec}

The purpose of this subsection is to provide sufficient conditions to
ensure tunneling. Assume from now on that the Markov process
$\{\eta^N_{t} : t\ge 0\}$ is irreducible and positive recurrent, and
denote by $\mu_N$ its unique invariant probability measure. It follows
from these hypotheses that the holding rates $\lambda^N (\eta) = -
R_N(\eta, \eta)$ are strictly positive, and that the discrete time
Markov chain on $E_N$ which jumps from $\eta$ to $\xi$ at rate
$R_N(\eta,\xi)/\lambda^N(\eta)$ is irreducible and recurrent.

Furthermore, for every $\eta\in E_N$ and $A\subseteq E_N$, $\mc
T^{A}_t$ diverges. Consequently, the trace of the Markov process
$\{\eta^N_{t} : t\ge 0\}$ on the set $A$, denoted by $\{\eta^A_{t} :
t\ge 0\}$, is well defined and takes values in $A$. In fact, we prove
in Proposition \ref{protra} that the trace $\{\eta^A_{t} : t\ge 0\}$
is an irreducible and positive recurrent Markov process with invariant
probability measure equal to the measure $\mu_N$ conditioned on the
set $A$.

Consider two sequences of sets $\ms W$ and $\ms B$ satisfying
\eqref{val1}. To keep notation simple, let $\Delta_N = B_N \setminus
W_N$, $\bs \Delta = (\Delta_N : N\ge 1)$, and $\ms E_N = W_N \cup
B_N^c$, $\ms E = ( {\ms E}_N : N\ge 1)$. Denote by
$$
R^{\ms E}_N(\eta,\xi)\,,\quad \eta \,,\,\xi \in \ms E_N \,,\,
\eta\not=\xi\,,
$$ 
the transition rates of the Markov process $\{\eta^{\ms E_N}_t : t\ge
0\}$, the trace of $\{\eta^N_t : t\ge 0\}$ on $\ms E_N$.  Let $R^{\ms
  W}_N: W_N \to \bb R_+$ be the rate at which the trace process jumps
to $B^c_N$:
\begin{equation*}
R^{\ms W}_N (\eta) \;:= \; \sum_{\xi\in B^c_N} R^{\ms E}_N (\eta,\xi)\;,
\end{equation*}
and let $r_{N}({\ms W}, {\ms B}^c)$ be the average of $R^{\ms W}_N$
over $W_N$ with respect to $\mu_N^{W_N}$, the measure $\mu_N$
conditioned on $W_N$:
\begin{equation}
\label{f-1}
\begin{split}
r_{N}({\ms W}, {\ms B}^c) \;:=&\; \frac 1{\mu_N(W_N)} \sum_{\eta\in W_N} 
R^{\ms W}_N (\eta) \, \mu_N(\eta) \\
\;=&\; \frac 1{\mu_N(W_N)} \sum_{\eta\in W_N} 
\sum_{\xi\in B^c_N} R^{\ms E}_N (\eta,\xi) \, \mu_N(\eta) \;.
\end{split}
\end{equation}

Next theorem presents sufficient conditions for $\ms W$ and $\ms B$ to
be the well and the basin of a valley.

\begin{theorem}
\label{teo0}
Assume that there exists a point $\bs \xi=(\xi^N : N\ge 1)$ in $\ms W$
such that for every point $\bs \eta =(\eta^N : N\ge 1)$ in $\ms W$,
\begin{equation}
\label{cc1}
\lim_{N\to \infty} {\mb E}_{\eta^N} \Big[\int_{0}^{T_{{\mb \xi}}} 
R^{\ms W}_N(\eta^N_s)\,{\bf 1}\{\eta^N_s\in W_N\} \,ds \Big] \,=\,0\;,
\end{equation}
\begin{equation} 
\label{cc2}
\lim_{N\to \infty}  r_N(\ms W,\ms B^c)\, 
{\mb E}_{\eta^N}[T_{{\bs \xi}}(\ms W)] \,=\,0\,
\end{equation}
and
\begin{equation}
\label{ccb}
\lim_{N\to \infty} r_N(\ms W,\ms B^c) \, 
\E_{\eta^N} \big[T_{{\ms B}^c}({\bf \Delta})\big] \;=\; 0\;.
\end{equation}
Then, $(\ms W, \ms B, \bs \xi)$ is a valley with depth ${\bs
  \theta}=(\theta_N : N\ge 1)$ where $\theta_N=1/r_N(\ms W,\ms B^c)$,
$N\ge 1$. 
\end{theorem}

Conditions (\ref{cc1}) and (\ref{cc2}) clearly follow from the
stronger condition
$$
\lim_{N\to\infty} \sup \{ R^{\ms W}_N(\eta) : \eta\in W_N \} 
\, {\bf E}_{\eta^N}[\, T_{\bs \xi} \,] \,=\,0\,.
$$

To state sufficient conditions for a tunneling behaviour, recall
the notation introduced in Subsection \ref{meta}.  Let $R^{\ms E}_N :
\ms E_N \times \ms E_N\to \bb R_+$ be the transition rates of the
trace process $\{\eta^{\ms E_N}_t : t\ge 0\}$, let $R^{x,y}_N:\ms
E^x_N\to\bb R_+$, $x,y\in S$, $x\not =y$, be the rate at which the
trace process jumps to the set $\ms E^y_N$:
$$
R^{x,y}_N(\eta) \;:= \; \sum_{\xi\in\ms E^y_N}R^{\ms
  E}_N(\eta,\xi)\;,
$$
and let $R^{x}_N:\ms E^x_N\to\bb R_+$, $x\in S$, be the rate at which
it jumps to the set $\breve{{\ms E}}^x_N$: $R^x_N=\sum_{y\not = x}
R^{x,y}_N$. Observe that $R^x_N$ coincides with $R^{\ms W}_N$ if $(
W_N, B_N) = ({\ms E}^x_N , {\ms E}^x_N \cup { \Delta}_N)$. 

Let $\mu^x_N$ stand for the probability measure $\mu_N$ conditioned on
$\ms E^x_N$. Denote by $r_N(\ms E^x,\ms E^y)$ the $\mu^x_N$-expectation of
$R^{x,y}_N$:
\begin{equation*}
r_N(\ms E^x,\ms E^y) \; := \; \frac{1}{\mu_N(\ms E^x_N)}
\sum_{\eta\in\ms E^x_N} R^{x,y}_N(\eta)\, \mu_N(\eta)
\end{equation*}
and by $r_N(\ms E^x,\breve{\ms E}^x_N)$ the
$\mu^x_N$-expectation of $R^x_N$ so that
$$
r_N(\ms E^x,\breve{\ms E}^x_N) \,=\, \sum_{y\not = x} 
r_N(\ms E^x,\ms E^y)\,.
$$

To guarantee that the process $\{\eta^{N}_t : t\ge 0\}$ exhibits a
tunneling behavior, we first require that each subset $\{ \ms E^x :
x\in S\}$ satisfies conditions (\ref{cc1}) and (\ref{cc2}): For each
$x\in S$, there exists a point $\bs \xi_x=(\xi^{N}_x : N\ge 1)$ in
$\ms E^x$ such that
\begin{equation}
\tag*{\bf (C2)}
\lim_{N\to \infty} {\mb E}_{\eta^N} \Big[\int_{0}^{T_{{\mb \xi}_x}}
R^{x}_N(\eta^N_s)\,{\bf 1}\{\eta^N_s\in \ms E^x_N\} \,ds \Big] \,=\,0
\end{equation}
and
\begin{equation}
\tag*{\bf (C3)}
\lim_{N\to \infty}  r_N(\ms E^x,\breve{\ms E}^x)\, 
{\mb E}_{\eta^N}[T_{{\bs \xi_x}}(\ms E^x)] \,=\,0
\end{equation}
for every point $\bs \eta=(\eta^N : N\ge 1)$ in $\ms E^x$. 

\begin{theorem}
\label{teo1}
Suppose {\bf (C2)}, {\bf (C3)} and that there exists a sequence $\bs
\theta=(\theta_N : N\ge 1)$ of positive numbers such that, for every
pair $x,y\in S$, $x\not= y$, the following limit exists
\begin{equation}
\tag*{\bf (H0)}
r(x,y) \;:= \; \lim_{N\to\infty} \theta_N\, r_N(\ms E^x,\ms E^y) \,.
\end{equation}
Assume, furthermore, that $\bf (M3)$ is satisfied for each absorbing
state $x$ of the Markov process on $S$ determined by the rates $r$
and that {\bf (C1)} holds for any non-absorbing state. Then, the
sequence of Markov processes $\{\eta^N_t : t\ge 0\}$, $N\ge 1$,
exhibits a tunneling behaviour on the time-scale $\bs \theta$, with
metastates $\{\ms E^x : x\in S\}$, metapoints $\{\bs \xi_x : x\in S\}$
and asymptotic Markov dynamics characterized by the rates $r(x,y)$,
$x,y\in S$.
\end{theorem}

Notice that in the previous theorem we might get
$$
\sum_{x\in S\setminus\{x_0\}}r(x,x_0) \,=\, 0 \quad \text{and}\quad
\sum_{x\in S\setminus\{x_0\}}r(x_0,x) \;>\; 0
$$
for some $x_0\in S_*$. In this case, the triple $(\ms E^{x_0}, \ms
E^{x_0}\cup {\bs \Delta}, {\bs \xi_{x_0}})$ turns out to be an
inaccessible valley, as it is illustrated in Example \ref{ex7}, even
tough it has the same depth than all the other wells involved in the
tunneling.

\subsection{The reversible case, potential theory}
\label{cap}

In addition to the positive recurrent assumption, let us now further
assume that $\mu_N$ is a reversible probability measure. In this case,
we may list simple conditions, all of them expressed in terms of the
capacities and the reversible measure $\mu_N$, which ensure the
existence of valleys and the tunneling behaviour.

As we have already seen, we need good estimates for the mean of entry
times. In the reversible case, the mean of an entry time has a simple
expression involving capacities, which are defined as follows. For two
disjoint subsets $A$, $B$ of $E_N$ define
$$
\mathcal C_N(A,B) \;:=\; \{f\in L^2(\mu_N) : \textrm{$f(\eta)=1$ 
$\forall$ $\eta\in A$ and $f(\xi)=0$ $\forall$ $\xi\in B$}\}\;.
$$
Let $\langle\cdot,\cdot\rangle_{\mu_N}$ stand for the scalar product
in $L^2(\mu_N)$. Denote by $D_N$ the Dirichlet form associated to the
generator $L_N$:
$$
D_N(f)\;:=\; \langle -L_N f,f \rangle_{\mu_N} \;,
$$
for every $f$ in $L^2(\mu_N)$. An elementary computation shows that
\begin{equation*}
D_N(f) \;=\; \frac 12 \sum_{\eta,\xi\in E_N} \mu_N(\eta) \, R_N(\eta,\xi)
\, \{f(\xi) - f(\eta)\}^2\;.
\end{equation*}
The capacity of two disjoint subsets $A$, $B$ of $E_N$ is defined as
$$
\Cap_N(A,B) \;:=\; \inf\big\{\,D_N(f) : f\in \mathcal
C_N(A,B)\,\big\}\;. 
$$

In the reversible context, the expressions appearing in Theorems
\ref{teo0} and \ref{teo1} can be computed by using capacities. Denote
by $f^N_{AB}:E_N \to\bb R$ the function in $\mc C_N (A,B)$ defined as
\begin{equation*}
f_{AB}^N(\eta) \;:=\; \prob_{\eta}\big[\,\tau_{A} < \tau_B \,\big]\;.
\end{equation*}
In addition, for two points $\bs \xi=(\xi^{N} : N\ge 1)$ and $\bs
\eta=(\eta^N : N\ge 1)$ in $\ms W$, $\eta^N\not=\xi^N$, $N\ge 1$, set
$f_N({\bs \eta},{\bs \xi}) \,=\, f^N_{\{\eta^N\}\{\xi^N\}}$ and
$\Cap_N({\bs \eta},{\bs \xi})=\Cap_N(\{\eta^N\},\{\xi^N\})$.

Consider two sequences of sets $\ms W$ and $\ms B$ satisfying
\eqref{val1} and recall the notation introduced in the previous
subsection. By \eqref{g02},
\begin{equation}
\label{for1}
{\mb E}_{\eta^N} \Big[\int_{0}^{T_{{\mb \xi}}} 
R^{\ms W}_N(\eta^N_s)\,{\bf 1}\{\eta^N_s\in  W_N\} \,ds \Big] 
\;=\; \frac{\langle \, R^{\ms W}_N \,  {\bf 1}\{W_N\} \,,\, 
f_N({\bs \eta},{\bs \xi}) \, \rangle_{\mu_N}}{\Cap_N( {\bs \eta} , 
{\bs \xi })}\;,
\end{equation}
\begin{equation}
\label{for2}
{\mb E}_{\eta^N}[T_{{\bs \xi}}(\ms W)] 
\;=\; \frac{  \langle \, {\bf 1}\{W_N\} \,,\, 
f_N({\bs \eta},{\bs \xi}) \,\rangle_{\mu_N} }
{ \Cap_N({\bs \eta},{\bs \xi}) }\;,
\end{equation}
and, by Lemma \ref{g03},
\begin{equation}
\label{for3}
r_N(\ms W,{\ms B}^c) \;=\; \frac{ \Cap_N(\ms W, \ms B^c) }
{ \mu_N(\ms W) \, }\,\cdot
\end{equation}
In the last identity $\Cap_N(\ms W, \ms B^c):=\Cap_{N}( W_N , {B}^c_N
)$ and $\mu_N(\ms W):= \mu_N(W_N)$. The previous relations can be used
to check conditions (\ref{cc1}) and (\ref{cc2}) in Theorem \ref{teo0}
as well as assumptions $\bf (C2)$ and $\bf (C3)$ in Theorem
\ref{teo1}.

Furthermore, since $0 \le f_N({\bs \eta},{\bs \xi}) \le 1$, and since,
by \eqref{f-1}, $\langle \, R^{\ms W}_N \, {\bf 1}\{W_N\}
\rangle_{\mu_N} = \mu_N(\ms W) r_N(\ms W,{\ms B}^c)$, by \eqref{for3},
\begin{equation*}
\begin{split}
&{\mb E}_{\eta^N} \Big[\int_{0}^{T_{{\mb \xi}}} 
R^{\ms W}_N(\eta^N_s)\,{\bf 1}\{\eta^N_s\in W_N\} \,ds \Big] \,\le \,
\frac{\Cap_N(\ms W,{\ms B}^c)} {\Cap_N({\bs \xi})} \\
&\quad\text{and}\quad 
r_N(\ms W,\ms B^c)\, {\mb E}_{\eta^N}[T_{{\bs \xi}}(\ms W)]
\,\le \, \frac{\Cap_N(\ms W,{\ms B}^c)} {\Cap_N({\bs \xi})} \; ,
\end{split}
\end{equation*}
where $\Cap_N({\bs \xi}) = \inf\{ \Cap_N({\eta}, {\xi^N}) : \eta \in
\ms W\setminus \{\xi^N\} \}$.  Hence, conditions (\ref{cc1}) and
(\ref{cc2}) follow from the stronger condition
\begin{equation}
\label{sufcond2}
\lim_{N\to\infty} \frac{\Cap_N(\ms W,{\ms B}^c)}
{\Cap_N({\bs \xi})}\,=\,0\,.
\end{equation}

\begin{theorem}
\label{teo0d}
Assume that (\ref{sufcond2}) holds for some point $\bs \xi=(\xi^N :
N\ge 1)$ in $\ms W$ and that
\begin{equation}
\label{sufcond2b}
\lim_{N\to\infty} \frac{\mu_N(B_N\setminus W_N)}{\mu_N(W_N)} 
\,=\,0\,.
\end{equation}
Then, for all points $\bs \zeta$ in $\ms W$, $(\ms W, \ms B,
\bs\zeta)$ is a valley of depth $\mu_N(\ms W)/\Cap_N(\ms W, \ms B^c)$,
$N\ge 1$.
\end{theorem}
 
Assumption (\ref{sufcond2}) is also powerful in the context of
tunneling. Recall the notation introduced at the beginning of
Subsection \ref{meta}.

\begin{theorem}
\label{teo1r}
Suppose that for each $x\in S$, there exists a point $\bs
\xi_x=(\xi^{N}_x : N\ge 1)$ in $\ms E^x$ such that
\begin{equation*}
\tag*{\bf (H1)} 
\lim_{N\to \infty} \frac{\Cap_N(\ms E^x,\breve{\ms E}^x)}
{\Cap_N(\bs \xi_x)}\;=\;0\; .
\end{equation*}
Suppose, furthermore, that $\bf (H0)$ holds for some ${\bs
  \theta}=(\theta_N : N\ge 1)$, that $\bf(M3)$ holds for each
absorbing state of the Markov dynamics on $S$ determined by the rates
$r$ and that
\begin{equation*}
\tag*{\bf (H2)}
\lim_{N \to \infty} \frac{ \mu_N(\bs \Delta) }{ \mu_N(\ms E^{x}) } 
\;=\; 0 \;,
\end{equation*}
for each non-absorbing state $x$. Then, the sequence of Markov
processes $\{\eta^N_t : t\ge 0\}$, $N\ge 1$, exhibits a tunneling
behaviour on the time-scale $\bs \theta$, with metastates $\{\ms E^x :
x\in S\}$, metapoints $\{\bs \xi_x : x\in S\}$ and asymptotic Markov
dynamics characterized by the rates $r(x,y)$, $x,y\in S$.
\end{theorem}

\begin{remark}
\label{cb2} 
In the previous theorem, we may replace condition $\bf(M3)$ for
absorbing states and condition {\rm ({\bf H2})} for non-absorbing
states by the assumption 
\begin{equation*}
\tag*{\bf (H2')}
\lim_{N \to \infty} \frac 1{\theta_N \, r_N(\ms E^x, \breve{\ms E}^x)}
\, \frac{ \mu_N(\bs \Delta) }{ \mu_N(\ms E^{x}) }\;=\; 0
\end{equation*}
for all states $x$.
\end{remark}

Note that condition ({\bf H2}) and ({\bf H2'}) are equivalent for
non-absorbing states if ({\bf H0}) holds.  This latter condition can
be expressed in terms of capacities since, by Lemma \ref{t02},
$\mu_N(\ms E^x)r_N(\ms E^x,\ms E^y)$ can be written as
$$
\frac{1}{2} \Big\{\,\Cap_N(\ms E^x, \breve{\ms E}^x) + \Cap_N(\ms
E^y, \breve{\ms E}^y) - \Cap_N\big({\ms E}^x\cup{\ms E}^y\,,\,
\ms E\setminus( {\ms E}^x\cup {\ms E}^y)\big)\,\Big\} 
$$
for every $x,y\in S$, $x\not=y$.
\medskip

One of the main steps in the proof of metastability is the replacement
result presented in Lemma \ref{s08} and in Corollary \ref{s09}. This
statement proposes a mathematical formulation of the notion of
thermalization by identifying this phenomenon with the possibility of
replacing the time integral of a function by the time integral of its
conditional expectation with respect to the $\sigma$-algebra generated
by the metastable states. The existence of attractors allows a simple
estimate, presented in Corollary \ref{s09}, which plays a key role in
all proofs.

The propositions stated above are proved in Section \ref{proof}, while
the theorems and the remark are proved in Section \ref{proof2}.

\section{Some examples}
\label{examples}

We present in this section some examples to justify the definitions of
the previous section and to illustrate some unexpected phenomena which
may occur.  

We start with a general remark concerning valleys on fixed state
spaces.  Consider a sequence of Markov processes $\{\eta^N_t : t\ge
0\}$ on some given countable space $E$ with generator $L_N$ described
by \eqref{c01}.  Denote by $\lambda_N(\eta) = \sum_{\xi\not =\eta}
R_N(\eta,\xi)$ the rate at which the process leaves the state $\eta$.
Clearly, the triple $(\{\eta\}, \{\eta\}, \eta)$ is a well of depth
$\lambda_N(\eta)^{-1}$ in the sense of Definition \ref{well}.

The first example highlights the role of condition ({\bf V3}) in
preventing some evanescent sets to be called wells.

\begin{example} 
\label{ex1}
Consider the sequence of Markov processes $\{\eta^N_t :
t\ge 0\}$ on $E=\{-1,0,1\}$ with rates given by
\begin{equation*}
R_N(-1,0) \;=\; R_N(1,0) \;=\; N\;, \quad 
R_N(0,-1) \;=\; R_N(0,1) \;=\; 1\;,
\end{equation*}
and $R_N(j,k)=0$ otherwise.
\end{example}

Obviously, we do not wish the triple $(\{-1\}, \{-1,0\}, -1)$ to be a
valley. Nevertheless, this triple satisfies conditions ({\bf V1}) and
({\bf V2}) of Definition \ref{well}. The first one is satisfied by
default. To check the second one, note that starting from $-1$
\begin{equation*} 
T_{\breve {\ms B}} \;=\; \sum_{j=1}^M \{ S_j + T_j\}\;,
\end{equation*}
where $\{S_j : j\ge 1\}$, $\{T_j : j\ge 1\}$ are independent sequences
of i.i.d.\!  exponential random variables of parameter $N$, $1$,
respectively, and $M$ is geometric random variable of parameter $1/2$,
independent of the sequences. Hence, $(1/2) \, T_{\ms B^c}$
converges in distribution, as $N\uparrow\infty$, to a mean $1$
exponential random variable. 

It is condition ({\bf V3}) which prevents the triple $(\{-1\},
\{-1,0\}, -1)$ to be a valley since the time spent at $0$ before
reaching $\{-1,1\}$ is a mean $1/2$ exponential random variable. \qed
\medskip

Next example illustrates the fact that conditions ({\bf V2}), ({\bf
  V3}) may hold while ({\bf V1}) fails.

\begin{example}
\label{ex8}
Consider the Markov process on $\{0,1,2,3\}$ with rates given by
\begin{equation*}
\begin{split}
& R_N(1,0) \;=\; R_N(2,3)\;=\; 1 - (1/N)\;,\quad
R_N(1,2) \;=\; R_N(2,1)\;=\; 1/N\;, \\
&\quad R_N(0,0) \;=\; a \;, \quad R_N(3,3)\;=\; b 
\end{split}
\end{equation*}
for some $a$, $b>0$, and $R_N(i,j)=0$ otherwise.
\end{example}

Consider the tripe $(\{1,2\}, \{1,2\}, 1)$. It is clear that condition
({\bf V1}) does not hold since the process starting from $2$ reaches
$\ms B^c = \{0,3\}$ before hitting $1$ with probability
$1-(1/N)$. Condition ({\bf V2}) is fulfilled for $\theta=1$ because
$T_{\ms B^c}$ converges to a mean one exponential time, independently
from the starting point, and condition ({\bf V3}) is in force by
default. \qed \medskip

The third example illustrates the fact that the depth of a valley
depends on the basin.

\begin{example}
\label{ex2}
Consider the sequence of Markov processes $\{\eta^N_t : t\ge 0\}$ on
$E=\{-1,0,1\}$ with rates given by
\begin{equation*}
R_N(-1,0) \;=\; R_N(1,0) \;=\; 1\;, \quad 
R_N(0,-1) \;=\; R_N(0,1) \;=\; N\;,
\end{equation*}
and $R_N(j,k)=0$ otherwise.
\end{example}

By the observation of the beginning of this section, the triple
$(\{-1\}, \{-1\}, -1)$ is a valley of depth $1$. On the other hand,
the triple $(\{-1\}, \{-1,0\}, -1)$ is a valley of depth $2$.
Condition ({\bf V1}) is satisfied by default, and condition ({\bf V2})
can be verified by representing the time needed to reach $\ms B^c$ as
a geometric sum of independent exponential random variables, as in
Example \ref{ex1}. Requirement ({\bf V3}) is readily checked.  \qed
\medskip

Next example shows that conditions ({\bf V1}), ({\bf V2}) and
\eqref{f03} do not imply ({\bf V3}).

\begin{example}
\label{ex6}
Consider the sequence of Markov processes $\{\eta^N_t : t\ge 0\}$ on
$E=\{1, 2, 3\}$ with rates given by $R_N(1,2) = N$, $R_N(2,1)=N-1$,
$R(2,3)=1$ and $R_N(i,j)=0$ otherwise. 
\end{example}

Let $\xi=1$, $W=\{1\}$ and $B=\{1,2\}$. Condition ({\bf V1}) is
fulfilled by default. Condition ({\bf V2}) is easily checked for
$\theta_N = 2$. In fact, the hitting time $\tau^N_3$ of $3$ starting
from $1$ can be written as $\sum_{1\le j\le M} \{S_j + T_j\}$, where
$\{S_j : j\ge 1\}$, $\{T_j : j\ge 1\}$ are independent sequences of
i.i.d.\!  mean $1/N$ exponential random variables and $M$ is a
geometric random variable of parameter $1/N$, independent of both
sequences. It follows from this representation that $\tau^N_3/2$
converges in distribution to a mean $1$ exponential random variable.

For similar reasons, conditions ({\bf V3}) fails: With the notation
just introduced, starting from $1$, the time spent at state $2$ before
hitting $3$, denoted in Section \ref{sec1} by $T^N_3(2)$, converges to
a mean $1$ exponential random variable.

Condition \eqref{f03}, however, is in force, since the hitting time of
the set $\{1,3\}$ starting from $2$ is of order $1/N$. \qed
\medskip

The fifth example shows that metastates might not be wells of
valleys. It presents also a triple which fulfills condition ({\bf
  V1}), ({\bf V3}) but not ({\bf V2}) nor ({\bf V1'}).

\begin{example}
\label{ex4}
Consider the sequence of Markov processes $\{\eta^N_t : t\ge 0\}$ on
$E=\{0, 1, 2\}$ with rates given by
\begin{equation*}
R_N(1,0)\; =\; N-1 \;, \;\; R_N(1,2) \; =\; 1\;, \;\;  
R_N(2,1) \; =\; N^{-1}\;, \;\; R_N(0,1) \; =\; N^2\;, 
\end{equation*}
and $R_N(j,k) = 0$ otherwise.
\end{example}

The triple $(\{1, 2\}, \{1, 2\}, 2)$ is not a well because condition
({\bf V1}) is violated. With overwhelming probability the process
starting from $1$ leaves the set $\{1,2\}$ before reaching $2$.  The
triple $(\{1, 2\}, \{1, 2\}, 1)$ is not a well either. While
conditions ({\bf V1}), ({\bf V3}) are clearly satisfied, it is not
difficult to show that condition ({\bf V2}) is violated. In fact,
starting from $1$, $T_{\ms B^c}$ converges to a mean one exponential
random variable, while starting from $2$, $N^{-1} T_{\ms B^c}$
converges to a mean one exponential random variable. It is also clear
that condition ({\bf V1}') fails in this case since on the scale of
order $1$ the process starting from $2$ never reaches $1$.

At the scale $N^{-2}$ the process exhibits a tunneling behaviour, as
described in Definition \ref{locmetadef}, with metastates $\ms E^1 =
\{0\}$ and $\ms E^2 = \{1,2\}$, $\bs \xi_2 = 1$, and asymptotic Markov
dynamics characterized by the rates $r(1,2) = 1$, $r(2,1)=0$.  It does
not exhibit a metastable behaviour, as described in Definition
\ref{metadef}, because condition ({\bf M1}') is violated. Starting
from state $2\in \ms E^2$, the process never reaches the attractor $1$
in the time scale $N^{-2}$. We have also here an example of an
absorbing set for the asymptotic dynamics which is not a valley due to
the existence of the well $\{2\}$ in the the set $\ms E^2$ of depth $N
\gg N^{-2}$. \qed\medskip

Next example shows that there might exist inaccessible valleys.

\begin{example}
\label{ex7}
Consider the sequence of Markov processes $\{\eta^N_t : t\ge 0\}$ on
$E=\{1, \dots, 5\}$ with rates given by
\begin{equation*}
\begin{split}
& R_N(j,k) \; =\; 1 \quad\text{if $j$ is even, $k$ odd and $|j-k|=1$\;, } \\
& \quad R_N(1,2) \; =\; R_N(3,4) \; =\; R_N(5,4) \;=\; N^{-1}\;, 
\quad R_N(3,2) \;=\; N^{-2}\;, \\
& \qquad R_N(j,k) \; =\; 0 \quad \text{otherwise\;.}
\end{split} 
\end{equation*}
\end{example}

The triples $(\{1\}, \{1,2\}, 1)$, $(\{3\}, \{3,4\}, 3)$, $(\{5\},
\{4, 5\}, 5)$ are valleys of depth $2N$. Moreover, at the time scale
$N$ the process exhibits a metastable behaviour, as described in
Definition \ref{metadef}, with metastates $\ms E^1 = \{1\}$, $\ms E^2
= \{3\}$, $\ms E^3 = \{5\}$ and asymptotic Markov dynamics
characterized by the rates $r(1,2) = r(2,3) = r(3,2)= 1/2$,
$r(i,j)=0$, otherwise. Note that the metastate $\ms E^1$ is
inaccessible in the sense that $r(2,1) + r(3,1)=0$. This means that in
the time scale $N$ the process starting from $1$ eventually leaves
this state, never to return. \qed \medskip

The penultimate example, very instructive, shows that different phenomena
may be observed on different scales. It also highlights the role of
conditions ({\bf V3}), ({\bf V3}').

\begin{example}
\label{ex5}
Consider the sequence of Markov processes $\{\eta^N_t : t\ge 0\}$ on
$E=\{1, \dots, 5\}$ with rates given by
\begin{equation*}
\begin{split}
& R_N(j,k) \; =\; 1 \quad\text{if $j$ is even, $k$ odd and $|j-k|=1$\;, } \\
& \quad R_N(1,2) \; =\; N^{-2}\;, \quad R_N(3,2) \;=\; N^{-3}\;,
\quad R_N(3,4) \; =\; R_N(5,4) \;=\;  N^{-1}\;, \\
& \qquad R_N(j,k) \; =\; 0 \quad \text{otherwise\;.}
\end{split} 
\end{equation*}
\end{example}

A simple computation shows that the measure $m_N$ on $E$ given by
$m_N(1) = N^2$, $m_N(2)=1$, $m_N(3)=N^3$, $m_N(4)=N^2$, $m_N(5)=N^3$
is reversible for the Markov process.  We leave to the reader to check
that $(\{3\}, \{3,4\}, 3)$, $(\{5\}, \{4, 5\}, 5)$ are valleys of
depth $2N$, and that $(\{1\}, \{1,2\}, 1)$, $(\{3,4,5\}, \{3,4,5\},
3)$, $(\{3,4,5\}, \{2,3,4$, $5\}, 3)$ are valleys of depth $2N^2$,
$2N^3$, $4N^3$, respectively. The presence of valleys of different
depths leads to diverse tunneling behaviors at different time scales.

This example illustrates that we may have valleys satisfying
conditions ({\bf V1}), ({\bf V2}) and ({\bf V3}), but not ({\bf V3}')
and \eqref{f03}. This is the case of the triple $(\{3\}, \{1,2,3$,
$4\}, 3)$. The latter conditions are violated because the annulus
$\{1,2,4\}$ contains the valley $(\{1\}, \{1,2\}, 1)$ of depth $2N^2$,
larger than $2N$ which is the depth of $(\{3\}, \{3,4\}, 3)$.  On the
scale $N$, the process starting from $1$ never reaches $3$ with
positive probability. However, condition ({\bf V3}) holds because on
the scale $N$ the process starting from $3$ never reaches $\{1,2\}$.

Note that transferring the points $1$, $2$ from $\bs \Delta$ to $\ms
B^c$, we transform the the valley $(\{3\}, \{1,2,3, 4\}, 3)$ in the
S-valley $(\{3\}, \{3, 4\}, 3)$.

At the scale $N$ one observes a tunneling between $\ms E^1 = \{3\}$
and $\ms E^2 = \{5\}$, characterized by the asymptotic Markov rates
$r(1,2) = r(2,1) = 1/2$. Assumption ({\bf M3}') is not satisfied
because the set $\Delta_N$ contains a well of depth larger than the
depth of the metastates. However, this well is never visited if the
process starts from one of the metastates.

To turn the tunneling behavior into a metastable one, we may add the
metastate $\ms E^3 = \{1\}$ and show that at scale $N$, the process
exhibits a metastable behaviour with metastates $\ms E^1 = \{3\}$,
$\ms E^2 = \{5\}$, $\ms E^3 = \{1\}$ and asymptotic Markov dynamics
characterized by the rates $r(1,2) = r(2,1) = 1/2$, $r(i,j)=0$,
otherwise. Observe that an isolated state has appeared in the
asymptotic dynamics.

At scale $N^2$, the metastates $\ms E^1 = \{3\}$, $\ms E^2 = \{5\}$
coalesce into one deeper well. In this scale the process exhibits the
metastable behaviour with metastates $\ms E^1 = \{1\}$, $\ms E^2 =
\{3,4,5\}$, and asymptotic Markov dynamics characterized by the rates
$r(1,2) = 1/2$, $r(2,1) = 0$. Note that we have here an absorbing
asymptotic state and that $\{3,4,5\}$ is not the well of a valley of
depth of order $N^2$, but the well of a valley of depth of order
$N^3$. \qed\medskip

The last example shows that the existence of an attractor is
superfluous in the definition of a valley. Denote by $E_N = (\bb
Z/N\bb Z)^d \cup (\bb Z/N\bb Z)^d$ the union of two $d$-dimensional
torii of length $N$ and denote by $(x,j)$, $x\in (\bb Z/N\bb Z)^d$,
$j=\pm 1$, the elements of $E_N$.

\begin{example}
\label{ex3}
Consider the sequence of
Markov processes $\{\eta^N_t : t\ge 0\}$ on $E_N$ with rate jumps
given by
\begin{equation*}
R_N((x,j),(x',j)) \;=\; \frac 1{2d} \,  \mb 1\{|x-x'|=1\}\;,\quad
R_N((x,j),(x,-j)) \;=\; \frac 1{\theta_N}\;,
\end{equation*}
for some rate $\theta_N$ such that $N^2 <\!\!< \theta_N <\!\!< N^d$,
and $R_N((x,i),(y,j)) =0$ otherwise.
\end{example}

It is well known that the spectral gap of the symmetric simple random
walk on the torus $(\bb Z/N\bb Z)^d$ is of order $N^{-2}$. The
evolution of the process $\eta^N_t$ is therefore quite clear. In a
time scale of order $N^2$, the process thermalizes in the torus where
it started from, and after an exponential time of order $\theta_N$ it
jumps to the other torus, replicating there the same qualitative
behavior.

Hence, each torus satisfies all reasonable conditions to be qualified
as a valley of depth $\theta_N$. Nevertheless, there is no attractor in
this example since a specific state is visited by the symmetric simple
random walk only in the scale $N^d$. \qed\medskip

%%%%%%%%%%%%%%%%%%%%%%%%%%%%%%%%%%%%%%%%%%%%%%%%%%%%%%%%%%
\section{Valleys and metastability}
\label{proof}
%%%%%%%%%%%%%%%%%%%%%%%%%%%%%%%%%%%%%%%%%%%%%%%%%%%%%%%%%%

In this section we prove some results on valleys and on tunneling.
The first lemma states that we may replace condition ({\bf V1}) by
condition ({\bf V1}') in the definition of a valley.

\begin{lemma}
\label{s05}
In Definition \ref{well}, condition {\rm ({\bf V1})} may be replaced by
condition {\rm ({\bf V1}')}.
\end{lemma}

\begin{proof}
Let us denote by $\Theta_t:=\Theta^N_t$, $t\ge 0$, the time-shift
operators on the path space $D(\bb R_+, E_N)$. Let $(\ms W, \ms B, \bs
\xi)$ be a valley of depth $\bs \theta =(\theta_N : N\ge 1)$.  Fix a
point $\bs \eta = (\eta_N : N\ge 1)$ in $\ms W$ as the starting point.
Consider the pair of random variables $T_{\bs \xi}$, $T_{\ms B^c}\circ
\Theta_{T_{\bs \xi}}$, which are independent by the strong Markov
property. According to assumption ({\bf V1}), the event $\{T_{\bs \xi}
< T_{\ms B^c}\}$ has asymptotic probability equal to one. On this
event $ T_{\bs \xi} \,+\, T_{\ms B^c}\circ \Theta_{T_{\bs \xi}} \;=\;
T_{\ms B^c}\;.$ Since, by assumption ({\bf V1}), $\theta_N^{-1} T_{\ms
  B^c}$ converges to a mean one exponential random variable,
$\theta_N^{-1} \{T_{\bs \xi} + T_{\ms B^c}\circ \Theta_{T_{\bs
    \xi}}\}$ also converges to a mean one exponential random variable.

Suppose by contradiction that there exist $\delta$, $\epsilon>0$
such that
\begin{equation}
\label{f01}
\limsup_{N\to\infty}  \prob_{\eta_N} \Big[\, \frac 1{\theta_N} 
\, T_{\bs \xi} > \delta\,\Big] \;=\; \epsilon  \;.
\end{equation}
By assumptions ({\bf V1}), ({\bf V2}), $\theta_N^{-1} (T_{\ms B^c}
\circ \Theta_{T_{\bs \xi}})$ converges to a mean one exponential
random variable and, by \eqref{f01}, $\theta_N^{-1} T_{\bs \xi}>
\delta$ with strictly positive probability. In particular,
$\theta_N^{-1} \{T_{\bs \xi} + T_{\ms B^c} \circ \Theta_{T_{\bs
    \xi}}\}$ may not converge to an exponential random variable, in
contradiction with the conclusion reached above.

Conversely, the event $\{T_{\bs \xi} < T_{\ms B^c} \}$ contains the
event $\{T_{\bs \xi} < \delta \theta_N\} \cap \{T_{\ms B^c} > \delta
\theta_N\}$ for every $\delta>0$. By assumptions ({\bf V1}'), ({\bf
  V2}), the $\mb P_{\eta_N}$- probability of this event converges to
$1$ as $N\uparrow\infty$ and then $\delta\downarrow 0$. This concludes
the proof of the lemma.
\end{proof}

The second result examines the assumptions ({\bf V3}) and ({\bf V3}')
in the definition of valleys.

\begin{lemma}
\label{s04b}
In Definition \ref{well}, assumption {\rm ({\bf V3})} may be replaced
by \eqref{f02}, and in Definition \ref{well2} assumption {\rm ({\bf
    V3}')} may be replaced by \eqref{f05}.  
\end{lemma}

\begin{proof}
The time integral in \eqref{f02} is bounded above by $\min\{ t,
\theta_N^{-1} T_{\ms B^c}(\bs \Delta)\}$. Therefore, \eqref{f02}
follows from {\rm ({\bf V3})}.

Conversely, the time integral in \eqref{f02} is bounded below by
$\theta_N^{-1} T_{\ms B^c}(\bs \Delta) \mb 1\{T_{\ms B^c} \le t
\theta_N\}$. This expression is itself bounded below by $\min \{
\theta_N^{-1} T_{\ms B^c}(\bs \Delta) , a\} \mb 1\{T_{\ms B^c} \le t
\theta_N\}$ for every $a>0$. Therefore,
\begin{equation*}
\min \{ \theta_N^{-1} T_{\ms B^c}(\bs \Delta) , a\} \;\le\;
\int_0^{\min\{t, \theta_N^{-1} T_{\ms B^c}\}}
\mb 1\{ \eta_{s\theta_N} \in \Delta_N \} \, ds \;+\;
a \mb 1\{T_{\ms B^c} > t \theta_N\}
\end{equation*}
for every $a>0$. Fix a point $\bs \eta=(\eta^N : N\ge 1)$ in $\ms
W$. By \eqref{f02}, the expectation with respect to $\mb P_{\eta^N}$
of the first term on the right hand side vanishes as $N\uparrow\infty$
for every $t>0$. By ({\bf V2}), the expectation with respect to $\mb
P_{\eta^N}$ of the second term vanishes as $N\uparrow\infty$ and then
$t\uparrow\infty$. Therefore, for every $a>0$
\begin{equation*}
\lim_{N\to\infty} \mb E_{\eta^N} \Big[ \min \{ \theta_N^{-1} T_{\ms
  B^c}(\bs \Delta) , a\} \Big] \;=\; 0\;.
\end{equation*}
This proves ({\bf V3}).

In the same way we prove that we may substitute assumption {\rm ({\bf
    V3}')} by \eqref{f05} in Definition \ref{well2}.  This concludes
the proof of the lemma.
\end{proof}

Next lemma is needed in the proof of Proposition \ref{prop1}, one of
the main results of this section.

\begin{lemma}
\label{lui}
Consider a subset $\ms A =(A_N : N\ge 1)$ of $(E_N : N\ge 1)$. Assume
that there exists $t>0$ and $\epsilon <1$ such that
\begin{equation}
\label{supd}
\limsup_{N\to \infty} \sup_{\eta \in W_N} \prob_{\eta}
\big[ T_{\ms A} > t \theta_N \big] \, < \, \epsilon\,.
\end{equation}
Then, $\sup_{\eta \in W_N} \E_{\eta}[T_{\ms A}({\ms W})] \le
[t/(1-\epsilon)] \theta_N$ for every $N$ sufficiently large and 
\begin{equation}
\label{f12}
\lim_{K\to \infty} \limsup_{N\to \infty} \sup_{\eta \in W_N} \mb E_{\eta}
\Big[ \theta_N^{-1} T_{\ms A}({\ms W}) \, 
\mb 1\{ T_{\ms A}({\ms W}) > K \theta_N\} \Big] \;=\; 0\;.
\end{equation}
\end{lemma}

\begin{proof}
The proof is a simple consequence of the strong Markov property and
assumption \eqref{supd}. Consider the sequence of stopping times
$\{I_k : k\ge 1\}$, $\{J_k : k\ge 1\}$ defined as follows. $I_1=0$,
$J_1 = t \theta_N$, 
\begin{equation*}
I_{k+1} \;=\; \inf \big\{t>J_k : \eta^N_t \in W_N \big\}\;, \quad
J_{k+1} = I_{k+1} + t \theta_N \;,\quad k\ge 1\;,
\end{equation*}
with the convention that $J_{k}=I_{k+1}=\infty$ if $I_k=\infty$ for
some $k\ge 1$. Let $M$ be the first time interval $[I_k, J_k]$ in
which the process visits $A_N$:
\begin{equation*}
M \;=\; \min \big \{ k\ge 1 : \eta^N_t \in A_N \text{ for some $t\in
  [I_k,J_k]$ or $I_k=\infty$} \big\}\;.
\end{equation*}
Clearly, $T_{\ms A}(\ms W) \le t\theta_N M$. On the other hand, for
$N$ sufficiently large, by definition of the stopping times $\{I_k:
k\ge 1\}$ and by assumption \eqref{supd}, $M$ is stochastically
dominated by a random variable $M'$ with geometric distribution given
by $P[M'=k] = (1-\epsilon) \epsilon^{k-1}$, $k\ge 1$. This concludes
the proof of the lemma.
\end{proof}

Next proposition gives an equivalent definition of a valley with
attractor.

\begin{proposition}
\label{prop1}
Assume that $(\ms W, \ms B, \bs \xi)$ is a valley of depth $\bs
\theta$ and attractor $\bs \xi$. Then, for any point $\bs \eta=(\eta^N
: N\ge 1)$ in $\ms W$,

\begin{enumerate}
\item[(i)] The hitting time of the attractor $\bs \xi$ is negligible
  with respect to the escape time from the basin $\ms B$ in the sense
  that
\begin{equation*}
\lim_{N\to\infty} \frac{\E_{\eta^N}[T_{\bs\xi}(\ms W)]}
{\E_{\eta^N}[T_{{\ms B}^c}(\ms W)]} \;=\; 0\;;
\end{equation*}

\item[(ii)] Under $\prob_{\eta^N}$, the law of the random variable
  $\,T_{{\ms B}^c}(\ms W)/\E_{\eta^N}[T_{{\ms B}^c}(\ms W)]\,$
  converges to a mean-one exponential distribution\,;

\item[(iii)] For every $\delta>0$,
\begin{equation*}
\lim_{N\to\infty} \prob_{\eta^N}\Big [\,
\frac {T_{{\ms B}^c} (\bs \Delta)} 
{\mb E_{\eta^N}[T_{{\ms B}^c}(\ms W)]} 
> \delta \, \Big ]\,=\,0\,.
\end{equation*}
\end{enumerate}
Moreover, the sequences $\theta_N$ and $\E_{\eta^N}[T_{{\ms
    B}^c}(\ms W)]$ are asymptotically equivalent in the sense that
$\lim_{N\to\infty} \theta^{-1}_N \E_{\eta^N}[T_{{\ms B}^c}(\ms W)] =
1$,

Conversely, if $(\ms W, \ms B, \bs \xi)$ is a triple satisfying
\eqref{val1} for which {\rm (i) -- (iii)} hold, then for any point
$\bs \eta=(\eta^N : N\ge 1)$ in $\ms W$, the sequence
$\E_{\eta^N}[T_{{\ms B}^c}(\ms W)]$ is asymptotically equivalent to
$\E_{\xi^N}[T_{\ms B^c}(\ms W)]$:
\begin{equation*}
\lim_{N\to\infty} \frac{\E_{\eta^N}[T_{\ms B^c}(\ms W)]}
{\E_{\xi^N}[T_{\ms B^c}(\ms W)]} \;=\; 1\;;
\end{equation*}
and $(\ms W, \ms B, \bs \xi)$ is a valley of depth $\bs \theta$,
where $\theta_N = \E_{\xi^N}[T_{\ms B^c}(\ms W)]$.
\end{proposition}

It is implicit in the statement of this proposition that the time
spent in the well $\ms W$ before leaving the basin $\ms B$, $T_{{\ms
    B}^c}(\ms W)$, has finite expectation with respect to any $\mb
P_{\eta^N}$ for sufficiently large $N$, as well as the time spent in
the well $\ms W$ before reaching the attractor $\bs \xi$,
$T_{\bs\xi}(\ms W)$. 

\begin{proof}[Proof of Proposition \ref{prop1}]
Assume that $(\ms W, \ms B, \bs \xi)$ is a valley of depth $\bs
\theta$ and attractor $\bs \xi$. We first claim that 
\begin{equation}
\label{f09}
\lim_{N\to \infty} \sup_{\eta \in W_N} \mb E_{\eta}
\big[ \theta_N^{-1} \, T_{\bs \xi} (\ms W) \big] \;=\; 0 \,.
\end{equation}
This assertion follows from ({\bf V1}') and the previous lemma with
$\ms A = \{\bs \xi\}$, $t=\delta$, $\epsilon = 1/2$. 

Fix a point $\bs \eta=(\eta^N : N\ge 1)$ in $\ms W$. We claim that
\begin{equation}
\label{f10}
\lim_{N\to \infty} \mb E_{\eta^N}
\big[ \theta_N^{-1} \, T_{\ms B^c} (\ms W) \big] \;=\; 1 \,.
\end{equation}
Three ingredients are needed to prove this result. The convergence of
$\theta_N^{-1} \, T_{\ms B^c}$ to a mean one exponential random
variable, a bound on $\mb E_{\eta^N} [ \theta_N^{-1} \, T_{\ms B^c}
(\ms W)]$ provided by the previous lemma, and the fact that the
process does not spend too much time in $\bs \Delta$.

We start with the proof of the lower bound. Fix $\delta>0$, $t>0$.  On
the set $\{T_{\ms B^c} (\bs \Delta) \le \delta \theta_N\}$, we have
that $T_{\ms B^c} (\ms W) \ge T_{\ms B^c} - \delta
\theta_N$. Therefore,
\begin{equation*}
T_{\ms B^c} (\ms W) \;\ge\; - \delta \theta_N \;+\; T_{\ms B^c} 
\, \mb 1\{T_{\ms B^c} (\bs \Delta) \le \delta \theta_N\} \;.
\end{equation*}
Replacing $T_{\ms B^c}$ by $\min\{T_{\ms B^c} , t \theta_N\}$ we
obtain the estimate
\begin{equation*}
T_{\ms B^c} (\ms W) \;\ge\; - \delta \theta_N \;-\; t \theta_N
\mb 1\{T_{\ms B^c} (\bs \Delta) > \delta \theta_N\}
\;+\; \min\{T_{\ms B^c} , t \theta_N\}
\end{equation*}
which holds for all $\delta>0$, $t>0$.

By ({\bf V3}), the expectation with respect to $\mb P_{\eta^N}$ of the
second term on the right hand side divided by $\theta_N$ vanishes as
$N\uparrow\infty$ for any fixed $\delta>0$, $t>0$.  By ({\bf V2}), the
expectation with respect to $\mb P_{\eta^N}$ of the third term on the
right hand side divided by $\theta_N$ converges to $1$ as
$N\uparrow\infty$ and then $t\uparrow\infty$. Therefore,
\begin{equation*}
\liminf_{N\to\infty} \mb E_{\eta^N} \big[ \theta_N^{-1}
T_{\ms B^c} (\ms W) \big] \;\ge\; 1\;.
\end{equation*}
 
The proof of the upper bound is simpler. For every $A>0$,
\begin{equation*}
\mb E_{\eta^N} \big[ T_{\ms B^c} (\ms W) \big] \;\le\;
\mb E_{\eta^N} \big[ \min\{ T_{\ms B^c} , A \theta_N \} \big]
\;+\; \mb E_{\eta^N} \big[ T_{\ms B^c} (\ms W) 
\mb 1 \{ T_{\ms B^c} (\ms W) > A \theta_N \} \big]\;.
\end{equation*}
By ({\bf V2}), the first term on the right hand side divided by
$\theta_N$ converges to $1$ as $N\uparrow\infty$ and then
$A\uparrow\infty$. By ({\bf V2}), \eqref{supd} holds with $\ms A=
\ms B^c$, $\epsilon =1/2$ and some $t<\infty$. Therefore, by
\eqref{f12}, the second term divided by $\theta_N$ vanishes as
$N\uparrow\infty$ and then $A\uparrow\infty$. This concludes the proof
of \eqref{f10}.

Assertion (i) follows from \eqref{f09} and \eqref{f10}, and assertion
(iii) from ({\bf V3}) and \eqref{f10}. Finally, $T_{\ms B^c} (\ms W) =
T_{\ms B^c} - T_{\ms B^c} (\mb \Delta)$. By ({\bf V2}), $\theta_N^{-1}
T_{\ms B^c}$ converges in distribution to a mean one exponential
random variable, and, by ({\bf V3}), $\theta_N^{-1} T_{\ms B^c} (\mb
\Delta)$ converges to $0$ in probability. Assertion (ii) follows from
these facts and from \eqref{f10}. The final claim of the first part of
the proposition has been proved in \eqref{f09}.

\medskip

To prove the converse, suppose that conditions (i) -- (iii) hold. We
first prove that ({\bf V1}), ({\bf V2}), ({\bf V3}) are in force with
$\theta_N$ replaced by the sequence $\theta(\eta^N) = \mb E_{\eta^N} [
T_{\ms B^c} (\ms W) ]$, which depends on the point $\bs \eta=(\eta^N :
N\ge 1)$.  In this case, condition ({\bf V3}) corresponds to (iii). To
prove ({\bf V2}), note that $T_{\ms B^c} = T_{\ms B^c} (\ms W) +
T_{\ms B^c} (\bs \Delta)$. By (ii), $\theta(\eta^N)^{-1} T_{\ms B^c}
(\ms W)$ converges in distribution to a mean one exponential random
variable and, by (iii), $\theta(\eta^N)^{-1} T_{\ms B^c} (\bs \Delta)$
vanishes in probability. Therefore, ({\bf V2}) holds. Finally, on the
set $\{T_{\bs \xi} < T_{\ms B^c}\}$, $T_{\bs \xi} = T_{\bs \xi} (\ms
W) + T_{\bs \xi} (\bs \Delta)$ and $T_{\bs \xi} (\bs \Delta) \le
T_{\ms B^c} (\bs \Delta)$. By (i) and (iii), $\theta(\eta^N)^{-1}
T_{\bs \xi} (\ms W)$ and $\theta(\eta^N)^{-1} T_{\ms B^c} (\bs
\Delta)$ vanish in probability as $N\uparrow\infty$. On the other
hand, by ({\bf V2}), already proved, $\theta(\eta^N)^{-1} T_{\ms B^c}$
converges in distribution to a mean one exponential variable.  This
proves ({\bf V1}).

It remains to show that the sequences $\theta(\eta^N) = \mb E_{\eta^N}
[ T_{\ms B^c} (\ms W) ]$ and $\mb E_{\xi^N} [ T_{\ms B^c} (\ms W) ]$
are asymptotically equivalent in the sense that their ratio converges
to $1$.

By (ii) and Lemma \ref{lui}, the sequence $\theta(\eta^N)^{-1} T_{\ms
  B^c} (\ms W)$ is uniformly integrable with respect to $\mb
P_{\eta^N}$. Therefore, by ({\bf V1}),
\begin{equation*}
\lim_{N\to\infty} \frac 1{\theta(\eta^N)} 
\mb E_{\eta^N} \Big[ T_{\ms B^c} (\ms W) 
\, \mb 1 \{T_{\bs \xi} < T_{\ms B^c}\} \Big] \;=\; 1\; .
\end{equation*}
By the strong Markov property and the explicit form of $T_{\ms B^c}
(\ms W)$, the expectation is equal to
\begin{equation*}
\frac 1{\theta(\eta^N)} \mb E_{\eta^N} \Big[ T_{\bs \xi} (\ms W)
\, \mb 1 \{T_{\bs \xi} < T_{\ms B^c}\} \Big] \;+\;
\frac 1{\theta(\eta^N)} \mb E_{\xi^N} \Big[ T_{\ms B^c} (\ms W)
\Big] \, \mb P_{\eta^N} \big[ T_{\bs \xi} < T_{\ms B^c} \big]\;. 
\end{equation*}
By (i), the first term vanishes as $N\uparrow\infty$. Since by ({\bf
  V1}) $\mb P_{\eta^N} [ T_{\bs \xi} < T_{\ms B^c} ]$ converges to
$1$, $\mb E_{\eta^N} [ T_{\ms B^c} (\ms W) ]$ and $\mb E_{\xi^N} [
T_{\ms B^c} (\ms W) ]$ are asymptotically equivalent. This concludes
the proof of the proposition.
\end{proof}

We conclude this section with the proofs of Propositions \ref{prom3}
and \ref{corol}. Let us first fix a metric in the path space $D(\bb
R_+,S\cup \{{\mf d}\})$ which induces the Skorohod topology. In what
follows, we identify the point $\mf d$ with $0\in \bb Z$ so that
$S\cup \{{\mf d}\}$ is a metric space with the metric induced by $\bb
Z$.

For each integer $m\ge 1$, let $\Lambda_m$ denote the class of
strictly increasing, continuous mappings of $[0,m]$ onto itself. If
${\lambda} \in \Lambda_m$, then $\lambda_0=0$ and $\lambda_m=m$. In
addition, consider the function
$$
g_m(t)\;=\;\left\{
\begin{array}{ll}
1 & \textrm{if $t\le m-1$}\;,\\
m-t & \textrm{if $m-1 \le t\le m$\;,}\\
0 & \textrm{if $t\ge m$}\;.
\end{array}\right.
$$
For any integer $m\ge 1$ and $e, \hat e \in D(\bb R_+,S\cup\{\mf
d\})$, define $d_m(e,\hat e)$ to be the infimum of those positive
$\epsilon$ for which there exists in $\Lambda_m$ a $\lambda$
satisfying
$$
\sup_{t\in [0,m]} |\lambda_t-t| \;<\; \epsilon
$$
and
$$
\sup_{t\in [0,m]}|\, g_m(\lambda_t) \,e_{\lambda_t} - g_m(t) 
\,\hat e_{t} \,| \;<\; \epsilon\;.
$$
Finally, we define the metric in $D(\bb R_+,S\cup \{\mf d\})$ by
$$
d( e,\hat e) \;=\; \sum_{m=1}^{\infty} 
2^{-m}(1\land d_m( e, \hat e))\;.
$$
This metric induces the Skorohod topology in the path space $D(\bb
R_+,S\cup \{\mf d\})$ (cf. \cite{b}).

For any path $ e\in D(\bb R_+,S\cup\{\mf d\})$ denote by $( \tau_n( e)
: n\ge 0 )$ the sequence of jumping times of $ e$: Set $\tau_0( e)=0$
and, for $n\ge 1$, we define $\tau_n( e)$ as
\begin{equation*}
\tau_n( e) \;:=\; \inf\{t>\tau_{n-1}( e) : 
 e_t \neq  e_{\tau_{n-1}( e)} \} \;,
\end{equation*}
with the convention that $\tau_n= \infty$ if $\tau_{n-1}=\infty$ and,
as usual, $\inf \varnothing = +\infty $. \medskip

Proposition \ref{corol} is a consequence of the following result.

\begin{proposition}
\label{nt}
Suppose that $\{\eta^N_t:t\ge 0\}$, $N\ge 1$, satisfies $\bf (M3)$ for
any $x\in S$. Then, for any $x\in S$ and point $\bs \eta=(\eta^N :
N\ge 1)$ in $\ms E^x$,
\begin{equation*}
\lim_{N\to\infty} {\bf E}_{\eta^N}\big[ \,d(X^N,\hat X^N)\, \big] \;=\; 0\;.
\end{equation*}
\end{proposition}

\begin{proof}
Fix arbitrary integers $m\ge 1$ and $N\ge 1$. To keep notation
simple, set $\tau_n:=\tau_n(X^N)$ and $\hat \tau_n := \tau_n(\hat
X^N)$, $n\ge 0$. Define the random
variables
\begin{equation*}
{\mf n} \;:=\;\sup\{j\ge 0 : \hat \tau_j < m\}\;
\end{equation*}
and
$$
T(X^N)\;:=\; \tau_{\mf n + 1} \land m\;.
$$
In Lemma \ref{dmx} below we show that ${\bf P}_{\eta^N}$-a.s.,
\begin{equation}
\label{edmx}
d_m( X^N , \hat X^N ) \;\le\;  | S | \,\max 
\big\{  \hat \tau_{\mf n} - \tau_{\mf n} \;;\; m-T(X^N) \big\}\;.
\end{equation}
To estimate the right hand side in (\ref{edmx}), observe that
$$
\hat \tau_{\mf n} - \tau_{\mf n}
\;=\; \mc T_{\hat \tau_{\mf n}}^{\Delta_N} \;\le\; \mc
T^{\Delta_N}_m\;, 
$$
where $\mc T_t^{\Delta_N}$ is the time spent by $\{\eta^N_t:t\ge 0\}$ in
$\Delta_N$ in the time interval $[0,t]$, introduced in
\eqref{timecut}. On the other hand, in the case $\tau_{\mf n +1} < m$,
$m-T(X^N)$ can be written as $m-\{\hat \tau_{\mf n} + [\tau_{\mf n +1}
- \tau_{\mf n}]\} + [\hat \tau_{\mf n} - \tau_{\mf n}]$.  Since
$\hat \tau_{\mf n} - \tau_{\mf n}$ is the time spent by $\{\eta^N_t:t\ge
0\}$ in $\Delta_N$ in the time interval $[0, \hat \tau_{\mf n}]$ and
$m-\{\hat \tau_{\mf n} + [\tau_{\mf n +1} - \tau_{\mf n}]\}$ is the
time spent in $\Delta_N$ in the time interval $[\hat \tau_{\mf n},
m]$,
\begin{equation*}
m-T(X^N) \;\le\; \mc T^{\Delta_N}_m\;. 
\end{equation*}
Therefore, by (\ref{edmx}) we have just shown that
$$
d(X^N,\hat X^N)\; \le \; \sum_{m=1}^{\infty} 2^{-m}
(1\land |S|\mc T^{\Delta_N}_m)\;.
$$
The desired result follows from this estimate and property ({\bf M3}).
\end{proof}

\begin{lemma}
\label{dmx}
For any integers $N,m\ge 1$, (\ref{edmx}) holds ${\bf
  P}_{\eta^N}$-almost surely.
\end{lemma}

\begin{proof}
Fix two integers $N,m\ge 1$. All assertions in what follows must be
understood in the ${\bf P}_{\eta^N}$-a.s. sense. Recall the notation
introduced in the previous lemma.

Let us list some evident properties of $X^N$ and $\hat X^N$: First
notice that for all $0\le j\le {\mf n}-1$, we have $\hat \tau_{j+1} -
\hat \tau_{j} \,\ge\, \tau_{j+1} - \tau_{j}$ and $ X^N_s=\hat X^N_{t}$
for $(s,t)\in [\tau_j,\tau_{j+1}[\times[\hat \tau_j,\hat
\tau_{j+1}[$. Furthermore, $\tau_{\mf n} < T(X^N) \le m$,
$X^N_{\tau_{\mf n}}\not = \mf d$ and $ X^N_s=\hat X^N_{t}$ for
$(s,t)\in [\tau_{\mf n},T(X^N)[ \times [\hat \tau_{\mf n},m[$.

In particular, since $\hat \tau_{\mf n} < m$, we may choose
$\epsilon>0$ small enough such that $\tau_{\mf n} < T(X^N)-\epsilon\,$
and $\,\hat \tau_{\mf n} < m-\epsilon$. Now, let $\lambda\in
\Lambda_m$ be given by: $\lambda_{\hat\tau_j}=\tau_j$, for $j\le \mf
n$, $\lambda_{m-\epsilon}=T(X^N)-\epsilon$, $\lambda_m=m$ and we
complete $\lambda$ on $[0,m]$ by linear interpolation. Then,
$$
\sup_{t\in [0,m]} |\lambda_t-t| \; \le \; 
\max\{ \hat\tau_{\mf n} - \tau_{\mf n}, m - T(X^N)\}\;.
$$
Moreover, since $\lambda_t \le t$, $0\le t\le m$,
\begin{eqnarray*}
\sup_{t\in[0,m-\epsilon]}\big|g_m(\lambda_t) X^N_{\lambda_t}
- g_m(t) \hat X^N_t \big| &\le& |S| \,\sup_{t\in[0,m-\epsilon]}
|g_m(\lambda_t)-g_m(t)| \\
&\le& |S| \,\sup_{t\in[m-1,m-\epsilon]} |\lambda_t-t|
\end{eqnarray*}
and
\begin{eqnarray*}
\sup_{t\in[m-\epsilon,m]}\big|g_m(\lambda_t)
 X^N_{\lambda_t} - g_m(t) \hat X^N_t\big| &\le&
|S| \, \sup_{t\in[m-\epsilon,m]}\big(\, |g_m(\lambda_t)-g_m(t)| 
+ 2|g_m(t)|\,\big)\\
&\le& |S| \,\sup_{t\in[m-\epsilon,m]} |\lambda_t-t| + 2\kappa \epsilon \;.
\end{eqnarray*}
Since $\epsilon$ may be taken arbitrary small, the claim is proved.
\end{proof}

We now turn to the poof of Proposition \ref{prom3}. For every $e\in
D(\bb R_+ , S\cup \{ \mf d\})$, denote by $J_t(e)$ the number of jumps
up to time $t$:
$$
J_t( e) \; := \; \sup\{j\ge 0 : \tau_j( e)\le t\}\;.
$$

\begin{proof}[Proof of Proposition \ref{prom3}]
Fix an arbitrary non-absorbing state $x_*\in S$ for the Markov process
$\{ \bb P_{x} : x\in S \}$, a point ${\bs \eta} = (\eta^N : N\ge 1)$
in $\ms E^{x_*}$ and a time $t>0$. It suffices to show that ${\bf
  E}_{\eta^N}[\mc T^{\Delta_N}_t] \to 0$ as $N\to\infty$.

For any integer $K\ge 1$,
\begin{equation}
\label{mmt}
 \mc T^{\Delta_N}_t \;\le\; {\bf 1}\{ J_t(\hat X^N) \ge K \} \, t 
\;+\; {\bf 1}\{ J_t(\hat X^N) < K \} \,\mc T^{\Delta_N}_t \;, 
\end{equation}
${\bf P}_{\eta^N}$-almost surely. The subset $\{J_t \ge K\} \subseteq
D(\bb R_+, S \cup \{\mf d\})$ is closed for the Skorohod
topology. Therefore, by property $\bf (M2)$,
$$
\limsup_{N\to \infty} {\bf P}_{\eta^N}[J_t(\hat X^N) \ge K] 
\;\le\; \limsup_{N\to \infty} {\bf P}_{\eta^N}[J_t(X^N) \ge K] 
\;\le\; {\bb P}_{x^*}[J_t \ge K]\;.
$$
The right hand side vanishes as $K\uparrow \infty$. From this and
(\ref{mmt}), it follows that
$$
\limsup_{N\to \infty} {\bf E}_{\eta^N}[\mc T^{\Delta_N}_t] 
\;\le \; \limsup_{K\uparrow\infty} \limsup_{N\to\infty} 
{\bf E}_{\eta^N}[ {\bf 1}\{ J_t(\hat X^N) < K \} \, \mc T^{\Delta_N}_t]\;.
$$
In consequence, in order to conclude the proof it is enough to show that
\begin{equation}
\lim_{N\to\infty} {\bf E}_{\eta^N}[ {\bf 1}\{ J_t(\hat X^N) = i \} 
\, \mc T^{\Delta_N}_t]\;=\;0\;,\quad \forall i\ge 0 \;.
\end{equation}
Fix some integer $i\ge 0$. To keep notation simple, denote $\hat
J_t:=J_t(\hat X^N)$ and let $( \hat\tau_n : n\ge 0 )$ stand for the
jumping times of $\hat X^N$. Recall that we denote by $S_*$ the set of
non-absorbing states for $\{ \bb P_x : x\in S\}$ and set $\ms E^*_N =
\cup_{x\in S_*} \ms E^x_N$. On the event $\{ \hat J_t=i \}$ let us
define
$$
I \;:=\; \inf \big\{ 0\le j\le i : \hat X^N_{\hat \tau_j} 
\in S\setminus S_* \big\}\;,
$$
so that $I=\infty$ if and only if $\hat X_s\in S_*$, for all $0\le
s\le t$.  On the one hand, ${\bf P}_{\eta^N}-$a.s.,
\begin{eqnarray*}
{\bf 1}\{ \hat J_t = i \,;\, I=\infty \}\, \mc T^{\Delta_N}_t 
&\le& {\bf 1}\{ \hat J_t = i \,;\, I=\infty \} \sum_{j=1}^{i+1} 
\int_{\hat \tau_{j-1}}^{\hat \tau_j\land t} 
{\bf 1}\{ \eta^N_s \in \Delta_N \} \,ds \\
&\le& \sum_{j=1}^{i+1} {\bf 1}\{ \eta^N_{\hat \tau_{j-1}}\in \ms E^*_N
\} \,\int_{\hat\tau_{j-1}}^{\hat\tau_j\land t} 
{\bf 1}\{ \eta^N_s \in \Delta_N \} \, ds  \;.
\end{eqnarray*}
Thus, applying the strong Markov property we get
$$
{\bf E}_{\eta^N} \big[ {\bf 1}\{ \hat J_t = i \,;\, I=\infty \}\, 
\mc T^{\Delta_N}_t \big] \;\le\; (i+1) \,\sup_{x\in S_*} 
\sup_{\eta \in \ms E^x_N} {\bf E}^N_{\eta}
\big[ t \land T_{\breve{\ms E}^x}({\bs\Delta})\big]\;.
$$
The right hand side vanishes as $N\to 0$ by assumption $({\bf C1})$
for the non-absorbing states. On the other hand, for any $0\le \ell
\le i$ we have that, ${\bf P}_{\eta^N}-$a.s., on the event $\{ \hat
J_t = i \,;\, I= \ell \}$,
\begin{eqnarray*}
\mc T^{\Delta_N}_t &\le&\sum_{j=1}^{\ell} 
\int_{\hat \tau_{j-1}}^{\hat \tau_j\land t} 
{\bf 1}\{ \eta^N_s \in \Delta_N \} \,ds 
\;+\; \int_{\hat \tau_{\ell}}^{\hat \tau_{\ell} + t} 
{\bf 1}\{ \eta^N_s \in \Delta_N \} \,ds\;.
\end{eqnarray*}
By applying the strong Markov processes as before, we show that ${\bf
  E}_{\eta^N} \big[ {\bf 1}\{ \hat J_t = i \,;\, I=\ell \}\, \mc
T^{\Delta_N}_t \big]$ is bounded above by
$$
\ell \,\sup_{x\in S_*} \sup_{\eta \in \ms E^x_N} 
{\bf E}^N_{\eta}\big[ t \land T_{\breve{\ms E}^x}({\bs\Delta})\big] 
\;+\; \sup_{\eta\in S\setminus S_*} {\bf E}^N_{\eta}
\big[ \mc T_{t}^{{\bs\Delta}} \big] \;.
$$
As $N \uparrow \infty$, the first term vanishes as before while the
second one vanishes by assumption $\bf (M3)$ for absorbing
states. This concludes the proof.
\end{proof}

The same proof yields the following version of Proposition \ref{prom3}
which does not distinguish between absorbing and non-absorbing states.

\begin{lemma}
\label{cb3} 
Assume that $\bf (M2)$ is fulfilled for a sequence of Markov processes
$\{\eta^N_t : t\ge 0\}$, $N\ge 1$. Then, condition $\bf (M3)$ is
satisfied if for each $x$ in $S$,
\begin{equation*}
\lim_{N\to \infty} \sup_{\eta\in \ms E^{x}_N} {\bf P}_{\eta} 
\Big[\, \frac{1}{\theta_N} T_{\breve{\ms E}^{x}}({\bf \Delta}) 
> \delta \,\Big] \,=\, 0\,.
\end{equation*}
\end{lemma}

\begin{proof}
The proof is simpler than the previous one. We do not need to
introduce the variable $I$. We estimate $\mc T^{\Delta_N}_t$ as in the
case $I=\infty$ to get that
\begin{equation*}
{\bf E}_{\eta^N} \big[ {\bf 1}\{ \hat J_t = i \}\, 
\mc T^{\Delta_N}_t \big] \;\le\; (i+1) \,\sup_{x\in S} 
\sup_{\eta \in \ms E^x_N} {\bf E}^N_{\eta}
\big[ t \land T_{\breve{\ms E}^x}({\bs\Delta})\big]\;.
\end{equation*}
This expression vanishes as $N\uparrow\infty$ by assumption.
\end{proof}

%%%%%%%%%%%%%%%%%%%%%%%%%%%%%%%%%%%%%%%%%%%%%%%%%%%%%%%%%%%%%%%%%%%%%%%
\section{Proof of the main theorems}
\label{proof2}
%%%%%%%%%%%%%%%%%%%%%%%%%%%%%%%%%%%%%%%%%%%%%%%%%%%%%%%%%%%%%%%%%%%%%%%

We prove in this section the main results of the article. The proofs
rely on some results on recurrent Markov processes presented in
Section \ref{sec03}.

\subsection*{Proof of Theorem \ref{teo0}}

Next statement plays a central role in the proof of Theorem \ref{teo0}.

\begin{proposition}
\label{cb}
Consider two sequences of sets $\ms W$ and $\ms B$ satisfying
\eqref{val1}. Assume that there exists a point $\bs \xi=(\xi^N : N\ge
1)$ in $\ms W$ such that for every point $\bs \eta =(\eta^N : N\ge 1)$
in $\ms W$ \eqref{cc1} and \eqref{cc2} hold. Then, condition {\bf
  (V1)} is in force. Moreover, the law of $r_N(\ms W,\ms B^c)$ $T_{\ms
  B^c}(\ms W)$ under ${\bf P}_{\eta^N}$ converges to a mean-one
exponential distribution, as $N\uparrow \infty$, and
\begin{equation}
\label{aa1}
\lim_{N\to \infty}   r_N(\ms W,\ms B^c) \, 
\E_{\eta^N} \big[T_{{\ms B}^c}(\ms W)\big] \;=\; 1
\end{equation}
for any point ${\bs \eta}=(\eta^N : N\ge 1)$ in $\ms W$
\end{proposition}

The proof of this proposition is divided in several lemmas.  Recall
that $\ms E_N=W_N\cup B^c_N$, $N\ge 1$, and that $\{\eta^{\ms E_N}_t :
t\ge 0\}$ stands for the trace of $\{\eta^N_t : t\ge 0\}$ on $\ms
E_N$. For any ${\bs \theta}=(\theta_N : N\ge 1)$, properties
(\ref{cc1}) and (\ref{cc2}) hold for $\{\eta^{\ms E_N}_t : t\ge 0\}$
if and only if they do so for the speeded up process $\{\eta^{\ms
  E_N}_{\theta_Nt} : t\ge 0\}$.  Furthermore, condition {\bf (V1)}
remains invariant by any re-scale of time, while \eqref{aa1} and the
assertion preceding it are implied by the corresponding claims for
$\{\eta^{\ms E_N}_{\theta_Nt} : t\ge 0\}$. In consequence, speeding up
the process appropriately, we may assume in Proposition \ref{cb} that
\begin{equation}
\label{f00}
r_N(\ms W, \ms B^c) \,=\, 1 \quad \forall N\ge 1
\end{equation}
and condition (\ref{cc2}) becomes
\begin{equation}
\label{cc2b}
\lim_{N\to \infty} \sup_{\eta\in W_N}
{\bf E}_{\eta}[T_{\bs \xi}(\ms W)]\,=\,0\,.
\end{equation}
To prove the last two assertions of Proposition \ref{cb}, we show that
the law of $T_{\ms B^c}(\ms W)$ under ${\bf P}_{\eta^N}$ converges to
a mean-one exponential distribution and that
\begin{equation}
\label{lc}
\lim_{N\to \infty} {\bf E}_{\eta^N}[T_{\ms B^c}(\ms W)] \, = \, 1\,.
\end{equation}

We identify the trace process $\{\eta^{\ms E_N}_t : t\ge 0\}$ with the
first marginal of the $\ms E_N\times \{0,1\}-$valued Markov process
$\{(\eta^{\ms E_N}_t,X^N_t) : t\ge 0 \}$ defined as follows. To keep
notation simple, let $\breve x = 1- x$, $x=0, 1$. The transition rates
for $\{(\eta^{\ms E_N}_t,X^N_t) : t\ge 0 \}$ are the following:

\begin{itemize}
\item From each $(\eta,x)\in W_N\times \{0,1\}$, the process jumps to
  $(\xi, x)$ (resp. to $(\xi, \breve x)$) with rate $ R^{\ms
    E}_N(\eta,\xi)$ for any $\xi\in W_N$ (resp. for any $\xi\in
  B_N^c$).

\item From each $(\eta,x)\in B^c_N\times \{0,1\}$, the process jumps
  to $(\xi,x)$ with rate $ R^{\ms E}_N(\eta,\xi)$, for any $\xi \in
  {\ms E}_N$.
\end{itemize}

Let ${\bf P}_{(\eta,x)}$, $(\eta,x)\in {\ms E_N}\times \{0,1\}$, be
the probability measure on $D(\bb R_+, {\ms E_N}\times \{0,1\})$
induced by the Markov process $\{(\eta^{\ms E_N}_t,X^N_t) : t\ge 0 \}$
starting from $(\eta,x)$. Hence, for any starting point $(\eta,x)\in
{\ms E_N}\times \{0,1\}$, the law of the marginal $\{ \eta_t^{\ms E_N}
: t\ge 0\}$ on $D(\bb R_+,\ms E_N)$ under ${\bf P}_{(\eta,x)}$ is
${\bf P}_\eta$.

By Proposition \ref{s06}, the conditioned probability measure
$\mu^{\ms E}_N(\,\cdot\,) := \mu_N(\,\cdot \,|\, \ms E_N)$ is the
invariant probability measure for the trace process $\{\eta^{\ms E_N}
: t\ge 0\}$. Define the probability measure on $\ms E_N\times \{0,1\}$
by 
$$
{\mf m}_N(\eta,x) \,=\, (1/2) \,\mu^{\ms E}_N(\eta)\,,
\quad \textrm{ for $(\eta,x)\in {\ms E_N}\times \{0,1\}$ \,.}
$$
We may check that ${\mf m}_N$ is an invariant probability measure for
$\{(\eta^{\ms E_N}_t,X^N_t) : t\ge 0 \}$. In particular, $\{(\eta^{\ms
  E_N}_t,X^N_t) : t\ge 0 \}$ is positive recurrent.

Clearly, for any $\eta\in W_N$, the law of $T_{{\ms B^c}}(\ms W)$
under $\prob_\eta$ coincides with the law of the first jump
$$
\inf\big\{ t>0 : X^N_{t} \not= X^N_{0}\big\}
$$
under $\prob_{(\eta,x)}$, for any $x\in\{0,1\}$. Hence, to prove that
$T_{{\ms B^c}}(\ms W)$ converges to a mean one exponential law it is
enough to show that the second coordinate of the trace of the process
$\{(\eta^{\ms E_N}_t,X^N_t) : t\ge 0 \}$ on $W_N\times \{0,1\}$
converges to a Markov process on $\{0,1\}$ which jumps from $x$ to
$1-x$ at rate $1$. This is done in two steps. We first prove in Lemma
\ref{tight1} that the sequence of processes $\{X^N_t : t\ge 0 \}$ is a
tight family. Then, we characterize in Lemma \ref{uniq1} all limit
points by showing that they solve a martingale problem. Both
statements rely on a replacement result, stated in Lemma \ref{r1a} and
Lemma \ref{rep1}, which allows the substitution of a function by its
conditional expectation.

Conditions (\ref{cc1}) and (\ref{cc2b}) imply that
\begin{equation}
\label{cc3}
\lim_{N\to\infty} \sup_{\eta\in W_N} \E_{\eta}\Big[ 
\int_0^{T_{\bs \xi}(\ms E)} \big\{ R^{\ms W}_N(\eta^{\ms E_N}_s) + 1
\big\} \, {\bs 1}\big\{ \eta^{\ms E_N}_s\in W_N \big\} \, ds \Big] 
\,=\, 0\,,
\end{equation}
where $T_{\bs \xi}(\ms E)=T^N_{\bs \xi}(\ms E):=\inf \{ t\ge 0 :
\eta^{\ms E_N}_t = \xi^N \}$. As a consequence of (\ref{cc3}), we get
the following lemma.

\begin{lemma}
\label{r1a}
For every $t>0$,
$$
\lim_{N\to +\infty}\sup_{\eta\in \ms E_N} \Big|\,\E_{\eta}\Big[ 
\int_0^t  \big( R^{\ms W}_N(\eta^{\ms E_N}_s) - 1 \big) 
\,{\bs 1}\big\{ \eta^{\ms E_N}_s\in W_N \big\}\,ds \Big]\Big|\,=\,0\,.
$$
\end{lemma}

\begin{proof}
Recall the notation introduced in Subsection \ref{ss54}.  Let $g: \ms
E_N\to \bb R$ be given by $g(\eta) = R^{\ms W}_N(\eta) {\bs 1} \{ \eta
\in W_N \}$ so that the expectation of $g$ with respect to $\mu^{\ms
  E}_N$ is equal to $\mu_N (W_N)/\mu_N(\ms E_N)$ in view of
\eqref{f00}. Consider the partition $\pi = \{W_N, B^c_N\}$ of $\ms
E_N$ and note that the conditional expectation $\<g | \pi\>_{\mu^{\ms
    E}_N} = \mb 1\{\eta \in W_N\}$. Since $g$ is integrable with
respect to $\mu^{\ms E}_N$, the statement follows from \eqref{cc3} and
Corollary \ref{s09} applied to the process $\{\eta^{\ms E_N}_t : t\ge
0 \}$.
\end{proof}

We use this lemma to show tightness for the sequence $\{X^N_t: t\ge 0\}$.

\begin{lemma}
\label{tight1}
Fix an arbitrary point ${\bs \eta}=(\eta^N : N\ge 1)$ in $\ms W$ and
$z\in \{0,1\}$. For each $N\ge 1$, denote by $\bb Q_N$ the law of $\{
X^N_t : t\ge 0\}$ under $\prob_{(\eta^N,z)}$. Then the sequence of
laws $({\bb Q}_N : N\ge 1)$ is tight.
\end{lemma}

\begin{proof}
For each $T>0$, let $\mf T_T$ denote the set of all stopping times
bounded by $T$. By Aldous criterion (see Theorem 16.10 in \cite{b}) we
just need to show that
\begin{equation}
\label{aldous0}
\lim_{\delta\downarrow 0}\limsup_{N\to\infty}\sup_{\theta\le\delta}
\sup_{\tau\in\mf T_T}\prob_{(\eta^N,z)} \big[ \,
| X^{N}_{\tau + \theta} - X^{N}_{\tau}| > \epsilon \,\big] \,=\, 0
\end{equation}
for every $\epsilon>0$ and $T>0$. Denote by ${\bf L}^{{\ms E}}_N$ the
generator of $(\eta^{\ms E_N}_{t}, X^{N}_{t})_{t\ge 0}$ and by ${\bf
  p}:{\ms E}_N\times \{0,1\}\to \{0,1\}$ the projection on the second
coordinate. Consider the martingale
\begin{equation*}
M^N_t \;:=\; X^{N}_t \;-\; X^{N}_0 \;-\; \int_0^t 
({\bf L}^{{\ms E}}_N {\bf p })(\eta^{\ms E_N}_{s}, X^{N}_{s})\,ds\, .
\end{equation*}
It is therefore enough to show that \eqref{aldous0} holds with
$X^{N}_{\tau+\theta}\,- \,X^{N}_{\tau}$ replaced by $M^N_{\tau+\theta}
\,-\, M^N_{\tau}$ and by $\int_\tau^{\tau+\theta}({\bf L}^{{\ms E}}_N
{\bf p})(\eta^{\ms E_N}_{s},X^{N}_{s})\,ds$.

Consider the integral term. By Chebychev inequality and by the strong
Markov property, we need to prove that
\begin{equation*}
\lim_{\delta\downarrow 0}\, \limsup_{N\to\infty}\, \sup_{\theta\le\delta}
\, \sup_{(\eta,x) \in {\ms E}_N\times \{0,1\}} \E_{(\eta,x)} \Big[ \,
\int_0^\theta \Big| \, ({\bf L}^{{\ms E}}_N 
{\bf p})(\eta^{\ms E_N}_{s},X^{N}_{s})\, \Big | \,ds  \,\Big] 
\,=\, 0\,,
\end{equation*}
where $\E_{(\eta,x)}$ stands for the expectation with respect to
$\prob_{(\eta,x)}$. A simple computation provides
$$
({\bf L}^{{\ms E}}_N {\bf p})(\eta,x)\,=\,\{\breve x - x \}\, 
R^{\ms W}_N(\eta) \,{\bs 1}\big\{\eta \in W_N \big\}\,.
$$
The proof is thus reduced to the claim
\begin{equation}
\label{x1}
\lim_{\delta\downarrow 0}\, \limsup_{N\to\infty}\,
\sup_{ \eta \in \ms E_N} \, \E_{\eta} \Big[  
\int_0^\delta R^{\ms W}_N(\eta^{\ms E_N}_s) 
\,{\bs 1}\big\{\eta^{\ms E_N}_s\in W_N\big\} \, ds \Big] \,=\, 0\,.
\end{equation}
Since the expectation above is less than or equal to
$$
\Big| \E_{\eta} \Big[  \int_0^\delta \big( 
R^{\ms W}_N(\eta^{\ms E_N}_s) - 1 \big) 
\,{\bs 1}\big\{\eta^{\ms E_N}_s\in W_N\big\} \, ds \Big]\Big| 
\,+\, \delta\,,
$$
the limit (\ref{x1}) follows from Lemma \ref{r1a}.

We now turn to the martingale part $\{M^N_t : t\ge 0 \}$, whose
quadratic variation is given by
\begin{eqnarray*}
\langle M\rangle^N_{t} &=& \int_0^{t}\big\{{\bf L}^{{\ms E}}_N 
({\bf p}^2) - 2 {\bf p} {\bf L}^{{\ms E}}_N {\bf p}\big\} 
(\eta^{\ms E_N}_{s},X^{N}_{s})\,ds \\
&=& \int_0^t R^{\ms W}_N(\eta^{\ms E_N}_s)\,{\bs 1}
\big\{\eta^{\ms E_N}_s\in W_N\big\}\,ds\,.
\end{eqnarray*}
By Chebychev inequality
$$
\prob_{(\eta^N,z)} \big[ \, | M^N_{\tau+\theta} 
- M^N_{\tau} |  > \epsilon \big] \;\le\; 
\frac 1{\epsilon^2} \, \E_{(\eta^N,z)}
\big[\, \<M\>^N_{\tau+\theta} \,-\, \<M\>^N_\tau \,\big]\,. 
$$
Finally, by the explicit formula for the quadratic variation and by
the strong Markov property, the right hand side above is less than or
equal to
$$
\frac {1}{\epsilon^2} \, \sup_{\eta\in \ms E_N}
\E_{\eta} \Big[ \int_0^{\delta} R^{\ms W}_N(\eta^{\ms E_N}_s)
\,{\bs 1}\big\{\eta^{\ms E_N}_s\in W_N\big\}\,ds\, \Big]\,.
$$
It remains to use (\ref{x1}).
\end{proof}

As a consequence of Lemma \ref{tight1} we obtain condition ({\bf V1})
for the triple $(\ms W, \ms W, {\bs \xi})$ with respect to the trace
process $\{\eta^{\ms E_N}_t : t\ge 0\}$.

\begin{lemma}\label{pro1}
For any point ${\bs \eta}=(\eta^N:N\ge 1)$ in $\ms W$,
$$
\lim_{N\to+\infty}\; \prob_{\eta^N}[\, T_{\bs \xi}(\ms W) 
< T_{\ms B^c}(\ms W) \,]\,=\,1\,.
$$
\end{lemma}

\begin{proof}
Fix $\eta^N$ in $W_N$, $N\ge 1$. Consider the modified uniform
modulus of continuity $\omega '_\delta : D(\bb R_+,\{0,1\})\to \bb
R_+$ given by
$$
\omega'_{\delta}(x_{\bs \cdot})\,:=\, \inf_{\{t_i\}}\; 
\max_{0\le i <r} \;\sup_{t_i\le s< t < t_{i+1}} \big| x_t - x_s \big|\,,
$$
where the first infimum is taken over all partitions $\{t_i : 0\le i
\le r\}$ of the interval $[0,1]$ such that
$$
\left\{
\begin{array}{l}
0=t_0<t_1<\cdots<t_r=T \\
t_i - t_{i-1} > \delta\,, \quad {\rm for}\;\;i=1,\dots,r\,.
\end{array}
\right.
$$
By the previous lemma (see e.g. Theorem 1.3 in Chapter 4 of
\cite{kl}),
$$
\lim_{\delta\downarrow 0}\limsup_{N\to +\infty} 
\prob_{(\eta^N,0)}\big[\, \omega'_{\delta}
(X^{N}_{\bs \cdot}) = 1 \,\big]\,=\,0\,.
$$
Therefore, since for all $\delta>0$ $\{T_{\ms B^c}(\ms W) \le \delta
\} \subset \{\omega'_{\delta}(X^{N}_{\bs \cdot}) = 1 \}$
$\prob_{(\eta^N,z)}$--almost surely,
\begin{equation}
\label{eqtt}
\lim_{\delta\downarrow 0} \liminf_{N\to +\infty} 
\prob_{\eta^N}\big[\, T_{\ms B^c}(\ms W) > \delta \,\big]\,=\,1\,.
\end{equation}

On the other hand, by (\ref{cc2b}), we have
\begin{equation}
\label{eqt}
\lim_{N\to+\infty}\prob_{\eta^N}[\, T_{\bs \xi}(\ms W) 
> \delta \,]\,=\,0
\end{equation}
for any $\delta>0$. The desired result follows from (\ref{eqt}) and
(\ref{eqtt}).
\end{proof}

Actually, since $ \{ T_{\bs \xi}(\ms W) < T_{\ms B^c}(\ms W)\} \,
\subseteq\, \{ T_{\bs \xi} < T_{\ms B^c}\} \,, $ Lemma \ref{pro1}
proves condition ({\bf V1}) for the triple $(\ms W, \ms B, \bs \xi)$
with respect to the process $\{\eta^N_t ; t\ge 0\}$: For any point
${\bs \eta}=(\eta^N:N\ge 1)$ in $\ms W$,
$$
\lim_{N\to+\infty}\; \prob_{\eta^N}[\, T_{\bs \xi} 
< T_{\ms B^c} \,]\,=\,1\,.
$$

We now consider the trace of $\{ (\eta^{\ms E_N}_t,X^{N}_{t}):t\ge 0
\}$ on $W_N\times \{0,1\}$, denoted by $\{(\eta^{W_N}_t, X^{W_N}_t) :
t\ge 0\}$. As we shall see in Section \ref{sec03}, since $\{
(\eta^{\ms E_N}_t,X^{N}_{t}):t\ge 0 \}$ is positive recurrent, the
trace process $\{(\eta^{W_N}_t, X^{W_N}_t) : t\ge 0\}$ is positive
recurrent as well. Moreover, the invariant probability measure for the
trace process, denoted by ${\mf m}^{\ms W}_N$, coincides with ${\mf
  m}_N$ conditioned to $W_N\times \{0,1\}$:
$$
{\mf m}^{\ms W}_N(\eta,x) \,:=\, (1/2) \mu^{\ms W}_N(\eta) 
\,, \quad \textrm{for $(\eta,x)\in W_N\times\{0,1\}$}\,. 
$$
The marginal process $\{\eta^{W_N}_t : t\ge 0\}$ corresponds to the
trace of $\{\eta^{\ms E_N}_t : t\ge 0\}$ on $W_N$.

Let ${\bf L}^{\ms W}_N$ denote the Markov generator of
$\{(\eta^{W_N}_t, X^{W_N}_t) : t\ge 0\}$. Define, in addition, the
Markov generator $\mc L$ as
\begin{equation}
\label{ll}
\mc LF(x) \,:=\, F(\breve x) - F(x) \,,\quad x\in \{0,1\}\,,
\end{equation}
for every $F:\{0,1\}\to \bb R$. For each $N\ge 1$, let ${\bf p}={\bf
  p}_N$ be the projection function on the second coordinate ${\bf
  p}:W_N \times\{0,1\}\to \{0,1\}$. If ${\bf R}^{\ms W}_N\big( {\bs
  \cdot} , {\bs\cdot} \big)$ stands for the transition rates of
$\{(\eta^{W_N}_t, X^{W_N}_t) : t\ge 0\}$, we have that
$$
{\bf L}^{\ms W}_N(F\circ {\bf p})(\eta,x)\,=\, 
\big\{ F(\breve x) - F(x) \big\} \sum_{\xi\in W_N} 
{\bf R}^{\ms W}_N\big(\, (\eta,x) , (\xi,\breve x) \,\big) 
$$
for any $(\eta,x)\in W_N\times\{0,1\}.$ By applying Corollary \ref{rf}
to the Markov process $\{(\eta^{\ms E_N}_t,X^{N}_t) : t\ge 0\}$ and
its trace on $W_N\times \{0,1\}$, we get that
$$
\sum_{\xi\in W_N}{\bf R}^{\ms W}_N\big(\, (\eta,x) , 
(\xi,\breve x) \,\big)\,=\, R^{\ms W}_N(\eta)
$$
for all $(\eta,x)\in W_N\times\{0,1\}$. Therefore, by \eqref{f00}, the
conditional expectation of ${\bf L}^{\ms W}_N(F\circ {\bf p})$, under
${\mf m}^{\ms W}_N$, given the $\sigma$-field generated by the
partition
\begin{equation} 
\label{parw}
W_N\times\{0,1\}=(W_N\times\{0\})\cup (W_N\times\{1\})\,, 
\end{equation}
is $(\mc L F)\circ {\bf p}$\,. Therefore, applying Corollary \ref{s10}
to the trace process $\{(\eta^{W_N}_t, X^{W_N}_t) : t\ge 0\}$,
the function ${\bf L}^{\ms W}_N(F\circ {\bf p})$ and the partition
(\ref{parw}), we obtain the following replacement lemma.

\begin{lemma}
\label{rep1}
For every $x\in \{0,1\}$, function $F:\{0,1\}\to\bb R$ and time $t>0$,
\begin{equation*}
\lim_{N\to \infty} \sup_{\eta\in W_N} \left|\, 
\E_{(\eta,x)} \Big[\, \int_0^t \Big\{ {\bf L}^{\ms W}_N
(F\circ {\bf p})(\eta^{W_N}_s,X^{W_N}_s) \,-\, 
{\mc L}F(X^{W_N}_s) \Big\} \,ds \,\Big]\,\right|\,=\, 0 \,.
\end{equation*}
\end{lemma}

\begin{proof}
Recall that the conditional expectation of ${\bf L}^{\ms W}_N(F\circ
{\bf p})$ is $(\mc L F)\circ {\bf p}$\,. Since
$$
|{\bf L}^{\ms W}_N(F\circ {\bf p})(\eta,x) \,-\, {\mc L}F(x)| 
\,\le\, ( R^{\ms W}_N(\eta) + 1) \max\{ |F(0)|, |F(1)|\}\,,
$$
in view of Corollary \ref{s09}, to prove the lemma we just need to
check that for any $x\in \{0,1\}$,
\begin{equation}
\label{tn0}
\lim_{N\to\infty} \sup_{\eta\in W_N} \E_{(\eta,x)}\Big[ 
\int_0^{{\bf T}^{\ms W}_{({\bs \xi},x)}} ( R^{\ms W}_N(\eta^{W_N}_s) +
1 ) \, {\bs 1}\{ X^{W_N}_s = x\} \, ds \Big] \,=\, 0\,,
\end{equation}
where, for each $N\ge 1$,
$$
{\bf T}^{\ms W}_{({\bs \xi},x)} \,=\, 
{\bf T}^{\ms W}_{({\bs \xi},x)}(N) \,:=\, 
\inf \{ t\ge 0 : (\eta^{W_N}_t,X^{W_N}_t) = (\xi^N,x)\}\,.
$$

Fix an arbitrary $x\in \{0,1\}$. It follows from conditions
(\ref{cc1}) and (\ref{cc2b}) that
\begin{equation}
\label{cc3w}
\lim_{N\to\infty} \sup_{\eta\in W_N} \E_{\eta}\Big[ 
\int_0^{T_{\bs \xi}(\ms W)} ( R^{\ms W}_N(\eta^{W_N}_s) + 1 ) \, ds 
\Big] \,=\, 0\,.
\end{equation}
To keep notation simple, let us denote
$$
{\bf T}_N \,:=\, \int_0^{{\bf T}^{\ms W}_{({\bs \xi},x)}} 
( R^{\ms W}_N(\eta^{W_N}_s) + 1 ) \, {\bs 1}\{ X^{W_N}_s = x\} \, ds\; .
$$
Since $\{T_{\bs \xi}(\ms W) < T_{\ms B^c}(\ms W) \} \subseteq \{
T_{\bs \xi}(\ms W) = {\bf T}^{\ms W}_{({\bs \xi},x)} \} $, by Lemma
\ref{pro1}, (\ref{cc3w}) and Chebychev inequality, for every $t>0$,
\begin{equation}
\label{a1}
\lim_{N\to\infty} \sup_{\eta\in W_N} {\bf P}_{(\eta,x)}[ 
\, {\bf T}_N > t \,] \,=\,0\,.
\end{equation}
By the strong Markov property, (\ref{a1}) and the arguments presented
in the proof of Lemma \ref{lui}, 
$$
\lim_{A\to \infty} \limsup_{N\to \infty} 
\sup_{\eta \in W_N} {\bf E}_{(\eta,x)} [\, {\bf T}_N \, 
{\bs 1}\{ {\bf T}_N > A \} \,]\, =\, 0 \, .
$$
Hence, by Lemma \ref{pro1},
\begin{equation}
\label{tn1}
\limsup_{N\to \infty} \sup_{\eta \in W_N} {\bf E}_{(\eta,x)} 
[\, {\bf T}_N \, {\bs 1}\{ T_{\bs \xi}(\ms W) \ge 
T_{\ms B^c}(\ms W) \} \,]\,=\,0 \,.
\end{equation}
On the other hand,
$$
{\bf E}_{(\eta,x)} [\, {\bf T}_N \, {\bs 1}\{ T_{\bs \xi}(\ms W) 
< T_{\ms B^c}(\ms W) \} \,] \,\le\, \E_{\eta}\Big[ 
\int_0^{T_{\bs \xi}(\ms W)} ( R^{\ms W}_N(\eta^{W_N}_s) + 1 ) \, ds 
\Big]\,.
$$
Therefore, (\ref{tn0}) follows from this estimate, (\ref{cc3w}) and
(\ref{tn1}).
\end{proof}

We now prove the convergence in law of $\{X^{W_N}_t : t\ge 0\}$ as
$N\uparrow \infty$. Fix an arbitrary point ${\bs \eta}=(\eta^N : N\ge
1)$ in $\ms W$. For each $N\ge 1$, denote by $\bb P_N$ the law of
$\{X^{W_N}_t : t\ge 0\}$ under $\prob_{(\eta^N,0)}$. Following the
same argument presented in the proof of Lemma \ref{tight1} we can show
that $(\bb P_N : N\ge 1)$ is tight.

The uniqueness of limit points for this sequence is established as
follows. Assume without loss of generality that $\bb P_N \to \bb P$,
as $N\to \infty$, for some probability measure $\bb P$ on $D(\bb
R_+,\{0,1\})$. For $t\ge 0$, let $X_t$ denote the time-projection
$X_t: D(\bb R_+,\{0,1\})\to \{0,1\}$. We shall prove in the following
lemma that $\bb P$ solves the martingale problem associated to the
generator $\mc L$ defined in (\ref{ll}). It is well known that this
property together with the distribution of $X_0$, characterize the
measure $\bb P$.

\begin{lemma}
\label{uniq1}
Under $\bb P$, $X_0=0$ a.s. and
$$
M^F_t \,=\, F(X_t) - F(X_0) - \int_0^t \mc LF(X_s)\,ds\,,
\quad \textrm{for $t\ge 0$}\,,
$$
is a martingale for any function $F:\{0,1\}\to \bb R$.
\end{lemma}

\begin{proof}
The first claim is trivial. For the last one, fix $0\le s< t$, a
function $F: \{0,1\} \mapsto \bb R$ and a bounded function $U:D(\bb
R_+,\{0,1\})\mapsto \bb R$ depending only on $\{X_r : 0\le r\le s\}$
and continuous for the Skorohod topology. Denote by $\bb E$ and $\bb
E_{N}$ the expectation with respect to $\bb P$ and $\bb P_N$,
respectively. We shall prove that
\begin{equation}
\label{f11}
\bb E\,\big[M^F_t U\big] \;=\; \bb E\,\big[M^F_s U\big]\;.
\end{equation}
Recall that ${\bf L}^{\ms W}_N$ denotes the generator of $\{
(\eta^{W_N}_t,X^{W_N}_t) : t\ge 0 \}$. For $N\ge 1$, consider the
${\bf P}_{(\eta^N,0)}$-martingale $\{ M^N_t : t\ge 0\}$, defined by
\begin{equation*}
M^N_t \,=\, F(X^{W_N}_t) - F(0) - \int_0^t {\bf L}^{\ms W}_N
(F\circ {\bf p})(\eta^{W_N}_{s},X^{W_N}_{s}) \, ds\,,\quad t\ge 0\,.
\end{equation*}
Denote $U^N := U(X^{W_N}_{\bs \cdot})$. As $\{M^N_t : t\ge 0 \}$ is a
martingale,
$$
\E_{(\eta^N,0)}\big[M^N_t U^N\big]
\, = \, \E_{(\eta^N,0)} \big[M^N_s U^N\big]
$$
so that
\begin{equation*}
\E_{(\eta^N,0)} \Big[ U^N \Big\{ F(X^{W_N}_{t}) - F(X^{W_N}_{s}) -
\int_s^t {\bf L}^{\ms W}_N
(F\circ {\bf p})(\eta^{W_N}_{r},X^{W_N}_{r})\, dr \Big\} \Big ] \,=\, 0\,.
\end{equation*}
On the other hand, since $U^N$ is bounded and $\mc F_s$-measurable, it
follows from the Markov property and Lemma \ref{rep1} that
\begin{equation*}
\lim_{N\to\infty}
\E_{(\eta^N,0)} \Big [U^N \int_s^t \Big\{ {\bf L}^{\ms W}_N
(F\circ{\bf p})(\eta^{W_N}_{r},X^{W_N}_{r})  -
\mc L F(X^{W_N}_{r}) \Big\} \, dr \Big]\,=\,0\,.
\end{equation*}
Putting the last two assertions together we get
\begin{equation}\label{f08a}
\lim_{N\to \infty} \bb E_N \Big[ U \Big\{ F(X_{t}) - F(X_{s}) -
\int_s^t \mc L F(X_{r}) \, dr \Big\} \Big ] \,=\, 0\,. 
\end{equation}
Now, since $\bb P_{N}$ converges to $\bb P$, time averages of $\bb
E_{N} \big[F(X_t) U \big]$ and $\bb E_{N}\big[ F(X_s) U \big]$
converge to time averages of $\bb E\,\big[F(X_t) U\big]$ and $\bb
E\,\big[F(X_s) U\big]$, respectively. Hence, from this last
observation and (\ref{f08a}) it follows that
\begin{equation*}
\frac 1\epsilon \int_0^\epsilon dr \, \bb
E\, \Big [ U \Big\{ F(X_{t+r}) - F(X_{s+r}) -
\int_{s+r}^{t+r} \mc L F(X_s) \Big\} \Big ] \,=\, 0
\end{equation*}
for every $\epsilon >0$. It remains to let $\epsilon \downarrow 0$ and
use the right continuity of the process to deduce \eqref{f11}, which
concludes the proof of the lemma.
\end{proof}

Under $\bb P$, $\{ X_t : t\ge 0\}$ is therefore a Markov chain on
$\{0,1\}$ with generator $\mc L$ and starting at $0$. We have thus
shown that, the law of
$$
T_{{\ms B}^c}(\ms W) \,=\, \inf\big\{ t>0 : X^{W_N}_{t}  = 1 \big\} \,,
$$
under ${\bf P}_{(\eta^N,0)}$, converges to a mean-one exponential
distribution. To conclude the proof of Proposition \ref{cb}  it
remains to check \eqref{lc}. By Lemma \ref{pro1} and the convergence
in law of $T_{\ms B^c}(\ms W)$, $(\ms W,\ms W,{\bs \xi})$ is a valley
for the trace process $\{\eta^{\ms E_N}_t : t\ge 0\}$ with depth $\bs
1$. Therefore, applying item $(ii)$ of Proposition \ref{prop1} to
$(\ms W,\ms W,{\bs \xi})$ and $\{\eta^{\ms E_N}_t : t\ge 0\}$ we get
(\ref{lc}). This concludes the proof of Proposition \ref{cb}
\qed\smallskip

We are now in a position to prove Theorem \ref{teo0}. Condition {\bf
  (V1)} follows from Proposition \ref{cb} and Condition {\bf (V3)}
from (\ref{ccb}) and Chebychev inequality. Condition {\bf (V2)}
follows from {\bf (V3)} and the convergence in law of $r_N(\ms W,\ms
B^c)\, T_{\ms B^c}(\ms W)$ stated in Proposition \ref{cb}.

\subsection*{Proof of Theorem \ref{teo1}}

Next result is the main step in the proof of Theorem \ref{teo1}.

\begin{proposition}
\label{cb9} 
Recall the notation introduced in Subsection \ref{meta}.  If
conditions {\bf (C2)}, {\bf (C3)} and {\bf (H0)} are in force, then so
are {\bf (M1)} and {\bf (M2)}.
\end{proposition}

The proof of this result is divided in three lemmas. As in the proof
of Proposition \ref{cb}, without loss of generality, we may assume
that $\theta_N=1$, $\forall N\ge 1$. In this way, condition $\bf (H0)$
guarantees that, for every $x,y\in S$, $x\not = y$,
\begin{equation}
\label{c3}
\lim_{N\to\infty} r_N(\ms E^x,\ms E^y) \,=\, r(x,y)\,,
\end{equation}
and we shall prove the convergence in law of the sequence $\{X^N_t :
t\ge 0\}$, $N\ge 0$.

Clearly, conditions $\bf (C2)$ and $\bf (C3)$ imply
\begin{equation}
\label{cc4}
\lim_{N\to\infty} \sup_{\eta\in \ms E^x_N} \E_{\eta} 
\Big[ \int_0^{T_{\bs \xi_x}} \big\{ 
R^{x}_N(\eta^{\ms E_N}_s) + r_N(\ms E^x,\breve{\ms E}^x) \big\}
\, \mb 1\{\eta^{\ms E_N}_s \in \ms E^x_N\}
\,  ds \Big] \,=\, 0
\end{equation}
for any $x\in S$, where $\{\eta^{\ms E_N}_t : t\ge 0\}$ stands for
the trace process of $\{\eta_t : t\ge 0\}$ on $\ms E_N$. Let us
define $V_N: \ms E_N \mapsto \bb R$ as
$$
V_N(\eta) \,:=\, \sum_{x\in S} R^x_N(\eta) \, 
{\bs 1}\{\eta\in \ms E^x_N\}\,,\quad \eta\in {\ms E}_N\,.
$$
Let $\mu^{\ms E}_N$ be the measure $\mu_N$ conditioned to $\ms E_N$
and denote by $\hat V_N$ the $\mu^{\ms E}_N$-conditional expectation
of $V_N$ given the $\sigma$-algebra generated by the partition $\ms
E_N = \cup_{x\in S}\ms E^x_{N}$:
$$
\hat V_N(\eta)\,:=\, \sum_{x\in S} r_N(\ms E^x,\breve{\ms E}^x) 
\,{\bs 1}\{\eta\in {\ms E}^x_N\}\,,\quad \forall \eta\in {\ms E}_N\,.
$$
Since $V_N$ is integrable with respect to $\mu^{\ms E}_N$, it follows
from Corollary \ref{s09} and from (\ref{cc4}) that, for
any $t>0$,
\begin{equation}\label{s03}
\lim_{N\to\infty} \sup_{\eta\in \ms E_N} \Big\vert\, 
{\bf E}_{\eta} \Big[ \int_0^t  
\big\{ V_N - \hat V_N \big\}(\eta^{\ms E_N}_s) 
\, ds \Big] \, \Big\vert \;=\; 0\;.
\end{equation}

In order to prove $\bf (M2)$, fix some $x\in S$ and a point ${\bs
  \eta}=(\eta^N: N\ge 1)$ in $\ms E^x$. For each $N\ge 1$, denote by
$\bb P_N$ the law of $\{ X^N_t : t\ge 0\}$ under ${\bf P}_{\eta^N}$.
The convergence of the sequence $(\bb P_N : N\ge 1)$ stated in $\bf
(M2)$, follows from tightness and uniqueness of limit points. We first
examine the tightness.

\begin{lemma}
\label{s11}
The sequence $(\bb P_N : N\ge 1)$ is tight.
\end{lemma}

\begin{proof} 
For each $T>0$, let $\mf T_T$ denote the set of all stopping times
bounded by $T$. By Aldous criterion (see Theorem 16.10 in \cite{b}) we
just need to show that
\begin{equation}
\label{aldous}
\lim_{\delta\downarrow 0}\lim_{N\to\infty}\sup_{\theta\le\delta}
\sup_{\tau\in\mf T_T}\prob_{\eta_N} \big[ \;
| X^N_{\tau + \theta} - X^N_{\tau}| > \epsilon \;\big] \;=\; 0
\end{equation}
for every $\epsilon>0$ and $T>0$.

Let $L^{\ms E}_N$ be the generator of the trace process $\{ \eta^{\ms
  E_N}_{t}:t\ge 0 \}$ and let $\{M^N_t: t\ge 0\}$ be the martingale
defined by
\begin{equation*}
M^N_t \;=\; X^N_t \;-\; X^N_0 \;-\; \int_0^t 
L^{\ms E}_N \Psi_N(\eta^{\ms E_N}_s)\, ds\; .
\end{equation*}
To prove tightness, it is therefore enough to show that \eqref{aldous}
holds with the difference $X^N_{\tau+\theta} - X^N_{\tau}$ replaced by
$M^N_{\tau+\theta} - M^N_{\tau}$ and by $\int_\tau^{\tau+\theta}
L^{\ms E}_N X^N_s ds$.

Consider the integral term. By Chebychev inequality and by the strong
Markov property, we need to prove that
\begin{equation*}
\lim_{\delta\downarrow 0}\, \lim_{N\to\infty}\, \sup_{\theta\le\delta}
\, \sup_{\eta \in\ms E_N} \E_{\eta} \Big[ \;
\int_0^\theta \big|\,  L^{\ms E}_N \Psi_N(\eta^{\ms E_N}_s) \big |\,
ds \;\Big] \;=\; 0\;.
\end{equation*}
An elementary computation shows that
\begin{eqnarray*}
L^{\ms E}_N \Psi_N (\eta) &=& \sum_{x,y\in S}
\{y - x\} \, R^{x,y}_N(\eta)\,{\bs 1}\{\eta\in {\ms E}^x_N\}\,,
\end{eqnarray*}
for any $\eta\in\ms E_N$. Since $|L^{\ms E}_N \Psi_N| \le \kappa V_N$,
the proof is reduced to the claim
\begin{equation*}
\lim_{\delta\downarrow 0}\, \lim_{N\to\infty}\,
\sup_{\eta \in\ms E_N} \, \E_{\eta} \Big[  
\int_0^\delta V_N(\eta^{\ms E_N}_s) ds \Big] \;=\; 0\,.
\end{equation*}
The left hand side can be written as
$$
\lim_{\delta\downarrow 0}\, \lim_{N\to\infty}\,
\sup_{\eta \in\ms E_N} \, \Big\{ \E_{\eta} 
\Big[  \int_0^\delta \big\{V_N - \hat V_N \big\} 
(\eta^{\ms E_N}_s) ds \Big] \,+\,  
\E_{\eta} \Big[  \int_0^\delta \hat V_N(\eta^{\ms E_N}_s) ds 
\, \Big]\Big\}\,.
$$
The first term converges to zero as $N\uparrow\infty$, for any
$\delta>0$, by (\ref{s03}). The second term is bounded above by
$$
\lim_{\delta\downarrow 0} \lim_{N\to\infty} \delta\, 
\sum_{x\in S} \, r_N(\ms E^x_N,\breve{\ms E}^x_N)\,,
$$
which is equal to zero by (\ref{c3}).

We now turn to the martingale part, whose quadratic variation, denoted
by $\< M^N\>_{t}$, is given by 
$$
\< M^N\>_{t} \,=\, \int_0^t \big\{ L^{\ms E}_N (\Psi_N)^2(\eta^{\ms
  E_N}_s) - 2 X^N_s (L^{\ms E}_N \Psi_N)(\eta^{\ms E_N}_s) \big\}\, ds
\,, \quad t\ge 0\,.
$$
An elementary computation shows that this expression is equal to
\begin{equation*}
\sum_{x,y\in S}
\{y - x\}^2 \, R^{x,y}_N(\eta)\,{\bs 1}\{\eta\in {\ms E}^x_N\}\;.
\end{equation*}

By the explicit formula for the quadratic variation, by
Chebychev inequality and by the strong Markov property,
\begin{eqnarray*}
\prob_{\eta^N} \big[ \, \big| M^N_{\tau+\theta} 
-  M^N_{\tau} \big| \, >\, \epsilon \big] &\le& 
\frac 1{\epsilon^2} \, \E_{\eta^N}
\big[\, \<M^N\>_{\tau+\theta} \,-\, \<M^N\>_\tau \,\big] \\
&\le& \frac {\kappa^2}{\epsilon^2} \, \sup_{\eta\in \ms E_N}
\E_{\eta} \Big[ \, \int_0^\delta 
V_N(\eta^{\ms E_N}_s) \, ds\, \Big]\,.
\end{eqnarray*}
It remains to repeat the arguments presented for the integral term of
the decomposition.
\end{proof}

Now we turn to the uniqueness of limit points. Assume without loss of
generality that the sequence $\bb Q_N$ converges to a measure $\bb
P\,$. Denote by $\mf L_N$ and $\mf L$ the Markov generators on the
state space $S=\{1,\dots,\kappa\}$ given by
$$
(\mf L_N F)(x) \;=\; \sum_{y\in S\setminus\{x\}} \{F(y) - F(x)\}
r_N(\ms E^x, \ms E^y)
$$
and
$$
(\mf LF)(x) \;=\; \sum_{y\in S\setminus\{x\}} \{F(y) - F(x)\}\, r(x,y)\;.
$$
For $t\ge 0$, let $X_t$ denote the projection $D(\bb R_+,S)\mapsto S$.
The probability $\bb P$ is completely determined by the properties
stated in the following lemma.

\begin{lemma}
\label{s02}
Under $\bb P\,$, $X_0 = x$ and
\begin{equation}
\label{f04}
M_t \;=\; F(X_t) \;-\; F(X_0) \;-\; \int_0^t \mf L F (X_s) \, ds
\end{equation}
is a martingale for any function $F: S \mapsto \bb R$.
\end{lemma}

The proof of this lemma follows closely the one of Lemma \ref{uniq1}.
It suffices, in particular, to show the following replacement lemma.
Let $L^{\ms E}_N$ stand for the generator of $\{\eta^{\ms E_N}_t :
t\ge 0\}$.

\begin{lemma}
For any $t>0$,
$$
\lim_{N\to\infty} \sup_{\eta\in \ms E_N}{\bf E}_{\eta} 
\Big [ \int_0^t \big\{ L^{\ms E}_N
(F\circ\Psi_N) - ({\mf L}F) \circ \Psi_N \big\}
(\eta^{\ms E_N}_s) \, ds \Big ] \,=\, 0\,.
$$
\end{lemma}

\begin{proof}
First, by condition $\bf (H0)$, we have that
$$
\lim_{N\to\infty} \sup_{\eta\in \ms E_N}
{\bf E}_{\eta} \Big [ \int_0^t \big\{ (\mf L_N F)(X^N_s) -
(\mf L F)(X^N_s) \big\} \, ds \Big]\,=\,0\,.
$$
It remains to prove that
\begin{equation}
\label{r1}
\lim_{N\to\infty} \sup_{\eta\in \ms E_N}{\bf E}_{\eta} 
\Big [ \int_0^t \big\{ L^{\ms E}_N
(F\circ\Psi_N) - ({\mf L}_NF) \circ \Psi_N \big\}
(\eta^{\ms E_N}_s) \, ds \Big ] \,=\, 0\,.
\end{equation}
The $\mu^{\ms E}_N$-conditional expectation of $L^{\ms E}_N
(F\circ\Psi_N)$ given the $\sigma$-algebra generated by the partition
$\ms E_N=\cup_{x\in S} \ms E_N^x$ is $({\mf L}_NF) \circ \Psi_N$.  The
expectation of $|\, L^{\ms E}_N (F\circ\Psi_N)\, |$ with
respect to $\mu^{\ms E}_N$ is bounded by $C(F) \sum_{x\in S} r_N(\ms
E^x, \ms E^y)$ for some finite constant $C(F)$, depending only on $F$,
and, for any $\eta\in \ms E_N$, $|L^{\mc E}_N (F\circ\Psi_N)(\eta) -
(\mf L_N F) \circ \Psi_N(\eta)|$ is bounded above by
$$
2\max_{z\in S} |F(z)| \, \sum_{x\in S} 
\big\{ R^x_N(\eta) -  r_N(\ms E^x, \ms E^y) \big\}
\,{\bs 1}\{\eta\in {\ms E^x_N}\}\,.
$$
By Corollary \ref{s09}, applied to $g=L^{\ms E}_N (F\circ\Psi_N)$ and
by (\ref{cc4}), (\ref{r1}) holds, which concludes the proof of the
lemma.
\end{proof}

This concludes the proof of condition ({\bf M2}). Condition ({\bf M1})
follows from Proposition \ref{cb} with $\ms W = \ms E^x$, $\ms B = \ms
E^x \cup \mb \Delta$, which concludes the proof of Proposition
\ref{cb9}. \qed\smallskip

For Theorem \ref{teo1}, it remains to check condition ({\bf M3}) for
non-absorbing states. This follows from Proposition \ref{prom3} since
condition ({\bf M2}) has already been deduced.

\subsection*{Proof of Theorem  \ref{teo0d}.}

We assume in this subsection that the process is reversible and adopt
all notation introduced in Section \ref{sec1}. The proof of Theorem
\ref{teo0d} relies on the following result which states the important
fact that, under condition \eqref{sufcond2}, the capacity between $\ms
W$ and $\ms B^c$ is asymptotically equivalent to the capacity between
any point $\bs \zeta$ of $\ms W$ and $\ms B^c$.

\begin{proposition}
\label{teo0r} 
Consider two sequences of sets $\ms W$ and $\ms B$ satisfying
\eqref{val1}. Assume that condition (\ref{sufcond2}) holds for some
point $\bs \xi=(\xi^N : N\ge 1)$ in $\ms W$. Then, the assertions of
Proposition \ref{cb} are in force. Moreover, for every point
$\bs\zeta=(\zeta^N:N\ge 1)$ in $\ms W$,
\begin{equation}
\label{ce}
 \lim_{N\to \infty} \frac{ \Cap_N(\ms W, \ms B^c) }
{ \Cap_N({\bs \zeta} ,  \ms B^c ) } \;=\; 1\;,
\end{equation}
and
\begin{equation}
\label{vze}
\lim_{N\to \infty}\inf_{\eta\in W_N} 
{\bf P}_{\eta}[\, T_{\bs\zeta} < T_{\ms B^c} \,] \;=\; 1 \;.
\end{equation}
\end{proposition} 

\begin{proof}
We have shown just before the statement of Theorem \ref{teo0d} that
conditions \eqref{cc1}, \eqref{cc2} follow from (\ref{sufcond2}). 
In particular, the assertions of Proposition \ref{cb} hold.

Fix an arbitrary point ${\bs \zeta} = (\zeta^N : N\ge 1)$ in $\ms
W$. By \eqref{g02} applied to $\{\eta^N\}$, $g= \mb 1\{\ms W\}$,
$\{\xi^N\}$, and to $\{\xi^N\}$, $g= \mb 1\{\ms W\}$, $\{\zeta^N\}$,
for any ${\bs \eta}=(\eta^N : N\ge 1)$ in $\ms W$,
\begin{eqnarray*}
{\bf E}_{\eta^N} [\, T_{\bs \zeta}(\ms W) \,] &\le & 
{\bf E}_{\eta^N} [\, T_{\bs \xi}(\ms W) \,] \,+\, 
{\bf E}_{\xi^N} [\, T_{\bs \zeta}(\ms W) \,] \\
&\le&  \frac{\mu_N(\ms W)}{\Cap_N({\bs \eta},{\bs \xi})} 
{\bs 1}\{ \eta^N\not = \xi^N\} \,+\, 
\frac {\mu_N(\ms W)}{\Cap_N({\bs \zeta},{\bs \xi})} 
{\bs 1}\{ \zeta^N\not = \xi^N\} \\
&\le& \frac{ 2 \,\mu_N(\ms W) }{ \Cap_N({\bs \xi}) }\; \cdot
\end{eqnarray*}
 From this estimate, identity (\ref{for3}) and hypothesis
(\ref{sufcond2}), it follows that
\begin{equation*}
\lim_{N\to\infty} r_N(\ms W,\ms B^c) \, {\bf E}_{\eta^N} 
[\, T_{\bs \zeta}(\ms W) \,] \;=\; 0\,,
\end{equation*}
which, by (\ref{aa1}) in Proposition \ref{cb}, implies that
\begin{equation}
\label{cb10}
\lim_{N\to\infty} \frac {{\bf E}_{\eta^N} [\, T_{\bs \zeta}(\ms W)
  \,]} {{\bf E}_{\eta^N}[T_{\ms B^c}(\ms W)]} \;=\; 0 \;.
\end{equation}
This limit corresponds to item (i) of Proposition \ref{prop1} with the
point $\bs \zeta$ instead of $\bs \xi$. Item (ii) of Proposition
\ref{prop1} follows from the last two assertions of Proposition
\ref{cb}. From items (i) and (ii) we conclude that $(\ms W,\ms W,{\bs
  \zeta})$ is a valley for the trace process $\{\eta^{\ms E_N}_t :
t\ge 0\}$. Hence,
$$
\lim_{N\to\infty}{\bf P}_{\eta^N}[\,T_{\bs \zeta}(\ms E) 
< T_{\ms B^c}(\ms E) \,]\,=\,1\,,
$$
which implies condition {\bf (V1)} for the triple $(\ms W,\ms B^c,{\bs
  \zeta})$ because $\{T_{\bs \zeta}(\ms E) < T_{\ms B^c}(\ms E)\}
\subseteq \{T_{\bs \zeta} < T_{\ms B^c}\}$, proving (\ref{vze}).

By Proposition \ref{bovier} with $A=\{\eta\}$, $B = \ms B^c$ and
$g=\mb 1\{\ms W\}$, and by identity (\ref{for3}), the limit
(\ref{aa1}) can be re-written as
$$
\lim_{N\to \infty} \frac{\langle {\bs 1}\{W_N\} , 
f_N({\bs \eta}, \ms B^c) \rangle_{\mu_N}\, 
\Cap_N(\ms W, \ms B^c)}{\mu_N(\ms W) \, \Cap_N({\bs \eta}, \ms B^c)} 
\;=\; 1\,.
$$
Replace $\bs \eta$ by $\bs \zeta$ in this formula.  By (\ref{vze}),
the infimum of $f_N( {\bs \zeta}, \mc B^c ) $ over $W_N$ converges to
$1$ as $N\uparrow \infty$. Therefore, (\ref{ce}) follows from this
observation and the previous identity.
\end{proof}

We are now in a position to prove Theorem \ref{teo0d}. We first show
that $(\ms W, \ms B, \bs\xi)$ is a valley of depth $\theta_N = r_N(\ms
W, \ms B^c)^{-1}=\mu_N(\ms W)/\Cap_N(\ms W, \ms B^c)$. Identity
(\ref{for3}) and Proposition \ref{bovier} show that
\begin{equation*}
r_N(\ms W,\ms B^c) \, {\bf E}_{\eta^N}[T_{\ms B^c}({\bs \Delta}) ] 
\,=\, \frac{\langle {\bs 1}\{\Delta_N\} , 
f_N({\bs \eta} ,\ms B^c)\rangle_{\mu_N} \, 
\Cap_N(\ms W,\ms B^c)}{ \mu_N(\ms W) \,\Cap_N({\bs \eta}, \ms B^c)} \;
\cdot 
\end{equation*}
By Proposition \ref{teo0r}, \eqref{ce} holds. Since $f_N({\bs \eta}
,\ms B^c)$ is bounded by one, (\ref{ce}) along with hypothesis
(\ref{sufcond2b}) proves (\ref{ccb}). Since \eqref{cc1} and
\eqref{cc2} follow from (\ref{sufcond2}), all the hypotheses of
Theorem \ref{teo0} are fulfilled. Therefore, $(\ms W, \ms B, \bs\xi)$
is a valley of depth $\theta_N = r_N(\ms W, \ms B^c)^{-1}=\mu_N(\ms
W)/\Cap_N(\ms W, \ms B^c)$. Last identity follows from Lemma
\ref{g03}.

Fix now a point $\bs \zeta$ in $\ms W$. To prove that $(\ms W, \ms B,
\bs\zeta)$ is a valley, we check conditions (i)--(iii) of Proposition
\ref{prop1}. Property (i) has been proved in \eqref{cb10}. Since $(\ms
W, \ms B, \bs\xi)$ is a valley, conditions (ii) and (iii) are in force
due to the first part of Proposition \ref{prop1}. Hence, by the second
part of this proposition, $(\ms W, \ms B, \bs\zeta)$ is a valley of
depth ${\bf E}_{\zeta^N}[T_{\ms B^c}({\ms W}) ]$. Finally, since $(\ms
W, \ms B, \bs\xi)$ is a valley, by the first part of this proposition,
$\theta_N$ and ${\bf E}_{\zeta^N}[T_{\ms B^c}({\ms W}) ]$ are
asymptotically equivalent sequences.

\subsection*{Proof of Theorem \ref{teo1r}}

We need to check that all assumptions of Theorem \ref{teo1} are
satisfied. As in the proof of Theorem \ref{teo0d}, conditions ({\bf
  C2}), ({\bf C3}) follow from assumption ({\bf H1}). It remains to
show that ({\bf C1}) is fulfilled for all non-absorbing states.  Fix
such a state $x$. It is enough to prove that
\begin{equation}
\label{cb6}
\limsup_{N \to \infty} \sup_{\eta\in \ms E^{x}_N} \frac{1}{\theta_N} 
{\bf E}_{\eta} \Big[\,  T_{\breve{\ms E}^{x}}({\bf \Delta}) \,\Big] 
\,=\, 0\,.
\end{equation}
By Proposition \ref{bovier} and since $f_{\bs\eta \breve{\ms E}^{x}}$
is bounded by $1$, the expectation is less than or equal to $\mu_N(\bs
\Delta)/ \Cap (\eta, \breve{\ms E}^{x})$. By \eqref{ce}, we may
replace asymptotically $\bs \eta$ by $\ms E^x$ in the previous
capacity. By Lemma \ref{g03}, $\Cap (\ms E^x, \breve{\ms E}^{x})$ 
is equal to $\mu_N(\ms E^x) r_N(\ms E^x,\breve{\ms E}^{x})$. In
conclusion, we have shown that
\begin{equation}
\label{cb5}
\limsup_{N \to \infty} \sup_{\eta\in \ms E^{x}_N}
\frac{1}{\theta_N} {\bf E}_{\eta} \Big[\,  
T_{\breve{\ms E}^{x}}({\bf \Delta}) \,\Big] \;\le\;
\limsup_{N \to \infty} \frac 1{\theta_N \, r_N(\ms E^x, \breve{\ms E}^x)}
\, \frac{ \mu_N(\bs \Delta) }{ \mu_N(\ms E^{x}) }\; \cdot
\end{equation}
Since $x$ is a non-absorbing point, by assumptions ({\bf H0}), ({\bf
  H2}), the right hand side is equal to $0$. This concludes the proof.

\subsection*{Proof of Remark \ref{cb2}}

We need to show that ({\bf H2}) holds for non-absorbing states and
that ({\bf M3}) holds for absorbing states. Clearly, ({\bf H2})
follows from ({\bf H2'}) for non-absorbing states. On the other hand,
by Proposition \ref{cb9}, ({\bf M2}) is fulfilled. Hence, by Lemma
\ref{cb3}, ({\bf M3}) for absorbing (and non-absorbing) states is a
consequence of \eqref{cb6}. By \eqref{cb5}, assumption ({\bf H2'})
implies \eqref{cb6}, which concludes the proof.

\section{Continuous time Markov chains}
\label{sec03}

We state in this section several properties of continuous time Markov
chains used throughout the article. We start assuming that the holding
rates are strictly positive and finite and that the jump chain
associated is irreducible and recurrent. We add assumptions as we
progress. At the end, we consider the case of positive recurrent,
reversible Markov chains whose holding times belong to $L^1(\mu)$,
where $\mu$ is the unique invariant probability measure.

Consider a countable set $E$ and a matrix $R : E\times E \to \bb R$
such that $R(\eta, \xi)\ge 0$, $\eta\not = \xi$, $-\infty < R(\eta,
\eta) <0$, $\sum_{\xi\not = \eta} R(\eta,\xi)=0$, $\eta\in E$. Let
$\lambda(\eta) = - R(\eta,\eta)$. Since $\lambda(\eta)$ is finite and
strictly positive, we may define the transition probabilities $\{
p(\eta,\xi) : \eta,\xi\in E \}$ as
\begin{equation}
\label{t01}
p(\eta,\xi) \;=\; \frac 1{\lambda(\eta)} \, R(\eta,\xi)
\quad \textrm{for $\eta\not = \xi$}\;,
\end{equation}
and $p(\eta,\eta)=0$ for $\eta\in E$. We assume throughout this
section that $\{ p(\eta,\xi) : \eta,\xi\in E \}$ are the transition
probabilities of a irreducible and recurrent discrete time Markov
chain.

We claim that there exists a unique stochastic semigroup $\{p_t : t\ge
0\}$ on $E$ satisfying
\begin{equation}
\label{sg}
\lim_{t\downarrow 0} \frac{ p_t(\eta,\xi) - p_0(\eta,\xi)}{t}
\;=\; R(\eta,\xi)\quad \textrm{and} \quad p_0(\eta,\xi)
\;=\;\delta_{\eta,\xi}
\end{equation}
for every $\eta$, $\xi \in E$, where $\delta_{\eta,\xi}$ is the delta
of Kroenecker. To prove the existence, we construct a Markov process
$\{\eta_t : t\ge 0\}$ on $E$ whose Markov semigroup satisfies
\eqref{sg}. We shall use this construction in some of the proofs
below.

Let $Y=\{Y_n : n\ge 0\}$ be an irreducible, recurrent, $E$-valued
discrete time Markov chain with transition probabilities
$\{p(\eta,\xi) : \eta,\xi\in E\}$ given by \eqref{t01}.  Let $(e_n :
n\ge 0)$ be a sequence of i.i.d.\!  mean one exponential random
variables, independent of $Y$. We associate to every sample path of
$Y$ the sequence of random times $T=(T_n: n\ge 0)$ given by
$$
T_n \;=\; \frac{e_n}{\lambda(Y_n)} \; \cdot
$$
Since $Y$ is recurrent, $\sum_{i\ge 0} T_i=\infty$ a.s. In particular,
the time-change
\begin{equation}
\label{alpha}
\alpha(t) \;=\; \min\{ n\ge 0 : \sum_{i=0}^nT_i>t \} 
\end{equation}
is a.s.\! finite for every $t\ge 0$ and $\eta_t = Y_{\alpha(t)}$ is
a.s.\! well defined for all $t\ge 0$. In Theorem 2.8.1 of \cite{n} it
is proved that $\{\eta_t : t\ge 0\}$ is a strong Markov process with
respect to the filtration $\{\mc F_t : t\ge 0\}$, $\mc F_t = \sigma
(\eta_s : s\le t)$. The stochastic semigroup corresponding to
$\{\eta_t : t\ge 0\}$ fulfills \eqref{sg}, as follows from the proof
of Theorem 2.8.4 in \cite{n}. On the other hand, the uniqueness of the
stochastic semigroup is a consequence of Theorem (51) in Chapter 7 of
\cite{f} along with the recurrence of the transition probabilities
$p(\cdot,\cdot)$. Note that there is no explosion since $\sum_{i\ge 0}
T_i=\infty$ a.s.

In conclusion, a collection of nonnegative numbers $\{R(\eta,\xi) :
\eta,\xi\in E\}$ satisfying the conditions listed at the beginning of
this section determines uniquely the law of a strong Markov process
$\{ \eta_t : t\ge 0\}$. We shall refer to $R(\cdot,\cdot)$,
$\lambda(\cdot)$ and $p(\cdot,\cdot)$ as the transition rates, holding
rates and jump probabilities of $\{ \eta_t : t\ge 0\}$,
respectively. The Markov chain $Y=\{Y_n : n\ge 0\}$ is called the jump
chain associated to $\{\eta_t : t\ge 0\}$.

Of course, since the jump chain $Y$ is irreducible and recurrent, so
is the corresponding Markov process $\{\eta_t : t\ge 0\}$. In
consequence, $\{\eta_t : t\ge 0\}$ has an invariant measure $\mu$
which is unique up to scalar multiples. Moreover,
\begin{equation}
\label{m}
M(\eta) \;:=\; \lambda(\eta)\mu(\eta)\;, \quad \eta\in E\;,
\end{equation}
is the invariant measure for the jump chain $Y$, also unique up to
scalar multiples. The proofs of these assertions can be found in
Sections 3.4 and 3.5 of \cite{n}. \medskip

Recall that $\tau_A: D(\bb R_+, E) \to \bb R_+$, $A\subseteq E$,
denotes the hitting time of the set $A$:
\begin{equation*}
\tau_{A}(e_{\cdot}) \;=\;\inf\{ t> 0 : e_{t}\in A \}\;.
\end{equation*}
Let $T_A:=\tau_A(\sp)$ and $T_{\eta}:=T_{\{\eta\}}$, $\eta\in
E$. Define the stopping time $\tau^+_A: D(\bb R_+, E) \to \bb R_+$ as
the first return to $A$:
\begin{equation*}
\tau^+_A(e_{\cdot}) \,=\, \inf\{ t>0 : e_t\in A, e_s\not=e_0
\;\;\textrm{for some $0< s < t$}\}\,, 
\end{equation*}
and let $T^+_A:=\tau^+_A(\sp)$, $T^+_{\eta}:=T^+_{\{\eta\}}$, $\eta\in
E$. \medskip

Let $\prob_{\eta}$, $\eta\in E$, be the probability measure under
which the jump chain $\{Y_n : n\ge 0\}$ and the Markov chain $\{\eta_t
: t\ge 0\}$ start from $\eta$. Expectation with respect to
$\prob_{\eta}$ is denoted by $\E_{\eta}$. It follows from the proof of
Theorem 3.5.3 in \cite{n} that for any $\eta\in E$
\begin{equation}
\label{im}
\mu(\xi) \;=\; {\bf E}_{\eta}\Big[ \int_0^{T^+_{\eta}} 
{\bf 1}\{\eta_s = \xi \}\,ds\Big]\;,\quad \xi\in E\;,
\end{equation}
is an invariant measure for $\{\eta_t : t\ge 0\}$.

\subsection{The trace process}
\label{trace}

We present in this subsection some elementary properties of trace
processes and we state some identities used throughout the article.

Let $h:E\to \bb R_+$ be a nonnegative function with nonempty support
$F$:
\begin{equation}
\label{fne}
F:=\{\eta \in E : h(\eta)>0\}\not=\varnothing\,.
\end{equation}
Define the additive functional $\{\mc T^h_{t} : t\ge 0\}$ as
$$
\mc T^h_{t} \,:=\, \int_0^t h(\eta_s) \,ds\;.
$$
Notice that $\mc T^h_{t}\in \bb R_+$, $\prob_{\eta}$-a.s.\! for every
$\eta\in E$ and $t\ge 0$. Denote by $\{\mc S^h_t : t\ge 0\}$ the
generalized inverse of $\mc T^h_t$:
$$
\mc S^h_t \,:=\, \sup\{s\ge 0 : \mc T^h_s \le t \}\,.
$$
Since $\{\eta_t : t\ge 0\}$ is irreducible and recurrent, $\lim_{t\to
  \infty}\mc T^h_t = \infty$, $\prob_{\eta}$-a.s.\! for every $\eta\in
E$. Therefore, the random path $\{\eta^h_t : t\ge 0\}$, given by
$\eta^h_{t} = \eta_{\mc S^h_t}$, is $\prob_{\eta}$-a.s.\! well defined
for all $\eta\in E$ and takes value in the set $F$.  We call the
process $\{\eta^h_t : t\ge 0\}$ the $h$-trace of $\{\eta_t : t\ge
0\}$. Clearly, $\{\eta^h_t : t\ge 0\}$ coincides with the trace of
$\{\eta_t : t\ge 0\}$ on $F$, defined in Section \ref{sec1}, if $h=
\mb 1 \{F\}$.

A change of variables shows that for any subset $B$ of $F$ and for any
function $f:F\to \bb R_+$,
\begin{equation}
\label{time}
\int_{0}^{\tau_B(\sp^h)}f(\eta^h_t)\, dt \;=\;
\int_{0}^{T_B}f(\eta_t)\, h(\eta_t)\, dt 
\end{equation}
$\prob_{\eta}$-a.s. for every $\eta\in E$. This identity also holds if
we replace $\tau_B(\sp^h)$, $T_B$ by $\tau^+_B (\sp^h)$, $T^+_B$,
respectively. Furthermore, for any two disjoint subsets $A$, $B$ of
$F$, it follows from the construction of the Markov chain $\{\eta^h_t
: t\ge 0\}$ that
\begin{equation*}
\prob_{\eta}\big[ \, \tau_{A}(\sp^h) < \tau_{B}(\sp^h) \,\big] 
\;=\; \prob_{\eta}\big[ \,T_{A} < T_{B} \,\big]
\end{equation*}
for all $\eta$ in $F$. This identity needs to be reformulated if we
replace the hitting times by return times. Indeed, if the process
starting from $\eta$ returns to $F$ by $\eta$, while in the original
version the process returned to $\eta$, in the trace version the
process never left $\eta$. We claim that for all $\eta\in F$ and all
disjoint subsets $A$, $B$ of $F$,
\begin{equation}
\label{h}
\prob_{\eta}\big[ \tau^+_{A}(\eta^h_{\cdot}) < \tau^+_{B}(\eta^h_{\cdot}) \big] 
\;=\; \prob_{\eta} \big[\, T^+_{A} < T^+_{B} \,\big|\,
T^+_F = T_{F\setminus\{\eta\}} \,\big]\;.
\end{equation}
To derive this identity, intersect the event $\{
\tau^+_{A}(\eta^h_{\cdot}) < \tau^+_{B}(\eta^h_{\cdot})\}$ with the
partition $\{T^+_F = T_{F\setminus\{\eta\}}\}$, $\{T^+_F =
T^+_{\eta}\}$ and apply the strong Markov property to the second piece
to get that
\begin{multline*}
\prob_{\eta}[ \tau^+_{A}(\eta^h_{\cdot}) < \tau^+_{B}(\eta^h_{\cdot}) ] 
 \; = \; \prob_{\eta} \big[ \, \tau^+_{A}(\eta^h_{\cdot}) 
< \tau^+_{B}(\eta^h_{\cdot}) \, ; \,
T^+_F = T_{F\setminus\{\eta\}} \, \big] \\ 
+\; \prob_{\eta}\big[ \, T^+_F = T^+_{\eta} \,\big] \, 
\prob_{\eta}\big[\, \tau^+_{A}(\eta^h_{\cdot}) 
< \tau^+_{B}(\eta^h_{\cdot}) \, \big]  \;.
\end{multline*}
To conclude, observe that on the set $\{T^+_F =
T_{F\setminus\{\eta\}}\}$ we may replace $\sp^h$ by $\sp$ in the event
$\{ \tau^+_{A}(\eta^h_{\cdot}) < \tau^+_{B}(\eta^h_{\cdot})\}$.
\medskip

\begin{proposition}
\label{protra}
Under $\{\prob_{\eta} : \eta\in F\}$, $\{\eta^h_t : t\ge 0\}$ is an
irreducible, recurrent, strong Markov chain with transition rates
given by
$$
R^{h}(\eta,\xi)\,=\,\frac{\lambda(\eta)}{h(\eta)}\,\prob_{\eta}\big[\,
T^+_F = T^+_{\xi} \,\big]\,,\quad \eta\,,\, \xi\in F\,,\; \eta\not=\xi\,.
$$
\end{proposition}

\begin{proof}
Recall the explicit construction of the Markov chain $\{\eta_t : t\ge
0\}$ presented in the previous subsection. To derive the $h$-trace
from this construction, we consider first the trace of the jump chain
$\{Y_n : n\ge 0\}$ on $F$. 

Define the sequence of times $\{t_n : n\ge 0\}$ as $t_0=0$, $t_1=
\inf\{ n \ge 1 : Y_n\in F \}$ and $t_{n+1}=t_n+t_1\circ \Theta_{t_n}$,
$n\ge 1$, where $\{\Theta_k : k\ge 1\}$ are the discrete time shift
operators. Let $Y^{\mf h}=\{Y^{\mf h}_n : n\ge 0\}$ be given by
$Y^{\mf h}_{n}\,=\,Y_{t_n}$. When the jump chain $\{Y_n : n\ge 0\}$
starts in $F$, $Y^{\mf h}=\{Y^{\mf h}_n : n\ge 0\}$ defines a
$F$-valued discrete time Markov chain with transition probabilities
$$
{\mf p}(\eta,\xi) \;=\; \prob_{\eta} \big[\, T^+_{F} =T^+_{\xi}\, \big] 
\,,\quad \eta \,,\, \xi\in F \;.
$$
Note that ${\mf p}(\eta,\eta)$ may be strictly positive and that
$Y^{\mf h}$ inherits the irreducibility and the recurrence properties
from $Y$.

Let $T^{\mf h}= \{T^{\mf h}_{n} : n\ge 0\}$ be the sequence 
$$
T^{\mf h}_n \;=\;
h(Y^{\mf h}_{n})\,\frac{e_{t_n}}{\lambda(Y^{\mf h}_{n})} \;, \quad
n\ge 0\; .
$$
By definition, the $h$-trace of $\{\eta_t : t\ge 0\}$ is given by
$\eta^h_t=Y^{\mf h}_{\alpha(t)}$, $t\ge 0$, where $\alpha(\cdot)$
represents the time-change \eqref{alpha} with $Y^{\mf h}$ and $T^{\mf
  h}$ in place of $Y$ and $T$, respectively. Note that $\{e_{t_n} :
n\ge 0\}$ is a sequence of i.i.d.\! mean one exponential random
variables independent of the process $\{Y^{\mf h}_n : n\ge 0\}$. By
this observation and by the proof of Theorem 2.8.1 in \cite{n},
$\{\eta^h_t : t\ge 0\}$ is a strong Markov process on
$F$. 

The irreducibility and the recurrence of $\{\eta^h_t : t\ge 0\}$ are
inherited from the process $Y^{\mf h}$. On the other hand, the
transition rates $\{R^h(\eta,\xi) : \eta,\xi\in F\}$ of the strong
Markov process $\{\eta^h_t : t\ge 0\}$ are given by
$$
R^h(\eta,\xi) \;:=\; \lim_{t\downarrow 0} 
\frac{ {\bf P}_{\eta}[\, \eta^h_t=\xi \,]}{t} 
\;=\; \frac{{\mf p}(\eta,\xi)}{{\bf E}_{\eta}[\,T_0^{\mf h}\,] } 
\;=\; \frac{\lambda(\eta)}{h(\eta)} \,
\prob_{\eta}\big[\,T^+_F=T^+_{\xi}\,\big] 
$$
for $\eta, \xi\in F$, $\eta\not=\xi$. The second identity follows from
the proof of Theorem 2.8.4 in \cite{n}.
\end{proof}

It follows from this proposition that the holding rates $\{
\lambda^h(\eta) : \eta\in F\}$ and the jump probabilities $\{
p^h(\eta,\xi): \eta, \xi \in F \}$ of the $h$-trace process
$\{\eta^h_t : t\ge 0\}$ are given by
\begin{equation}
\label{lh}
\lambda^h(\eta)\,=\,\frac{\lambda(\eta)}{h(\eta)}\,
\prob_{\eta}\big[\,{T^+_F} = T^+_{F\setminus\{\eta\}}\,\big]\,,
\end{equation}
and, for $\eta\not=\xi$,
\begin{equation*}
p^h(\eta,\xi)\;=\;\frac{\prob_{\eta}[{T^+_F}=T^+_{\xi}]}
{\prob_{\eta}[{T^+_F} = T^+_{F\setminus\{\eta\}}]}
\;=\;\prob_{\eta}\big[ \, T_{F\setminus\{\eta\}}=T_{\xi} \,\big]\;. 
\end{equation*}
Note that $p^h( \cdot, \cdot)$ depends on $h$ only through its
support.  The second identity is obtained by intersecting the event
$\{T_{F\setminus\{\eta\}}=T_{\xi}\}$ with the partition $\{T^+_F =
T_{F\setminus\{\eta\}}\}$, $\{T^+_F = T^+_{\eta}\}$ and applying the
strong Markov property to the second piece as in the proof of
\eqref{h}. \medskip

When $h$ is the indicator function of a set $F$, we obtain an explicit
formula for the transition rates of the trace process.

\begin{corollary}
\label{rf}
Let $R^F$ stand for the transition rates of $\{\eta^h_t : t\ge 0\}$
when $h = \mb 1\{F\}$. Then, for $\eta$, $\xi$ in $F$, $\eta\not =
\xi$, 
$$
R^F(\eta,\xi) \,=\, R(\eta,\xi) \,+\, \sum_{\zeta\in F^c} 
R(\eta,\zeta)\, \prob_{\zeta}\big[\,T_F = T_{\xi}\,\big]\;.
$$
\end{corollary}

\begin{proof}
By Proposition \ref{protra}, $R^F(\eta,\xi) = \lambda(\eta)
\prob_{\eta} [\, T^+_F = T^+_{\xi} \,]$. Consider the stopping time
$T^+_{F^c}$ with the convention that $T^+_{F^c}=\infty$ if
$F^c=\varnothing$. Decomposing the event $\{T^+_F = T^+_{\xi}\}$
according to the event $\{T^+_F < T^+_{F^c}\}$ and its complement, we
get
\begin{equation*}
R^F(\eta,\xi) 
\;=\; \lambda(\eta)\, \prob_{\eta} \big[\, T^+_F = T^+_{\xi}
\,;\,T^+_F < T^+_{F^c} \,\big]
\;+\; \lambda(\eta)\, \prob_{\eta}\big[\, T^+_F = T^+_{\xi}
\,;\,T^+_{F^c} < T^+_{F} \,\big]\;.
\end{equation*}
The first probability on the right hand side is equal to
$\prob_{\eta}[\,T^+_E=T^+_{\xi}\,] = p(\eta,\xi)$, while the second term,
by the strong Markov property, is equal to
\begin{equation*}
\sum_{\zeta\in F^c} R(\eta,\zeta)\, \prob_{\zeta}
\big[\,T_F = T_{\xi}\,\big]\,.
\end{equation*}
This concludes the proof of the corollary.
\end{proof}

The previous corollary provides an explicit formula for the rates
$R^{F}$ in terms of the holding times $\lambda$ and the transition
probabilities $p$ in the case where $F= E\setminus \{\xi_0\}$:
\begin{equation*}
R^F(\eta,\xi) \;=\; R(\eta,\xi) \;+\; R(\eta,\xi_0) \, p(\xi_0,\xi)
\end{equation*}
for $\eta\not = \xi$, $\{\eta,\xi\}\subseteq E\setminus \{\xi_0\}$.
In particular, if $E$ is a finite set, the rates $R^F$ can be obtained
recursively. \medskip

Since $\{\eta^h_t : t\ge 0\}$ is recurrent and irreducible, it has an
invariant measure which is unique up to scalar multiplies. Let $\mu$
be an invariant measure for $\{\eta_t : t\ge 0\}$ and denote by
$\mu^h_o$ the measure on $F$ given by
$$
\mu^h_o(\xi) \;:=\;h(\xi) \, \mu(\xi)\;,\quad \xi\in F\;.
$$

\begin{proposition}
\label{s06}
$\mu^h_o$ is an invariant measure for $\{\eta^h_t : t\ge 0\}$. In
particular, if $h$ is ${\mu}-$integrable then $\{\eta^h_t : t\ge
0\}$ is positive recurrent. Moreover, if $\mu$ is a reversible
measure for $\{\eta_t : t\ge 0\}$ then $\mu^h_o$ is a reversible measure
for $\{\eta^h_t : t\ge 0\}$.
\end{proposition}

\begin{proof}
Without loss of generality, we may suppose that $\mu$ is of the form
\eqref{im} for some $\eta\in F$. Thus, by \eqref{time}, for any
$\xi\in E$,  
\begin{equation*}
h(\xi)\mu(\xi) \;=\; {\bf E}_{\eta}\Big[ \int_0^{T^+_{\eta}} 
h(\eta_s){\bf 1}\{\eta_s = \xi \}\,ds\Big] \\
\;=\; {\bf E}_{\eta}\Big[ \int_0^{\tau^+_{\eta}(\eta^h_{\cdot})} 
{\bf 1}\{\eta^h_s = \xi \}\,ds\Big]\;.
\end{equation*}
This shows that $\mu^h_o$ is an invariant measure for the $h$-trace
process. The second assertion follows from Theorem 3.5.3 in \cite{n}.

Suppose now that $\mu$ is reversible for $R(\cdot,\cdot)$. Then, the
measure $M$ defined in \eqref{m} is a reversible measure for the jump
chain $Y=\{Y_n : n\ge 0\}$. Since the events $\{T^+_F = T^+_{\xi}\}$
and $\{T^+_F = T^+_{\eta}\}$ depend only on $Y$,
$$
M(\eta) \, \prob_{\eta}\big[\, T^+_F = T^+_{\xi} \,\big] 
\;=\; M(\xi) \, \prob_{\xi}\big[\, T^+_F = T^+_{\eta} \,\big]\;,
$$
for any $\eta,\xi\in F$, $\eta\not = \xi$. In consequence, by the
formula for $R^h(\cdot,\cdot)$ obtained in Proposition \ref{protra},
$\mu^h_o$ is reversible for the $h$-trace process.
\end{proof}

\subsection{Positive recurrent case.}

We assume from now on that the Markov chain $\{\eta_t : t\ge 0\}$ is
positive recurrent. Denote by $\mu$ its unique invariant probability
measure.

\subsubsection*{Replacement Lemma}
\label{ss54}

For any probability measure $\nu$ on $E$, we denote by $\langle \cdot
\rangle_{\nu}$ the expected value with respect to $\nu$. 

\begin{lemma}
\label{s08}
Fix a function $g: E \to \bb R$ with nonempty support, integrable with
respect to $\mu$ and such that $\langle g\rangle_{\mu}=0$. Fix also
some $\xi$ in $A=\{\eta : g(\eta)\not = 0\}$. For every $t>0$,
\begin{equation*}
\sup_{\eta\in E} \Big| \, \E_{\eta} \Big[ \int_0^t  g(\eta_s) \, ds \Big]
\, \Big| \;\le\; 2\,\sup_{\eta\in A}  \, \E_{\eta} \Big[ 
\int_0^{T_{\xi}}|g(\eta_s)|\, ds \Big] \;.
\end{equation*}
\end{lemma}

\begin{proof}
Let $\{\Theta_t : t\ge 0\}$ stand for the time shift operators on
$D(\mathbb R_+,E)$. Define the random times $H_0=0$, $H_1= T^+_{\xi}$
and $H_{j+1}=H_j + {\tau}^+_{\xi}\circ \Theta_{H_{j}}(\eta_{\cdot})$,
$j\ge 1$. Fix an arbitrary $\eta\in E$ and let $h:E\to\bb R_+$ be a
nonnegative function, integrable with respect to $\mu$. By Proposition
\ref{s06}, the trace process $\{\eta^h_t : t\ge 0\}$ is positive
recurrent so that
\begin{equation}
\label{conh}
\E_{\eta} \Big[ \int_0^{T_{\xi}} h(\eta_s)\,ds \Big] 
\;=\; \E_{\eta} \big[ \tau_{\xi}(\eta^h_{\cdot}) \big] \;<\; \infty\,.
\end{equation}
Write
\begin{equation}
\label{f17}
\begin{split}
\E_{\eta} \Big[ \int_0^t  h(\eta_s) \, ds \Big] 
\; & =\; \sum_{j\ge 0} \E_{\eta} \Big[ \int_0^t  h(\eta_s) \, ds 
\;  \mb 1\{H_j \le t <H_{j+1} \} \Big]  \\
\; & =\;  \sum_{j\ge 0} \E_{\eta} \Big[
\int_{0}^{H_{j+1}} h(\eta_s) \, ds \; \mb 1\{H_j \le t <H_{j+1}
\} \Big]  \\
\; & -\; \sum_{j\ge 0} \E_{\eta} \Big[
\int_{t}^{H_{j+1}} h(\eta_s) \, ds \; \mb 1\{H_j \le t <H_{j+1}
\} \Big]\,. 
\end{split}
\end{equation}
In the last equation, we used the fact that both terms on the right
hand side are finite. To prove it, notice first that the second term
is bounded by the first one. By Tonelli's theorem, the first term is
equal to
\begin{equation*}
\begin{split}
& \sum_{j\ge 0} \sum_{k=0}^j\E_{\eta} \Big[
\int_{H_k}^{H_{k+1}} h(\eta_s) \, ds \; \mb 1\{H_j \le t <H_{j+1}
\} \Big]\,\\
& \quad =\; \sum_{k\ge 0} \E_{\eta} \Big[
\int_{H_k}^{H_{k+1}} h(\eta_s) \, ds \; \mb 1\{H_k \le t\} \Big]\;.
\end{split}
\end{equation*}
Taking conditional expectation with respect to $\mc F_{H_k}$, by the
strong Markov property, this sum is equal to
\begin{equation*}
\E_{\eta} \Big[ \int_{0}^{H_1} h(\eta_s) \, ds \Big]
\;+\; \E_{\xi} \Big[ \int_{0}^{H_1} h(\eta_s) \, ds \Big] \, 
\sum_{k\ge 1} \prob_{\eta} \big[ H_k \le t \big]
\end{equation*}
The first term of this sum is finite by \eqref{conh}. In the second
expectation, $\xi$ appears instead of $\eta$, and the expectation is
equal to $\langle h \rangle_{\mu}$ by \eqref{im}. Finally, the sum
is finite by the strong Markov property and because $\mb P_\xi [ T^+_\xi
\le t]$ is strictly smaller than $1$.

To estimate the last term in \eqref{f17}, note that the event $\{H_j
\le t < H_{j+1}\}$ belongs to $\mc F_t$ and that on this set $H_{j+1}
= t + \tau^+_{\xi} \circ \Theta_t(\eta_{\cdot})$. Therefore, by the
Markov property,
\begin{equation*}
\sum_{j\ge 0} \E_{\eta}\Big[ \int_t^{H_{j+1}} h(\eta_s)\,ds  \; 
\mb 1\{H_j \le t <H_{j+1}\} \Big] \;=\; \E_{\eta} \Big[ \E_{\eta_t} 
\Big[ \int_{0}^{H_1} h(\eta_s) \, ds \, \Big] \, \Big]\,.
\end{equation*}

Putting together the previous identities, we get that the left
hand side of \eqref{f17} is equal to
\begin{equation*}
\E_{\eta} \Big[ \int_{0}^{H_1} h(\eta_s) \, ds \Big]
\; +\;\langle h\rangle_{\mu} \, 
\sum_{k\ge 1} \prob_{\eta} \big[ H_k \le t \big] \;-\; 
\E_{\eta} \Big[ \E_{\eta_t} \Big[ \int_{0}^{H_1} h(\eta_s) \, ds \, 
\Big]\, \Big]\;.
\end{equation*}
Applying the previous identity to $g^+$ and $g^-$, since $\langle
g\rangle_{\mu}=0$, we obtain that
\begin{equation*}
\E_{\eta} \Big[ \int_0^t  g(\eta_s) \, ds \Big] 
\;=\; \E_{\eta} \Big[ \int_{0}^{H_1} g(\eta_s) \, ds \Big]
\;-\; \E_{\eta} \Big[ \E_{\eta_t} 
\Big[ \int_{0}^{H_1} g(\eta_s) \, ds \, \Big] \, \Big]\;.
\end{equation*}

We claim that we may replace the stopping time $H_1$ by $T_\xi$ in
both terms of the right hand side. Indeed, if $\eta$ is different from
$\xi$, $H_1 = T_{\xi}$ $\prob_{\eta}$-a.s.  Conversely, if the
starting point $\eta$ is equal to $\xi$, $T_{\xi}=0$ so that, by
\eqref{im},
$$
\E_{\eta} \Big[ \int_0^{H_1} g(\eta_s) \, ds \Big] 
\;=\; 0 \;=\; 
\E_{\eta} \Big[ \int_0^{T_{\xi}} g(\eta_s) \, ds \Big]\,.
$$
Thus, taking the supremum over $E$, we have proved that
$$
\sup_{\eta \in E} \Big|\E_{\eta} \Big[ \int_0^t  g(\eta_s) \, ds \Big] \Big|
\;\le\; 2\,\sup_{\eta \in E} \Big|\E_{\eta} \Big[
\int_{0}^{T_{\xi}} g (\eta_s) \, ds \Big]\Big| \;.
$$

Finally, since $g$ vanishes outside $A$ and since $\xi$ belongs to
$A$, by the strong Markov property,
\begin{equation*}
\E_{\eta} \Big[ \int_0^{T_{\xi}} g(\eta_s) \, ds \Big] \;=\;
\E_{\eta} \Big[ \int_{T_{A}}^{T_{\xi}} g(\eta_s) \, ds \Big]
\;=\; \mb E_{\eta} \Big[ \, \E_{\eta_{T_A}} \Big[ 
\int_{0}^{T_{\xi}} g(\eta_s) \, ds \Big] \, \Big] \;. 
\end{equation*}
Therefore,
\begin{equation*}
\sup_{\eta\in E} \Big| \, \E_{\eta} \Big[ 
\int_{0}^{T_{\xi}} g(\eta_s) \, ds \Big] \, \Big| 
\;\le\; \sup_{\eta\in A} \Big| \, \E_{\eta} \Big[ 
\int_{0}^{T_{\xi}} g(\eta_s) \, ds \Big] \, \Big|\;.
\end{equation*}
This concludes the proof of the lemma.
\end{proof}

Let $S$ be a finite set and let $\pi=\{A^x : x\in S\}$ be a partition
of $E$. Denote by $\mu^x$, $x\in S$, the stationary measure $\mu$
conditioned on $A^x$: $\mu^x(\bs \cdot) =\mu( \bs\cdot|A^x)$. Also,
for each $\mu$-integrable function $g$ denote by $\langle
g|\pi\rangle_{\mu}:E\to\bb R$ the conditional expectation of $g$,
under $\mu$, given the $\sigma$-algebra generated by $\pi$:
$$
\langle g|\pi\rangle_{\mu}\,=\,\sum_{x\in S}\langle g\rangle_{\mu^x}\,
\mb 1 \{A^x\}\,. 
$$

The next result shows that if the process thermalizes quickly in each
set of the partition, we may replace time averages of a bounded
function by time averages of the conditional expectation. This
statement plays a key role in our investigation of metastability. It
assumes the existence of an attractor, but similar versions should
exist under weaker assumptions on thermalization.

For each $x\in S$ and $\mu$-integrable function $g:E\to \bb R$, let  
$$
g^x \,:=\, (g-\langle g \rangle_{\mu^x}) \mb 1\{A^x\}
$$
and fix some state $\xi_x$ in $A^x$, for each $x$ in $S$. Next
statement follows from Lemma \ref{s08} applied to each $g^x$, $x\in
S$. Note that the right hand side does not depend on time.

\begin{corollary}
\label{s09}
Let $g:E\to\bb R$ be an integrable function. Then,
\begin{equation*}
\sup_{\eta\in E} \Big\vert \E_{\eta} \Big[ \int_0^t \big\{
g-\langle g | \pi\rangle_{\mu}\big\} (\eta_s) \, ds \Big] \Big\vert 
\;\le\; 2\sum_{x\in S}\,\sup_{\eta\in A^x} 
\E_{\eta} \Big[\, \int_0^{T_{\xi_x}} |g^x(\eta_s)|\,ds \,\Big] \,.
\end{equation*}
\end{corollary}

Clearly, the right hand side in the previous corollary is bounded
above by
\begin{equation*}
4\, \sum_{x\in S} \Vert g\Vert_x \sup_{\eta\in A^x}
\,\E_{\eta}\Big[\, \int_0^{T_{\xi_x}} 
{\bf 1}\{ \eta_s\in A^x\}\,ds \,\Big]\,,
\end{equation*}
where $\Vert g\Vert_x \,=\, \sup \{ |g(\eta)| : \eta \in A^x \}$.

\subsubsection*{Mean set rates}

Let $h:E \to \bb R_+$ be a nonnegative function satisfying \eqref{fne}
and belonging to $L^1(\mu)$. By Propositions \ref{protra} and
\ref{s06}, $\{\eta^h_t : t\ge 0\}$ is irreducible and positive
recurrent. Moreover, its invariant probability measure, denoted by
$\mu^h$, is given by
\begin{equation}
\label{g01}
\mu^h(\xi) \;=\; \frac {h(\xi)}{\langle h\rangle_{\mu}}\,  
\mu(\xi)\;,\quad \xi\in F\;.
\end{equation}
For each pair $A,B$ of disjoint subsets of $F$, denote by $r_h(A,B)$
the average rate at which the $h$-trace process jumps from $A$ to $B$:
\begin{eqnarray*}
r_h(A,B) &:=& \frac 1{\mu^h(A)} \sum_{\eta\in A} \mu^h(\eta) 
\sum_{\xi \in B} \, R^h(\eta,\xi) \\
&=& \frac 1 {\langle h,{\bf 1}\{A\}\rangle_{\mu}} 
\sum_{\eta\in A} M(\eta) \,{\bf P}_{\eta}\big[ \, T^+_F = T^+_B \, \big]\,,
\end{eqnarray*}
where $M$ has been introduced in \eqref{m}. We used relation
\eqref{g01} and Proposition \ref{protra} in the last equality. We
shall refer to $r_h(\cdot,\cdot)$ as the mean set rates associated to
the trace process.

When $h$ is the indicator function of a
set $F$, we denote $r_h$ by $r_F$. In this case,
\begin{equation}
\label{rh}
\mu(A) \, r_F(A,B)\;=\; \sum_{\eta\in A} M(\eta) \, 
\prob_{\eta}\big[ \, T^+_B < T^+_{F\setminus B}\, \big]\;.
\end{equation}

\subsection{The reversible case}
\label{traca}

 From now on, we shall assume in addition that the process is
reversible with respect to the invariant probability measure $\mu$ and
that the measure $M$ is finite:
\begin{equation}
\label{ar}
\sum_{\eta\in E } M(\eta) \;=\;\sum_{\eta\in E } 
\lambda(\eta ) \mu(\eta) \;<\; \infty\;.
\end{equation}
In particular, the mean set rates $r_h(A,B)$ are finite.

Assumption (\ref{ar}) reduces the potential theory of continuous time
Markov chains to the potential theory of discrete time Markov
chains.  Recall from Subsection \ref{cap} that $\<\cdot , \cdot \>_M$
represents the scalar product in $L^2(M)$, that $P:L^2(M) \to L^2(M)$
stands for the bounded operator defined by $(Pf)(\eta) = \sum_{\xi\in
  E} p(\eta, \xi) f(\xi)$, and that $D(f) = \< (I-P)f , f\>_M$, $f\in
L^2(M)$, is the Dirichlet form associated to the Markov process
$\{\eta_t : t\ge 0\}$. A simple computation shows that for every $f$
in $L^2(M)$,
\begin{equation*}
D(f) \;=\; \frac 12 \sum_{\eta,\xi \in E} M(\eta) p(\eta,\xi) 
\{f(\xi) - f(\eta) \}^2 \;.
\end{equation*}

Fix two disjoint subsets $A$, $B$ of $E$ and recall that $\mathcal
C(A,B) \;:=\; \{f\in L^2(M) : \textrm{$f(\eta)=1$ $\forall$
  $\eta\in A$ and $f(\xi)=0$ $\forall$ $\xi\in B$}\}$, and that the
capacity of $A$, $B$ is defined by
\begin{equation*}
\Cap (A,B) \;:=\; \inf\big\{\,D (f) : f\in \mathcal C (A,B)\,\big\}\;.
\end{equation*}
As $\max \{D(f\land 1),D(f\lor 0)\}\le D(f)$, $\forall f\in L^2(M)$,
we may restrict the infimum to functions bounded below by $0$ and
bounded above by $1$.

Denote by $f_{AB}:E \to\bb R$ the function in $\mc C (A,B)$
defined as
\begin{equation*}
f_{AB}(\eta) \;:=\; \prob_{\eta}\big[\,T_{A} < T_B \,\big]\;.  
\end{equation*}
An elementary computation shows that $f_{AB}$ solves the equation
\begin{equation}
\label{eqf2}
\left\{
\begin{array}{ll}
(L f)(\eta) =0 & \eta\in E\setminus (A\cup B)\;, \\
f(\eta) = 1 & \eta\in A\;, \\
f(\eta) = 0 & \eta\in B \;.
\end{array}
\right.
\end{equation}
Clearly, we may replace the generator $L$ by the operator $I-P$ in the
above equation. It is not difficult to show that \eqref{eqf2} has a
unique solution in $L^2(M)$ given by $f_{AB}$. Indeed, if $f$, $g$
are solutions, $D(f-g) = \<(I-P) (f-g), (f-g)\>_M=0$. In particular, by
the explicit expression of the Dirichlet form, $f-g$ is
constant. Since the difference vanishes on $A\cup B$, $f=g$.

\begin{lemma}
\label{s01}
For any two disjoint subsets $A$, $B$ of $E$,
\begin{equation*}
\Cap (A,B) \;=\; D(f_{AB}) \,=\,\sum_{\eta\in A} M(\eta) 
\, \prob_{\eta}\big[ \, T^+_B <T^+_{A} \, \big]\,.
\end{equation*}
\end{lemma}

\begin{proof}
We first claim that there exists a function $f$ in $\mc C(A,B)$ whose
Dirichlet form is equal to the capacity $\Cap (A,B)$. Indeed, we have
already seen that we may restrict the variational problem defining the
capacity to functions bounded below by $0$ and bounded above by
$1$. Consider a sequence $\{f_n : n\ge 1\}$ in $\mc C(A,B)$ such that
$0\le f_n\le 1$, $\lim_{n\to \infty} D(f_n) = \Cap (A,B)$. Since the
sequence $f_n$ is uniformly bounded, there exist $f$ in $\mc C(A,B)$,
$0\le f\le 1$, and a subsequence, still denoted by $\{f_n : n\ge 1\}$,
such that $f(\eta) = \lim_{n\to\infty} f_n(\eta)$ for every $\eta$ in
$E$. By Fatou's lemma, $D(f) \le \liminf_{n\to \infty} D(f_n) = \Cap
(A,B)$. Since $f$ belongs to $\mc C(A,B)$, $D(f) = \Cap (A,B)$, which
proves the claim.

We further claim that $f$ solves \eqref{eqf2}. Fix $\eta\not\in A\cup
B$. Since $f$ solves the variational problem for the capacity, it is
clear that $f(\eta)$ is the argument which minimizes the convex
function $F: \bb R \to \bb R$ defined by
\begin{equation*}
F(a) \;=\; \sum_{\xi \sim \eta} M(\eta) p(\eta,\xi) \{ f(\xi) - a
\}^2\; .
\end{equation*}
In this formula $\xi\sim\eta$ means that the underlying jump chain
may jump from $\eta$ to $\xi$, i.e., that $p(\eta,\xi)>0$.  An
elementary computation shows that the minimum is attained at $a =
\sum_\xi p(\eta,\xi) f(\xi)$ so that $f(\eta) = (Pf)(\eta)$. Since
$f_{AB}$ is the unique solution in $L^2(M)$ of \eqref{eqf2},
$f=f_{AB}$ and $\Cap (A,B) = D(f_{AB})$. This proves the first
statement of the lemma. The second one follows from a straightforward
computation.
\end{proof}

In particular, by \eqref{rh} we have the following very useful
identity between capacities and mean set rates.
\begin{lemma} \label{g03}
Assume that $F= A\cup
B$ and $A\cap B = \varnothing$. Then,
\begin{equation*}
\mu(A)\, r_F(A,B) \,=\,\Cap(A,B)\,. 
\end{equation*}
\end{lemma}

Next result shows that the mean set rates can be expressed in terms of
capacities.

\begin{lemma}
\label{t02}
Let $A$, $B$ be subsets of $F$ such that $A\cap B =
\varnothing$. Then, 
\begin{equation*}
\mu(A)\, r_F(A,B) \,=\,
\frac{1}{2} \Big\{ \Cap (A, F\setminus A) + \Cap(B, F\setminus B) 
- \Cap (A\cup B, F\setminus [A\cup B])\, \Big\} \,. 
\end{equation*}
\end{lemma}

\begin{proof}
The proof is elementary and follows from Lemma \ref{g03} and the identity
\begin{equation*}
\begin{split}
2 \, \mu(A)\, r_F(A,B) \, & =\, \mu(A)\, r_F(A, F\setminus A) 
\; +\; \mu(B)\, r_F (B, F\setminus B) \\
\; & -\; \mu (A\cup B) \, r_F (A\cup B, F\setminus [A\cup B]) \,.    
\end{split}
\end{equation*}
\end{proof}

By assumption \eqref{ar}, the holding rates $\lambda : E \to \bb R_+$
belong to $L^1(\mu)$. This property extends to the holding rates
$\{\lambda^h (\eta) : \eta \in E\}$ of the $h$-trace process if $h$
belongs to $L^1(\mu)$. Indeed, by (\ref{g01}) and (\ref{lh}),
$$
\sum_{\eta\in E} \lambda^h (\eta) \mu^h(\eta) 
\;=\; \frac{1} {\langle h\rangle_{\mu}} \,\sum_{\eta\in E} 
M(\eta)\,\prob_{\eta}\big[\,{T^+_F} = T^+_{F\setminus\{\eta\}}\,\big] 
\,<\, \infty\,.
$$ 
Therefore, assumption (\ref{ar}) holds for the $h$-trace process
whenever $h$ belongs to $L^1(\mu)$. In this case, its capacity,
denoted by $\Cap_h(\bs \cdot,\bs \cdot)$, is well defined. Next result
shows a simple relation between $\Cap_h(\bs \cdot,\bs \cdot)$ and the
capacity of the original process.

\begin{lemma}
\label{capah} 
Let $h:E\to \bb R_+$ be a nonnegative $\mu$-integrable function with
nonemp\-ty support denoted by $F$. Then, for every subsets $A$, $B$ of
$F$, $A\cap B = \varnothing$,
\begin{equation*}
\langle h\rangle_{\mu}\,\Cap_h(A,B)\,=\, \Cap(A,B)\,.
\end{equation*}
\end{lemma}

\begin{proof}
Fix a function $h:E\to \bb R_+$ with the properties required in the
statement of the lemma and two subsets $A$, $B$ of $F$ such that
$A\cap B = \varnothing$. By Lemma \ref{s01} applied to the process
$\{\eta^h_t : t\ge 0\}$ and by identities \eqref{h}, \eqref{g01} and
\eqref{lh},
\begin{eqnarray*}
\Cap_h(A,B) &=& \sum_{\eta\in A} \mu^h(\eta) \, \lambda^h(\eta) 
\, \prob_{\eta}\big[ \,\tau^+_{B}(\eta^h_{\cdot}) 
< \tau^+_{A}(\eta^h_{\cdot}) \, \big] \\
&=& \frac{1}{\<h\>_\mu} \,\sum_{\eta\in A} M(\eta) \, 
\prob_{\eta} \big[\, T^+_{B} < T^+_{A} \,\big|\,
T^+_F = T_{F\setminus\{\eta\}} \,\big] \, 
\prob_{\eta}\big[\,{T^+_F} = T^+_{F\setminus\{\eta\}}\,\big] \;.
\end{eqnarray*}
Since for $\eta\in A$, the event $\{T^+_B < T^+_{A}\}$ is contained in
the event $\{T^+_F = T_{F\setminus\{\eta\}}\}$ $\prob_{\eta}$-almost
surely, the previous expression is equal to
\begin{equation*}
\frac{1}{\langle h\rangle_{\mu}} \sum_{\eta\in A} M(\eta) \, 
\prob_{\eta}\big[ \, T^+_B < T^+_{A} \,\big]\,. 
\end{equation*}
By Lemma \ref{s01} this expression is equal to $\langle
h\rangle_{\mu}^{-1} \Cap(A,B)$, which proves the lemma.
\end{proof}

We conclude this subsection proving a relation between expectations of
time integrals of functions and capacities. Fix two disjoint subsets
$A$, $B$ of $E$. Define the probability measure $\nu_{AB}$ on $A$ as
$$
\nu_{AB}(\eta)\,=\, \frac{M(\eta) \,
\prob_{\eta} \big[ \, T^+_B <T^+_{A} \, \big] }{\Cap(A,B)}\,,
\quad \eta\in A\,.
$$
Denote by ${\bf E}_{\nu_{AB}}$ the expectation associated to the
Markov process $\{\eta_t : t\ge 0\}$ with initial distribution
$\nu_{AB}$. The proof of the following proposition is an adaptation of
the proof of identity (4.28) in \cite{g}.

\begin{proposition}
\label{bovier}
Fix two disjoint subsets $A$, $B$ of $E$.  Let $g:E \to\bb R$ be a
$\mu$-integrable function. Then,
$$
\E_{\nu_{AB}}\Big[ \int_0^{T_B} g(\eta_t)\,dt \Big] 
\;=\; \frac{\langle\, g \,,\, f_{AB}\rangle_{\mu}\, }
{\Cap(A,B)}\;\cdot
$$
\end{proposition}

\begin{proof}
We first claim that the proposition holds for indicator functions of
states.  Fix an arbitrary state $\xi\in E$. Consider the random time $t_B
:= \inf \{ n\ge 0 : Y_n\in B \}$ and the last exit time
$$
L_{AB} \;:=\; \sup \{ n\ge 0 : Y_n\in A \textrm{ \;and \;} n< t_B \}\;
$$
with the convention that $\sup \varnothing = -\infty$. Then, 
\begin{eqnarray*}
{\bf P}_{\xi}\big[ \, T_A < T_B \,\big] 
&=& \sum_{n\ge 0} {\bf P}_{\xi} \big[\, L_{AB} = n \,\big] \\
&=& \sum_{n\ge 0} \sum_{\eta\in A} 
{\bf P}_{\xi} \big[\, Y_n=\eta \,;\, n<t_B\,\big] \, 
{\bf P}_{\eta} \big[\, T^+_B < T^+_A  \,\big] \\
&=&  \sum_{\eta\in A}  {\bf P}_{\eta} \big[\, T^+_B < T^+_A  \,\big]\, 
\sum_{n\ge 0}  {\bf P}_{\xi} \big[\, Y_n=\eta \,;\, n<t_B\,\big]\;.
\end{eqnarray*}
Since $Y$ is reversible with respect to $M$, the last expression
is equal to
\begin{equation*}
\sum_{\eta\in A}  {\bf P}_{\eta} \big[\, T^+_B < T^+_A  \,\big]\, 
\frac{M(\eta)}{M(\xi)} \, \sum_{n\ge 0}  
{\bf P}_{\eta} \big[\, Y_n=\xi \,;\, n<t_B\,\big]\;.
\end{equation*}
Recall from the beginning of this section that $\{e_n : n\ge 0\}$ is a
sequence of i.i.d.\! mean one exponential random variables independent
of the jump chain $\{Y_n : n\ge 0\}$.  By definition of the measure
$\nu_{AB}$, this sum can be rewritten as
\begin{eqnarray*}
&& \Cap(A,B) \, \sum_{\eta\in A} \nu_{AB}(\eta)\, 
\frac{\lambda(\xi)}{M(\xi)} \, {\bf E}_{\eta}\Big[\, 
\sum_{n=0}^{t_B -1} \frac{e_n}{\lambda(\xi)} \,{\bf 1}\{ Y_n=\xi \} \, \Big] \\
&& \quad =\; \frac{\Cap(A,B)}{\mu(\xi)} \, {\bf E}_{\nu_{AB}} 
\Big[\, \int_{0}^{T_B} {\bf 1}\{ \eta_s = \xi \} \,ds \, \Big]\;.
\end{eqnarray*}
This proves the assertion for $g={\bf 1}\{\xi\}$. By linearity and the
monotone convergence theorem we get the desired result for positive
and then $\mu$-integrable functions.
\end{proof}

In particular, taking $A=\{\eta\}$ and $B=\{\xi\}$ for $\eta\not =
\xi$ we have that
\begin{equation}
\label{g02}
\E_{\eta} \Big[\, \int_0^{T_\xi} g(\eta_s)\,ds \, \Big] \;=\; 
\frac{ \langle \,g\,,\, f_{\{\eta\} \{\xi\}} \, \rangle_{\mu}}
{\Cap(\{\eta\},\{\xi\})}
\end{equation}
for any $\mu$-integrable function $g$.
 
This formula provides a more accurate estimate in Corollary \ref{s09}
in the reversible context. For each $x\in S$, let
$$
\Cap(\xi_x)\,:=\,\inf_{\eta\in A^x\setminus\{\xi_x\}}
\Cap(\{\eta\}, \{\xi_x\})\,.
$$

\begin{lemma}
\label{s10}
Let $g:E\to\bb R$ be a function integrable with respect to $\mu$.  If
the measure $\mu$ is reversible then, for each $x\in S$,
\begin{equation*}
\sup_{\eta\in A^x} \E_{\eta} \Big[\, \int_0^{T_{\xi_x}}
|g^x(\eta_s)|\,ds \,\Big] 
\;\le\; \frac{2 \, \langle \,|g|\,\rangle_{\mu^x}}{\Cap(\xi_x)}
\, \mu(A^x) \;, 
\end{equation*}
where $|g|(\eta)=|g(\eta)|$ for all $\eta$ in $E$.
\end{lemma}

\begin{proof}
By \eqref{g02} and the fact that $0\le f_{\{\eta\}\{\xi_x\}}\le 1$,
the left hand side is bounded above by
$$
\sup_{\eta\in A^x\setminus\{\xi^x\}}  
\frac{\langle \,\vert g^x\vert \,,\,f_{\{\eta\} \{\xi_x\} }\,\rangle_{\mu}}
{\Cap ( \{\eta\} , \{\xi_x\} )} \;\le\;
\frac{\langle \,\vert g^x\vert \,\rangle_{\mu}}{\Cap ( \xi_x )}
\;\le\; \frac{2 \,\langle \,\vert g\vert \, 
{\bf 1}\{A^x\}\,\rangle_{\mu}}{\Cap ( \xi_x )} 
$$
for each $x\in S$. This completes the proof.
\end{proof}

\medskip
\noindent{\bf Acknowledgments}: The authors would like to thank
E. Olivieri for fruitful discussions on metastability and the two
anonymous referees for their suggestions.

\end{document}